\documentclass[a4paper, 11pt, reqno]{amsart}
\usepackage{amsthm,amssymb,amsmath,enumerate,graphicx,psfrag,enumitem,mathrsfs}

\usepackage[margin=2.5cm]{geometry}
\usepackage{hyperref}
\hypersetup{colorlinks=true,
   citecolor=blue,
   filecolor=blue,
   linkcolor=blue,
   urlcolor=blue
}

\usepackage[textsize=footnotesize]{todonotes}

\usepackage{algorithm}
\usepackage{tikz}

\usetikzlibrary{backgrounds}

\newtheorem{definition}{Definition}[section]

\newtheorem{proposition}[definition]{Proposition}
\newtheorem{theorem}[definition]{Theorem}
\newtheorem{corollary}[definition]{Corollary}
\newtheorem{lemma}[definition]{Lemma}
\newtheorem{conjecture}[definition]{Conjecture}

\numberwithin{equation}{section}

\newcommand{\comment}[1]{}
\newcommand{\N}{\mathbb N}

\newcommand{\es}{\emptyset}

\newcommand{\cS}{\mathcal{S}}
\newcommand{\cF}{\mathcal{F}}
\newcommand{\cT}{\mathcal{T}}
\newcommand{\cC}{\mathcal{C}}
\newcommand{\cD}{\mathcal{D}}
\newcommand{\cH}{\mathcal{H}}
\newcommand{\cM}{\mathcal{M}}
\newcommand{\cI}{\mathcal{I}}
\newcommand{\cE}{\mathcal{E}}

\newcommand{\I}{I}

\newcommand{\sF}{\mathscr{F}}
\newcommand{\sT}{\mathscr{T}}

\newcommand{\Pro}{\mathbb{P}}
\newcommand{\Exp}{\mathbb{E}}

\newcommand{\den}{{\rm den}}
\newcommand{\dom}{{\rm dom}}
\newcommand{\p}{{\rm par}}
\newcommand{\ori}{{\rm orig}}

\renewcommand{\epsilon}{\varepsilon}

\newcommand{\sm}{\setminus}
\newcommand{\sub}{\subseteq}

\newcommand{\COMMENT}[1]{}

\newcounter{step}
\newcommand{\step}[1]{\bigskip\refstepcounter{step}\textbf{Step \thestep}. \textbf{#1}.}

\title{Optimal packings of bounded degree trees}
\author{Felix Joos}

\address{School of Mathematics, University of Birmingham, 
Edgbaston, Birmingham, B15 2TT, United Kingdom}
\email{f.joos@bham.ac.uk, j.kim.3@bham.ac.uk, d.kuhn@bham.ac.uk, d.osthus@bham.ac.uk}

\author{Jaehoon Kim}

\author{Daniela K\"uhn}

\author{Deryk Osthus}

\thanks{The research leading to these results was partially supported by the EPSRC, grant no. EP/M009408/1 (F.~Joos, D.~K\"uhn and D.~Osthus), 
and by the Royal Society and the Wolfson Foundation (D.~K\"uhn).
The research was  also partially supported by the European Research Council under the European Union's Seventh Framework Programme (FP/2007--2013) / ERC Grant 306349 (J.~Kim and D.~Osthus). }

\date{\today}

\begin{document}

\begin{abstract}
We prove that if $T_1,\dots, T_n$ is a sequence of bounded degree trees so that $T_i$ has $i$ vertices, then $K_n$ has a decomposition into $T_1,\dots, T_n$. This shows that the tree packing conjecture of Gy\'arf\'as and Lehel from 1976 holds for all bounded degree trees 
(in fact, we can allow the first $o(n)$ trees to have arbitrary degrees). 
Similarly, we show that Ringel's conjecture from 1963 holds for all bounded degree trees.
We deduce these results from a more general theorem, which yields decompositions of dense quasi-random graphs into suitable families of bounded degree graphs. 
Our proofs involve Szemer\'{e}di's regularity lemma, results on Hamilton decompositions of robust expanders, random walks,
iterative absorption as well as a recent blow-up lemma for approximate decompositions.
\end{abstract}
\maketitle

\section{Introduction}
For a collection of graphs $\cH=\{H_1,\ldots,H_s\}$ and a graph $G$,
we say $\cH$ \emph{packs} into $G$ if there are pairwise edge-disjoint copies of $H_1,\dots, H_s$ in $G$. If every edge of $G$ lies in one of these copies, then we have a {\em decomposition} of $G$ into $\cH$. Packing and decomposition problems are central to combinatorics and related areas. 
Famous early instances of such problems involve Steiner triple systems as well as Hamilton decompositions:
Kirkman's theorem on the existence of Steiner triple systems translates into decompositions of cliques into triangles
(subject to divisibility conditions) and Walecki's construction provides an analogue of this for decompositions into Hamilton cycles.
A celebrated theorem of Wilson~\cite{Wil72a,Wil72b,Wil72c} generalizes Kirkman's theorem to decompositions of cliques into arbitrary subgraphs $F$
-- this forms one of the cornerstones of design theory. These results now are part of a major area, with connections to coding and information theory, as well as extremal combinatorics
and algorithms.

Classical results on packings and decompositions have often been limited to symmetric structures, 
as these allow for the exploitation of these symmetries or the use of algebraic techniques. 
Probabilistic approaches are having an increasing impact on the area, and enable the construction of packings in more complex or general settings.
In particular, in this paper we build on such approaches to obtain optimal results on packing suitable families of bounded degree graphs into quasi-random graphs. This provides the first instance of optimal packing and decomposition results involving general families of
large (but sparse) structures which are not necessarily symmetric.

\subsection{Packing trees into complete graphs}

The famous tree packing conjecture of Gy\'arf\'as and Lehel has driven a large amount of research in the area.

\begin{conjecture}[Gy\'arf\'as and Lehel~\cite{GL78}] \label{gyarfaslehel}
Given $n\in\N$ and trees $T_1\ldots,T_n$ with $|T_i|=i$,
the complete graph $K_n$ has a decomposition into copies of $T_1,\ldots, T_n$. 
\end{conjecture}

The conjecture has been verified for several very special classes of trees (see e.g.~\cite{Dob02, GL78, HBK87, Rod88, Zak17}). 
Bollob\'{a}s~\cite{Bol83} showed that one can pack $T_1,\dots, T_{n/\sqrt{2}}$ into $K_n$
(and that a better bound would follow from the Erd\H os--S\'os conjecture). 
Balogh and Palmer~\cite{BP13} showed that one can pack $T_{n-n^{1/4}/10},\dots, T_{n-1}$ into $K_{n}$ for sufficiently large $n$. 
\.Zak~\cite{Zak17} showed (amongst others) that one can pack the final five trees into $K_n$. 
Fishburn~\cite{Fis83} verified that the set of degree sequences of the trees $T_1,\dots,T_n$ can be ``packed'' into the degree sequence of $K_n$. B\"ottcher, Hladk\'y, Piguet and Taraz~\cite{BHPT16} were able to prove that we can obtain a near-optimal packing if one restricts to bounded degree trees (again, provided $n$ is sufficiently large and this assumption also applies to all results we mention in the following). 
More generally, they showed that if $T_1,\dots, T_s$ have bounded degree, satisfy $|T_i|\leq (1-\epsilon)n$, and $\sum_{i=1}^{s} e(T_i) \leq (1-\epsilon) \binom{n}{2}$, then $T_1,\dots, T_s$ pack into $K_n$.
Messuti, R\"odl and Schacht~\cite{MRS16} generalized this to separable families of graphs
(which also includes minor-closed families of graphs). 
This was further improved by Ferber, Lee and Mousset~\cite{FLM17}, who obtained near-optimal packings of 
separable bounded degree graphs on $n$ vertices into $K_{n}$. 
Kim, K\"uhn, Osthus and Tyomkyn~\cite{KKOT16} proved that such a result in fact holds for all bounded degree graphs.
(Moreover, the results in~\cite{KKOT16} also hold for near-optimal packings into host graphs consisting of suitable
$\varepsilon$-regular pairs. As discussed below, this will be
crucial for the current paper.)
Very recently, Allen, B\"ottcher, Hladk\'y and Piguet \cite{ABHP17} proved the following result, which allows for packing trees of maximum degree up to $o(n/\log n)$:
if $H_1,\dots, H_s$ have bounded degeneracy, maximum degree $o(n/\log n)$, and $\sum_{i=1}^{s} e(H_i) \leq (1-\epsilon) \binom{n}{2}$,  
then $H_1,\dots, H_s$ pack into $K_n$.
Note that if we place no further restrictions on the $H_i$, then we cannot ask for an actual decomposition. 
However, in the setting of the tree packing conjecture, we do obtain such a decomposition in the current paper.

\begin{theorem}\label{thm:glboundeddegree}
For all $\Delta \in \N$, 
there are $N\in \N$ and $\epsilon>0$ such that for all $n\geq N$ the following holds. 
Suppose that for each $i\in [n]$, we have a tree $T_i$ with $|T_i|= i$ and suppose $\Delta(T_i)\leq \Delta$ for all $i> \epsilon n$. 
Then $K_n$ decomposes into $T_1,\dots, T_n$.
\end{theorem}

Note that this implies Conjecture~\ref{gyarfaslehel} for all bounded degree trees. 
(In fact, we do not require the first $\epsilon n$ trees to have bounded degree.)
We will deduce Theorem~\ref{thm:glboundeddegree} from a more general result (see Theorem~\ref{thm: main result}) in Section~\ref{sec:final}.

For bounded degree trees, 
we actually obtain the following more general result which has less restrictive assumptions on $|T_i|$.
For a family $\cH=\{H_1,\ldots,H_s\}$ of graphs we write $e(\cH):=\sum_{i=1}^se(H_i)$.

\begin{theorem}\label{thm:main simple}
For all $\Delta\in \N$ and $\delta>0$, there is $N\in \N$
such that for all $n\geq N$ the following holds.
Suppose that $\cT$ is a collection of trees such that
\begin{enumerate}[label=(\roman*)]
	\item $|T|\leq n$ and $\Delta(T)\leq \Delta$ for all $T\in \cT$,
	\item there are at least $(1/2 + \delta)n$ trees $T\in \cT$ such that $\delta n\leq |T|\leq (1-\delta)n$, and
	\item $e(\cT)= \binom{n}{2}$.
\end{enumerate}
Then $K_n$ decomposes into $\cT$.
\end{theorem}

Another beautiful open problem in the area of tree decompositions is Ringel's conjecture. 

\begin{conjecture}[Ringel~\cite{Rin63}] \label{ringel}
Given $n\in\N$ and a tree $T$ on $n+1$ vertices, the complete graph $K_{2n+1}$ has a decomposition into $2n+1$ copies of $T$. 
\end{conjecture}\COMMENT{
Here, $2n$ is best possible for stars. If we have $K_{m}$ with $m\leq 2n-1$, any packing leaves at least two vertices of the clique which is never in the image of a star centre. Then edge between them cannot be covered.

If graceful labelling conjecture is true, then for any tree $T$ on $n+1$ vertices, the complete graph $K_{2n}$ has a decomposition into $2n-1$ copies of $T$.
To see this, take a leaf $\ell$ of $T$ and a vertex $x$ which is the neighbour of $\ell$, and consider $T'=T-\ell$. Since $T'$ has $n$ vertices, by using graceful labelling conjecture, $K_{2n-1}$ has a decomposition of into $2n-1$ copies of $T'$. Moreover, we also have the following : 
when $V(K_{2n-1})= [2n-1]$, there are embeddings $\phi_i: T' \rightarrow [2n-1]$ for $i\in [2n-1]$ such that $\phi_i(x)= \phi_{j}(x) + j-i ~ (\mod 2n-1)$ and $\bigcup_{i=1}^{2n-1} E(\phi_i(T')) = E(K_{2n-1}).$ Since $\phi_i(x) \neq \phi_j(x)$ for $i\neq j$, for each $j'\in [2n-1]$ there exists unique $i\in [2n-1]$ such that $\phi_i(x) = j'$. Thus we let $\phi_i(\ell)= 2n$ for any $i$, then we get perfect decomposition of $K_{2n}$ into copies of $T$.
}

The conjecture has been verified for several very special classes of trees. 
A dynamic survey is maintained by Gallian~\cite{Gal98}.
Note the results in
~\cite{ABHP17} imply an approximate version of Conjecture~\ref{ringel} for
$(n+1)$-vertex trees of degree $o(n/\log n)$.
Recent progress on Conjecture~\ref{ringel} 
(and its bipartite analogue) for random trees was obtained in~\cite{DL14,Lla18}.
Snevily~\cite{Sne97} showed that Ringel's conjecture holds if we replace $K_{2n+1}$ by $K_m$, where $m=\Omega(n^3)$, which improved a bound of Yuster~\cite{Yus00}.
Recently Ringel's conjecture was generalized to allow for more than one tree by B\"ottcher, Hladk\'y, Piguet and Taraz.

\begin{conjecture}[\cite{BHPT16}]\label{BHPTconj43}
Given $n\in\N$, suppose $\cT$ is a collection of trees such that $|T|\leq n+1$ for all $T\in \cT$. 
If $e(\cT) \leq e(K_{2n+1})$, then $\cT$ packs into $K_{2n+1}$.
\end{conjecture}

Note that Theorem~\ref{thm:main simple} immediately implies that Ringel's conjecture holds for all bounded degree trees.
In fact, by concatenating small trees into large ones if necessary, 
one can easily use Theorem~\ref{thm:main simple} to show that the more general Conjecture~\ref{BHPTconj43} holds for all bounded degree trees, too.

\begin{corollary}\label{BHPT 43}
For all $\Delta \in \N$, there is $N\in \N$ such that for all $n\geq N$ the following holds. 
Suppose $\cT$ is a collection of trees such that $|T|\leq n+1$ and $\Delta(T)\leq \Delta$ for all $T\in \cT$. If $e(\cT) \leq e(K_{2n+1})$, then $\cT$ packs into $K_{2n+1}$.
\end{corollary}

\subsection{Quasi-random graphs}
We actually prove our results in the setting of dense quasi-random graphs, which we now define. 
Let $G$ be a graph and let $u,v \in V(G)$.
We denote by $d_G(u)$ the degree of $u$ and by $d_G(u,v)$ the size of the common neighbourhood of $u$ and $v$.
We say a graph $G$ on $n$ vertices is \emph{$(\epsilon,p)$-quasi-random} if $d_G(u)=(1\pm \epsilon)pn$ and $d_G(u,v)=(1\pm \epsilon)p^2n$ for all distinct vertices $u,v\in V(G)$. 

The next theorem is the main result of this paper (and implies all other results stated in the introduction).
It states that
we can decompose a dense quasi-random graph into a family of bounded degree graphs as long as this family contains sufficiently many trees which are not too small and not too large.

\begin{theorem}\label{thm: main result}
For all $\Delta \in \N$ and $\delta>0$, 
there are $N\in \N$ and $\epsilon>0$ such that
for all $n\geq N$ and all $p \in [0,1]$ the following holds.
Suppose $G$ is an $(\epsilon,p)$-quasi-random graph on $n$ vertices, and  $\cH,\cT$ are sets of graphs satisfying 
\begin{enumerate}[label=(\roman*)]
	\item\label{item:MT1} $|J|\leq n$ and $\Delta(J)\leq \Delta$ for all $J\in \cH\cup \cT$,
	\item\label{item:MT2} for all $T\in \cT$, the graph $T$ is a tree and $\delta n \leq |T|\leq (1- \delta)n$,
	\item\label{item:MT3} $|\cT|\geq (1/2+ \delta)n$, and
	\item\label{item:MT4} $e(\cH)+e(\cT)=e(G)$.
\end{enumerate}
Then $G$ decomposes into $\cH \cup \cT$.
\end{theorem}
Observe that $p\geq \delta$ by (ii)--(iv) and that Theorem~\ref{thm: main result} immediately implies Theorem~\ref{thm:main simple}.

A construction which shows that the upper bound in (ii) cannot be omitted completely (even if $\cH=\emptyset$) is described in~\cite[Section 9.1]{BHPT16}.
The bound in (iii) also arises naturally:
suppose $\cT$ consists of $n/2-1$ paths,
the collection of graphs $\cH$ consists only of Eulerian graphs, and $G$ is a regular graph of odd degree on $n$ vertices.
Then we can satisfy $e(\cH\cup \cT)=e(G)$,
but $G$ will not have a decomposition into $\cH \cup \cT$.

In Section~\ref{sec:9.2} we use Theorem~\ref{thm: main result} to deduce Corollary~\ref{cor: trees into quasi-random}.
Corollary~\ref{cor: trees into quasi-random} states that for trees of bounded maximum degree (an analogue of) Ringel's conjecture holds even in the dense quasi-random setting.
Note that Corollary~\ref{cor: trees into quasi-random} immediately implies Corollary~\ref{BHPT 43}.

\begin{corollary}\label{cor: trees into quasi-random}
For all $\alpha,\Delta,p_0  >0$, there are $N\in \N$ and $\epsilon>0$ such that for all $n\geq N$ and $p\geq p_0$ the following holds. 
Suppose $G$ is an $(\epsilon,p)$-quasi-random graph on $n$ vertices. 
Suppose that $\cT$ is a collection of trees such that each $T\in \cT$ satisfies
$|T|\leq (1-\alpha)pn$, $\Delta(T)\leq \Delta$ and $e(\cT) \leq e(G)$. 
Then $\cT$ packs into $G$.
\end{corollary}

A related result of~\cite{KKOT16} guarantees a near-optimal packing of a collection of bounded degree graphs (which are allowed to be spanning) into quasi-random graphs.
As indicated earlier, the results in~\cite{KKOT16} are actually proved in the even more general setting of $\epsilon$-regular graphs
(so they significantly strengthen the classical Blow-up lemma of Koml{\'o}s, S{\'a}rk{\"o}zy and Szemer{\'e}di~\cite{KSS97}). 
In combination with Szemer\'{e}di's regularity lemma these results will be a crucial tool in our proof 
(see Theorem~\ref{thm: blow up} and Theorem~\ref{thm: any graph packing} for the precise statement of the results from~\cite{KKOT16} that we use).

\subsection{Further work and open problems}

Here we discuss some further related directions:
Ferber and Samotij~\cite{FS16} were able to obtain 
approximate decompositions of the binomial random graph $G_{n,p}$ 
(for a wide range of $p$) into
spanning trees which may have large degree.
Related results on approximate decompositions 
into tree factors were previously obtained in~\cite{BFKL14}.

Another direction is to consider host graphs of large degree.
In particular, a very recent result in~\cite{CKKO17} implies that an approximate version of 
the tree packing conjecture for bounded degree trees holds for regular host graphs of degree
$r \ge (1/2+o(1))n$. (In fact, the results in~\cite{CKKO17} extend to approximate
decompositions of almost regular $n$-vertex host graphs of large degree into arbitrary $n$-vertex separable bounded degree graphs.)

A further more general setting is that of graceful labellings --
 Ringel's conjecture is generalized by the notorious graceful labelling conjecture,
 which states that any $n$-vertex tree $T$ can be injectively labelled by using the numbers ${1,2,\dots,n}$ in such a way that the absolute differences induced on the edges are pairwise distinct (see e.g.~\cite{AAGH16} on how this would imply Ringel's conjecture).  
An approximate version of the graceful labelling conjecture for trees of degree
$o(n/\log n)$ was recently proved by
Adamaszek, Allen, Grosu and Hladk\'y~\cite{AAGH16}.

The methods of the current paper might help to obtain exact versions of some of the 
above results.

\section{Outline of the argument}
\subsection{The tree packing conjecture}
Suppose our aim is to pack trees $T_1,\dots, T_n$ with $|T_i|=i$ and $\Delta(T_i)\leq \Delta$ into $G=K_n$, i.e.~we have the setting of Theorem~\ref{thm:glboundeddegree} with uniformly bounded degrees. 
Our proof develops an ``iterative absorbing approach'' going back to \cite{KKO15}.

Roughly speaking, an absorbing approach to find an optimal packing means that we first find and remove an absorbing graph $G_{\rm abs}$ from the host graph $G$. 
In our case, we would then aim to find a near-optimal packing $\phi$ of most of the trees $T_i$ into $G-E(G_{\rm abs})$, which leaves some sparse remainder $G_{\rm rem}$ of $G$ uncovered. 
The properties of $G_{\rm abs}$ then should ensure that we can find an optimal packing of the remaining trees into $G_{\rm abs}\cup G_{\rm rem}$, altogether leading to an (optimal) packing of $T_1,\dots, T_n$ into $G$.

A recent result of \cite{KKOT16} allows us to find the near-optimal packing required in the second step above. 
However, it is far from clear how to construct such an absorbing graph $G_{\rm abs}$. 
The iterative absorbing approach replaces the single absorbing step with a sequence of steps, each designed (amongst others) to reduce the number of uncovered ``leftover edges''. The argument is still concluded with an absorbing step, but the small size of the leftover and the added control over its location now makes this step feasible.

To be more precise, we consider a sequence $V(G)=A_0\supseteq A_1 \supseteq\ldots \supseteq A_{\Lambda}$ of sets with $|A_{i+1}|\ll |A_i|$ and 
such that $|A_\Lambda|\approx n^{1/3}$ and $\Lambda=\Theta(\log{n})$ (see Section~\ref{sec:9.1}). 
After the $i$th step of the iteration, we will ensure that any edges of $G$ not covered so far lie inside $A_i$
and that this leftover is still sufficiently dense, when viewed as having vertex set $A_i$.

We can achieve this as follows.
We split $\{T_1,\ldots,T_n\}$ into sets $\cT_0,\ldots,\cT_{\Lambda-1}$
such that each $\cT_i$ contains approximately $|A_i|$ trees of order at most $|A_i|$.
Assume we have packed $\cT_0\cup\ldots\cup\cT_{i-1}$
in such a way that 
\begin{enumerate}[label=(\alph*)]
	\item all edges in $G_{i-1}:=G[A_{i-1}]-E(G[A_{i}])$ are covered and
	\item most edges in $G[A_{i}]$ are not covered.
\end{enumerate}
Our aim then is to find a packing of most of the trees in $\cT_i$ into $G_i:=G[A_{i}]-E(G[A_{i+1}])$ using the results from \cite{KKOT16}. 
However, a direct application of the results in \cite{KKOT16} to $G_i$ 
would produce a leftover which is too dense.
Thus instead we apply Szemer\'{e}di's regularity lemma to $G_i$ in order to obtain a very efficient decomposition of $G_i$ into a bounded number of blown-up cycles (see Section~\ref{sec: decomp to cycle}). 
We can then use a random walk algorithm as well as the results in \cite{KKOT16} to obtain a very efficient packing of most of $\cT_i$ into these blown-up cycles (see Section~\ref{sec: trees} and Step~\ref{step2} in Section~\ref{sec:iteration}) and thus into $G_i$.

Roughly speaking we next use the ``unpacked'' trees in $\cT_i$ to cover the remaining edges of $G_i$ greedily. 
To achieve this, we will make use of a small number of edges from $G_{i+1}$,
but need to be very careful that we do not affect the structure of $G_{i+1}$ too much.
This task of covering the remaining edges is divided into several steps.
For example, one step consists of covering all remaining edges induced by $A_i\sm A_{i+1}$.
Another step consists of adjusting the parity of the degrees of the remaining graph.
(We achieve this by embedding a leaf onto a vertex of ``incorrect'' parity.)
The tools for carrying out these steps are provided in Section~\ref{sec:cleaning}.
Eventually, we have now packed $\cT_i$ in such a way that (a) and (b) hold with $i$ replaced by $i+1$.

After $\Lambda$ iterations we arrive at a final leftover graph $G_\Lambda^*$ on $A_\Lambda$,
which is almost complete.
So we have a very restricted leftover,
but have so far no method for carrying out the actual absorption step.
In order to prepare for this,
we actually run the iterative process on a slightly modified set of trees:
we choose an arbitrary collection $\cT^*$ of $m$ trees
from $T_2,\ldots,T_{n/2}$ (where $m$ is slightly smaller than $\binom{|A_\Lambda|}{2}$)
and remove a leaf $\ell_{T^*}$ from each $T^*\in \cT^*$.
Let $z_{T^*}$ be the vertex incident to $\ell_{T^*}$.
When carrying out the above iteration we do not embed $T^*\in \cT^*$ but instead we embed $T^*-\ell_{T^*}$
in such a way that $z_{T^*}$ is embedded into $A_\Lambda$
(and no other vertex of $T^*$ is embedded into $A_\Lambda$).

This means that after $\Lambda$ iterations,
we have embedded all (modified) trees $T_1,\ldots,T_n$
except for the leaves $\ell_{T^*}$ (and their incident edges) of the trees in $\cT^*$.
Also,
the uncovered leftover graph $G_{\Lambda}^*$ has exactly $m=|\cT^*|$ edges.
Thus in order to complete the absorption step 
it remains to assign these edges to the images of $T^*-\ell_{T^*}$ in a suitable way.
Hence each edge of $G_{\Lambda}^*$ will be the image of $z_{T^*}\ell_{T^*}$ for some $T^*\in \cT^*$.
(This is the reason why we embedded $z_{T^*}$ into $A_\Lambda$.)
This assignment is carried out by considering a suitable ``out-regular'' orientation of $G_{\Lambda}^*$.
We prove the existence of such an orientation in Section~\ref{sec:orientation}.
Note that in this process,
the role of the final absorbing graph
is played by the images of the modified trees in $\cT^*$.

\subsection{The general setting}
We now discuss the additional ideas needed to prove Theorem~\ref{thm: main result}.
It turns out that considering a quasi-random graph $G$ instead of $K_n$ does not affect the argument significantly.
Moreover,
the family of graphs $\cH$ of bounded degree in Theorem~\ref{thm: main result} can easily be packed (with a quasi-random remainder)
using the results of \cite{KKOT16}.
What does make a difference is that the family $\cT$ of trees in Theorem~\ref{thm: main result} consists entirely of trees of linear size whereas the iterative absorption argument outlined in the previous section only makes sense if in the $i$th iteration
the trees to be embedded have order at most $|A_i|=|G_i|$.

To overcome this problem, we run the iterative absorption with suitable subtrees of the original trees.
More precisely, in the $0$th iteration,
we start by packing almost all of the trees, but from a small proportion of the trees $F\in \cT$,
we cut off a small subforest $F^*$ and denote by $\cT^1$ the set consisting of all the $F^*$.
We then embed each $F-V(F^*)$ into $G[A_0\sm A_1]$.
In the first iteration,
we proceed similarly for $\cT^1$,
i.e.~most forests in $\cT^1$ are embedded in their entirety,
but from the others we cut off small subforests which together form $\cT^2$,
most of which we pack in the second iteration and so on.

A significant difficulty with this approach is of course that if we embed a tree in several iterations,
then the subforests embedded in the various rounds must fit together.
To achieve this we carefully embed the ``intersection points'' (or ``roots'') of these subforests in a separate step.

\subsection{Organisation of the paper}
In Section~\ref{sec: preliminaries},
we introduce some notation,
collect some probabilistic tools,
introduce  Szemer{\'e}di's regularity lemma,
and prove some basic results on packings and decompositions.
In Section~\ref{sec: approx cycle decomp},
we prove a result on near-optimal decompositions of suitable graphs into long cycles.
In Section~\ref{sec: decomp to cycle},
we combine this with Szemer{\'e}di's regularity lemma to obtain near-optimal decompositions of suitable graphs into blown-up cycles.
In Section~\ref{sec: trees},
we combine the results of \cite{KKOT16} with a random walk algorithm to find efficient packings of trees into blown-up cycles.
In Section~\ref{sec:cleaning},
we provide tools for covering a sparse set of leftover edges with trees.
In Section~\ref{sec:iteration},
we then combine the results from Section~\ref{sec: preliminaries}--\ref{sec:cleaning}
to prove the iteration lemma,
which forms the core of the proof of Theorem~\ref{thm: main result}.
In Section~\ref{sec:orientation},
we prove an orientation lemma which will be important in the final absorbing step.
In Section~\ref{sec:final},
we repeatedly apply the iteration lemma and then finally apply the orientation lemma
to prove a version of Theorem~\ref{thm: main result} where we require $\mathcal H=\emptyset$
(see Theorem~\ref{thm: main result step1}).
We then derive Theorem~\ref{thm: main result} itself (using a result from~\cite{KKOT16}).
We also derive  Corollary~\ref{cor: trees into quasi-random} and Theorem~\ref{thm:glboundeddegree}
from Theorem~\ref{thm: main result}.

\section{Preliminaries}
\label{sec: preliminaries}

\subsection{Basic definitions}

For $a,b,c\in \mathbb{R}$ we write $a = b\pm c$ if $b-c \leq a \leq b+c$. In order to simplify the presentation, we omit floors and ceilings and treat large numbers as integers whenever this does not affect the argument. The constants in the hierarchies used to state our results have to be chosen from right to left. More precisely, if we claim that a result holds whenever $1/n \ll a \ll b \leq 1$ (where $n\in \N$ is typically the order of a graph), then this means that there are non-decreasing
functions $f^* : (0, 1] \rightarrow (0, 1]$ and $g^* : (0, 1] \rightarrow (0, 1]$ such that the result holds for all $0 < a, b \leq 1 $ and all $n \in \mathbb{N}$ with $a \leq f^*(b)$ and $1/n \leq g^*(a)$. We will not calculate these functions explicitly. Hierarchies with more constants are
defined in a similar way.

For $N\in \N$, we define $[N]:=\{1,\dots, N\}$.
We define $\binom{X}{k}:=\{A\subseteq X: |A|=k\}$.

We simply refer to ``graphs'' when we consider simple, undirected and finite graphs,
and refer to ``multigraphs'' as graphs with potentially parallel edges, but without loops.

Let $G$ be a multigraph
and let $u,v \in V(G)$ and $U,V\sub V(G)$ such that $U\cap V=\es$.
We write $G[U,V]$ to denote the bipartite (multi-)subgraph of $G$ induced by the edges joining $U$ and $V$
and let $e_G(U,V):=e(G[U,V])$.
We extend $e_G(U,V)$ to sets $U,V$ which are not necessarily disjoint by
defining $$e_G(U,V) := e_{G}(U\setminus V, V) + e_{G}( V\setminus U, U\cap V) + 2e(G[U\cap V]).$$
In addition, let $\den_G(U,V):= e_G(U,V)/(|U||V|)$.

We define $d_{G,U}(v):=|\{e \in E(G): e=uv, u\in U\}|$. 
If $G$ is simple, we let $d_{G,U}(u,v):=|\{w \in U:  uw , vw\in E(G)\}|$.
We denote by $N_G(v)$ the set of all neighbours of $v$
and by $N_G(u,v)$ the set of all common neighbours of $u,v$.
We define $N_G[v]:=N_G(v)\cup \{v\}$.

Given $E\sub E(G)$, we write $V(E)$ for the set of all endvertices of edges in $E$.
It will sometimes be convenient to identify a set of edges $E\sub E(G)$ with the subgraph $H$ of $G$
such that $V(H)=V(E)$ and $E(H)=E$.
In particular, we write $\Delta(E)$ for $\Delta(H)$.

We say a set $I\sub V(G)$ in a graph $G$ is {\em $k$-independent} if for any two distinct $u,v\in I$, the distance between $u$ and $v$ in $G$ is at least $k$. 
Thus a $2$-independent set is just an independent set. 
We say an edge set $M$ is a {\em $k$-independent matching} if any two edges in $M$ have distance at least $k$. 
Note that a $2$-independent matching is an induced matching.

Given a packing of $\cH=\{H_1,\dots, H_s\}$ into $G$, we can naturally associate a function $\phi: \bigcup_{H\in \cH} V(H)\cup E(H)\to V(G)\cup E(G)$ with such a packing and we also say $\phi$ \emph{packs} $\cH$ into $G$.  For simplicity, we write $\phi(\cH):= \bigcup_{H\in \cH}\phi(H)$. 
Given $H\in \cH$, we sometimes abuse notation a little and view $\phi(H)$ both as a subgraph of $G$ and as a subset of $V(G)$.
We say $\cH$ \emph{decomposes} $G$ or $G$ \emph{has a decomposition into} $\cH$ if $\cH$ packs into $G$ and $e(G)=e(\cH)$ where $e(\cH):= \sum_{H\in \cH} e(H)$.

We say that a function $f$ on domain $\dom(f)$ is {\em consistent} with a function $g$ if $f(x)=g(x)$ for any $x \in \dom(f)\cap \dom(g)$. 
We say that a function $f$ is {\em consistent} with a collection of functions $\{g_1,\dots, g_{m}\}$ if $f$ is consistent with $g_i$ for every $i\in [m]$. 

We say that a bipartite graph $G$ with vertex partition $(A,B)$ is \emph{$\epsilon$-regular}
if for all sets $A'\subseteq A$, $B'\subseteq B$ with  $|A'|\geq \epsilon |A|$, $|B'|\geq \epsilon |B|$, we have
\begin{align*}
	| \den_G(A',B')- \den_G(A,B)| < \epsilon.
\end{align*}
Moreover,
$G$ is \emph{$(\epsilon,d)$-regular} if
for all sets $A'\subseteq A$, $B'\subseteq B$ with  $|A'|\geq \epsilon |A|$, $|B'|\geq \epsilon |B|$,
we have
\begin{align*}
	| \den_G(A',B')- d| < \epsilon.
\end{align*}
If $G$ is $(\epsilon,d)$-regular and $d_{G}(a)= (d\pm \epsilon)|B|$ for all $a\in A$ and $d_{G}(b)= (d\pm \epsilon)|A|$ for all $b\in B$, then we say that $G$ is {\em $(\epsilon,d)$-super-regular}.

We say a multigraph $G$ is {\em $(\beta,\alpha)$-dense} if any two (not necessarily disjoint) sets $U, V\subseteq V(G)$ with $|U|,|V|\geq \beta |G|$ satisfy
$$\den_G(U,V)\geq \alpha.$$

A bipartite graph with vertex partition $(A,B)$ is \emph{$(\epsilon,p)$-quasi-random}
if for any two distinct vertices $u,v \in A$, we have $d(u)=(1\pm \epsilon)p|B|$ and $d(u,v)=(1\pm \epsilon)p^2|B|$
and for any two distinct vertices $u,v \in B$, we have $d(u)=(1\pm \epsilon)p|A|$ and $d(u,v)=(1\pm \epsilon)p^2|A|$.

We say a partition $(V_1,\dots, V_k)$ of a set $V$ is an {\em equitable partition} if $||V_i|-|V_j||\leq 1$ for all $i,j \in [k]$.
We refer to a graph $G$ as having equitable partition $(V_1,\dots, V_k)$
if $(V_1,\dots, V_k)$ is an equitable partition of $V(G)$.
Note that in this case we still allow $G$ to have edges joining two vertices in $V_i$ for $i\in [k]$.

For $k,\ell\in\N$ with $\ell\geq 3$,
we define an \emph{$\ell$-cycle $k$-blow-up} as the graph $C(\ell,k)$ with 
vertex set consisting of the disjoint union of $V_1,\ldots,V_\ell$ such that $|V_i|=k$ for all $i\in [\ell]$.
Moreover, the edge set consists of all edges $uv$
with $u\in V_i,v\in V_{i+1}$ for $i\in [\ell]$, where we consider the index $i$ modulo $\ell$ (and we will always do this in such a case).
We refer to the sets $V_i$ as the \emph{clusters} of $C(\ell,k)$, 
and refer to $\ell$ as the {\em length} of the cycle blow-up.
Sometimes we omit the ``$\ell$" if we speak about an $\ell$-cycle $k$-blow-up for some $\ell$.

A spanning subgraph $G_1$ of $C(\ell,k)$ is \emph{internally $k'$-regular} for some $k'\in \N$ if $G_1[V_i,V_{i+1}]$ is $k'$-regular for all $i \in [\ell]$.
A spanning subgraph $G_2$ of $C(\ell,k)$ is an \emph{$(\epsilon,d)$-(super)-regular $\ell$-cycle $k$-blow-up}
if $G_2[V_i,V_{i+1}]$ is $(\epsilon,d)$-(super)-regular for all $i\in [\ell]$.


\subsection{Probabilistic tools}
We frequently use random processes to show that (sub-)graphs with certain properties exist.
To this end, we will use several concentration inequalities which we introduce now.
A sequence $X_0,\dots, X_N$ of random variables is a {\em martingale} 
if $X_0$ is a fixed real number and $\mathbb{E}[X_{n}\mid X_0,\dots,X_{n-1}] = X_{n-1}$ for all $n\in [N]$. 
We say that the martingale $X_0,\dots, X_N$ is {\em $c$-Lipschitz} if $|X_{n}-X_{n-1}| \leq c$ holds for all $n\in[N]$. 
Our applications of Azuma's inequality will involve \emph{exposure martingales} (also known as Doob martingales). These are martingales of the form $X_i:=\mathbb{E}[X\mid Y_1,\dots, Y_i]$, where $X$ and $Y_1,\dots,Y_i$ are some previously defined random variables.

\begin{theorem}[Azuma's inequality \cite{Azu67, Hoe63}]\label{Azuma} 
Suppose that $\lambda, c >0$ and that $X_0,\dots, X_N$ is a $c$-Lipschitz martingale. 
Then 
\begin{align*}
\mathbb{P}[\left|X_N-X_0\right|\geq \lambda]\leq 2e^{\frac{-\lambda^2}{2Nc^2}}.
\end{align*}
\end{theorem}

For $m,n,N\in \mathbb{N}$ with $m,n<N$, 
the \emph{hypergeometric distribution} with parameters $N$, $n$ and $m$ is the distribution of the random variable $X$ defined as follows. 
Let $S$ be a uniformly chosen random subset of $[N]$ of size $n$ and let $X:=|S\cap [m]|$. 
The first part of the following lemma is essentially same as Corollary~A.1.7 in~\cite{AS08} and the second part is from \cite{HS05}.

\begin{lemma}
[see {\cite{AS08, HS05}}] \label{lem: chernoff} 
Suppose $X_1,\dots, X_n$ are independent random variables such that $\mathbb{P}[X_i=0]=p_i$ and $\mathbb{P}[X_i=1]=1-p_i$ for all $i\in [n]$. Let $X:= X_1+\dots + X_n$. Then for all $t>0$, $\mathbb{P}[|X - \mathbb{E}[X]| \geq t] \leq 2e^{-t^2/(2n)}$. 
Suppose $Y$ has a hypergeometric distribution with parameters $N,n,m$,
then
$\mathbb{P}[|Y - \mathbb{E}[Y]| \geq t] \leq 2e^{-t^2/(3n)}$.
\end{lemma}

We say that $(X_t)_{t\geq 0}$ is a \emph{symmetric random walk} on $[\ell]$ if $X_0\in [\ell]$ and for all $t\geq 1$,
we have $\mathbb{P}[X_{t}=X_{t-1}+1]=\mathbb{P}[X_{t}=X_{t-1}-1]=1/2$,
independently of $X_0,\ldots,X_{t-1}$,  where $X_t$ is considered modulo $\ell$. 

\begin{lemma}
{\rm (see \cite[Chapter 18.4]{Kle14})}
\label{lem: random walk}
Let $\ell\geq 3$ be an odd integer. 
If $(X_t)_{t\geq 0}$ is a symmetric random walk on~$[\ell]$,
then there is a $0<\gamma<1$ and an integer $t_0>0$ such that for every $i\in [\ell]$ and every $t\geq t_0$,
$$\Pro[X_t=i]=\frac{1\pm \gamma^t}{\ell}.$$
\end{lemma}

\subsection{Regularity and quasi-randomness tools}

In this section we collect several basic statements involving graph regularity and quasi-randomness as well as
some of their consequences.
The following three propositions follow easily from the definition of $(\epsilon,d)$-regularity.

\begin{proposition}\label{prop: reg smaller}
Let $0<\epsilon\leq \delta \leq d \leq 1$. Suppose $G$ is an $(\epsilon,d)$-regular bipartite graph with vertex partition $(A,B)$ and $A' \subseteq A,B'\subseteq B$ with ${|A'|}/{|A|}, {|B'|}/{|B|}\geq \delta$.
Then $G[A',B']$ is $(\epsilon/\delta, d)$-regular.
\end{proposition}

\begin{proposition}\label{prop: matching}
Suppose $n\in \N$ with $1/n\ll \epsilon \ll d \leq 1$.
If $G$ is an $(\epsilon,d)$-regular bipartite graph with vertex partition $(A,B)$ and $n=|A|\leq|B|$,
then there is a matching of size at least $(1-\epsilon)n$ in $G$.
\end{proposition}

\begin{proposition}\label{prop: reg right deg}
Suppose $G$ is an $(\epsilon, d)$-regular bipartite graph with vertex partition $(A,B)$ and $B' \subseteq B$ with $|B'|\geq \epsilon|B|$. Then all but
at most $2\epsilon |A|$ vertices in $A$ have degree $(d \pm \epsilon)|B'|$ in $B'$.
\end{proposition}

The following result follows easily from Proposition~\ref{prop: reg smaller} and \ref{prop: reg right deg}.

\begin{proposition}\label{prop: making super reg}
Suppose $\ell,n\in \N$ and $1/n \ll\epsilon\ll d\leq 1$.
Then
every $(\epsilon,d)$-regular $\ell$-cycle $n$-blow-up
contains a $(6\epsilon,d)$-super-regular $\ell$-cycle $(1- 4\epsilon)n$-blow-up.
\end{proposition}
\COMMENT{\begin{proof}
Let $G$ be a $(\epsilon,d)$-regular $\ell$-cycle $n$-blow-up
with vertex set $V_1\cup\ldots\cup V_\ell$.
For every $i\in [\ell]$,
let $X_i$ and $Y_i$ be the set of vertices in $V_i$
that do not have $(d\pm\epsilon)n$ neighbours in $V_{i-1}$ and $V_{i+1}$, respectively, where we consider the indices modulo $\ell$.

By Proposition~\ref{prop: reg right deg}, $|X_i|,|Y_i|\leq  2\epsilon n$.
Let $Z_i\subseteq V_i$ be a set of $4\epsilon n$ vertices containing $X_i \cup Y_i$.
Let $G'$ be the graph obtained from $G$ by deleting all vertices in $Z_1\cup \ldots \cup Z_\ell$.
By Proposition~\ref{prop: reg smaller}, $G'$ is a $(6\epsilon,d)$-super-regular $\ell$-cycle $(1- 4\epsilon)n$-blow-up.
\end{proof}}

We will also use the next result from~\cite{DLR95}.
(In~\cite{DLR95} it is proved in the case when $|A|=|B|$ with $16\epsilon^{1/5}$ instead of $\epsilon^{1/6}$.
The version stated below can be easily derived from this.)

\begin{theorem}\label{thm: almost quasirandom}
Suppose $n\in \N$ with $1/n\ll\epsilon \ll \alpha,p \leq 1$.
Suppose $G$ is a bipartite graph with vertex partition $(A,B)$ such that $|A|=n$, $\alpha n \leq |B|\leq \alpha^{-1}n$ and at least $(1-5\epsilon)n^2/2$ pairs $u,v\in A$
satisfy $d(u),d(v)\geq (p- \epsilon)|B|$ and $d(u,v)\leq (p+ \epsilon)^2|B|$.
Then $G$ is $\epsilon^{1/6}$-regular. In particular, if $G$ is $(\epsilon,p)$-quasi-random, it is $(\epsilon^{1/6},p)$-super-regular.
\end{theorem}

The next three propositions involve basic properties of quasi-random and of $(\beta,\alpha)$-dense graphs which follow easily from the definitions.
\begin{proposition}\label{prop: quasi-random subgraph}
Suppose $\epsilon \leq 1/10$. Let $G$ be a $(\epsilon,p)$-quasi-random graph on $n$ vertices. 
Let $U \sub V(G)$ be a set of vertices with $|U|\geq (1-\epsilon) n$ and 
let $E_1,E_2$ be collections of edges on $V(G)$ such that $\Delta(E_1)\leq \epsilon n$ and $\Delta(E_2)\leq \epsilon n$. 
Then $(G\cup E_1)[U] -E_2$ is $(10p^{-2}\epsilon,p)$-quasi-random.
\end{proposition}
\COMMENT{\begin{proof}
Let $G':= (G\cup E_1)[U]-E_2$.
For $u,v \in V(G')$, 
we have $d_{G'}(u,v) = d_{G}(u,v) \pm (\epsilon n + 2\Delta(E_1)+ 2\Delta(E_2)) = d_{G}(u,v) \pm 5\epsilon n = (1\pm 10 p^{-2}\epsilon)p^2 |U|$.\COMMENT{ $(1\pm 6 p^{-2}\epsilon)p^2 n = (1\pm 6 p^{-2}\epsilon)p^2|U|/(1\pm \epsilon) = (1\pm 10 p^{-2}\epsilon)p^2 |U|$}
In a similar way, $d_{G'}(u)= d_{G}(u)\pm 3\epsilon n = (1\pm 10p^{-1}\epsilon)p |U|$. Thus $G'$ is $(10p^{-2}\epsilon,p)$-quasi-random.
\end{proof}}

\begin{proposition}\label{prop: quasi-random implies dense}
Suppose $n\in \mathbb{N}$ with $1/n\ll \beta\ll p \leq 1$. If $G$ is a $(\beta,p)$-quasi-random graph on $n$ vertices, then it is 
$(\beta^{1/6},p-\beta^{1/6})$-dense.
\end{proposition}
\begin{proof}
Let $V:=V(G)=\{v_1,\dots, v_n\}$ and $V'=\{v'_1,\dots, v'_n\}$. 
Consider the bipartite graph $H$ with bipartition $V\cup V'$ such that $v_iv'_j \in E(H)$ if and only if $v_iv_j \in E(G)$.
Observe that $H$ is $(\beta,p)$-quasi-random. Hence, by Theorem~\ref{thm: almost quasirandom}, $H$ is $(\beta^{1/6},p)$-super-regular.
Let $U_1, U_2\subseteq V$ with $|U_1|,|U_2|\geq \beta^{1/6}n$. 
Let $U'_2=\{v'_i: v_i\in U_2\}$.
We conclude that $\den_G(U_1,U_2)= \den_H(U_1,U'_2) \geq p- \beta^{1/6}$.
Thus $G$ is $(\beta^{1/6},p-\beta^{1/6})$-dense.
\end{proof}

\begin{proposition}\label{prop: dense deletion}
Suppose $0< \alpha,\beta < 1$ and $C\geq 1$. 
If $G$ is a $(\beta,\alpha)$-dense graph on $n$ vertices and $G'$ is a spanning subgraph of $G$ such that $d_{G}(v)- d_{G'}(v) \leq C \beta n$ for all $v\in V(G)$, 
then $G'$ is $(C\beta^{1/2},\alpha-\beta^{1/2})$-dense. 
\end{proposition}
\COMMENT{
\begin{proof}
For any two sets $U,U'\subseteq V(G)$ with $|U|,|U'|\geq C \beta^{1/2} n$, 
we obtain $$\den_{G'}(U,U') 
\geq \frac{e_{G}(U,U') - C\beta n|U|}{|U||U'|} 
\geq \alpha  - \frac{C\beta n}{|U'|} 
\geq \alpha - \beta^{1/2},$$ 
which shows that $G'$ is $(C\beta^{1/2},\alpha-\beta^{1/2})$-dense. 
\end{proof}
}

The next result is a slight extension of Szemer\'edi's regularity lemma. We will apply it in Section~\ref{sec: decomp to cycle}.

\begin{lemma}\label{lem: rl}
Suppose $M,M',n,s,t \in \N$, $1/n \ll 1/M \ll 1/M'\ll \epsilon\ll 1/t\ll \beta \ll \alpha,1/s,d\leq 1$.
Suppose $G$ is a graph on $n$ vertices with equitable partition $(U_1,\dots, U_s)$. 
Then there is a $k\in \N$ such that $G$ has a vertex partition $(V_0,V_1,\ldots,V_k)$ satisfying the following,
where we write $d_{ij}:=\den_G(V_i,V_j)$:
\begin{enumerate}[label=(\Roman*)]
	\item $M'\leq k \leq M$,
	\item $|V_0|\leq \epsilon n$,
	\item $|V_1|=\ldots=|V_k|$, 
	\item for every $i \in [k]$, the graph $G[V_i,V_j]$ is $(\epsilon,d_{ij})$-regular except for at most $\epsilon k$ indices $j\in[k]$, and  
	\item for every $i\in [k]$, there exists a unique $j \in [s]$ such that $V_i\subseteq U_j$
	and $L:=|\{i:V_i \sub U_{j'}\}|=|\{i:V_i \sub U_{j''}\}|$ is odd for all $j',j''\in [s]$.
\end{enumerate}
Let $R_t$ be the multigraph on $[k]$ such that $R_t$ contains an edge $ij$ with multiplicity $\lfloor t d_{ij} \rfloor$ 
if and only if $G[V_i,V_j]$ is $(\epsilon,d_{ij})$-regular and $d_{ij}\geq 1/t^{1/2}$. 
Then the following hold:
\begin{enumerate}
\item[(VI)] For each $j\in [s]$, let $U_j^R:=\{i\in [k]:V_i\sub U_j\}$.
If $d_{G, U_j}(v) = (d\pm \beta) |U_j|$ for every $v\in V(G)$, 
then for every $i\in [k]$ we have $$d_{R_t,U_j^R}(i)= (d\pm 3\beta )t|U^R_j|.$$
\item[(VII)] If $G$ is $(\beta,\alpha)$-dense, then $R_t$ is $(2\beta, t\alpha/2)$-dense.
\end{enumerate}
\end{lemma}
\begin{proof}
As (I)--(V) are standard, we omit the proof.
Let $d_{ij}':=d_{ij}$ if $G[V_i,V_j]$ is $(\epsilon,d_{ij})$-regular and $d_{ij}\geq 1/t^{1/2}$ and let $d_{ij}'=0$ otherwise.
We first show (VI). 
So assume that $d_{G,U_j}(v) = (d\pm \beta) |U_j|$ for all $v\in V(G)$. 
Let $e_{ij}$ be the number of edges between $V_i$ and $U_j$ in $G$ and let $ \overline{n}:=|V_1|=\dots=|V_k|$.
Note that $\overline{n}= (1\pm \epsilon)n/k.$
Then
\begin{align}\label{eq: e ij 1}
	e_{ij} 
	&= \sum_{v\in V_i}d_{G,U_j}(v) \pm e_G(V_i)
	= (d\pm 3\beta/2)|U_j||V_i| \nonumber \\
	&= (d\pm 3\beta/2)(|U_j^R|\overline{n}\pm \epsilon n)\overline{n}
	=(d\pm 2\beta)|U_j^R|\overline{n}^2.
\end{align}
By (IV), there are at most $\epsilon k \overline{n}^2$ edges between $V_i$ and all those $V_{j'}$ for which $G[V_i,V_{j'}]$ is not $(\epsilon,d_{ij'})$-regular.
Moreover, there are at most $k\overline{n}^2/t^{1/2}$ edges between $V_i$ and all those $V_{j'}$ for which $d_{ij'}\leq 1/t^{1/2}$. Furthermore, $e_{G}(V_i,U_j\cap V_0) \leq \epsilon \overline{n}n$. Thus 
\begin{align}\label{eq: e ij 2}
e_{ij} &= \sum_{j'\in U_j^R} d_{ij'}' \overline{n}^2 \pm (\epsilon k \overline{n}^2 + k\overline{n}^2/t^{1/2} + \epsilon \overline{n} n)  \nonumber \\
        &= \frac{\overline{n}^2}{t}\sum_{j'\in U_j^R} \lfloor td_{ij'}'\rfloor \pm \frac{2k\overline{n}^2}{t^{1/2}} 
				= d_{R_t, U_j^R}(i) \frac{\overline{n}^2}{t} \pm \frac{2k\overline{n}^2}{t^{1/2}} .
\end{align}
By combining \eqref{eq: e ij 1} and \eqref{eq: e ij 2}, we obtain
\COMMENT{
\begin{align*}
	((d\pm 2\beta)|U_j^R| \pm \frac{2k}{t^{1/2}})\overline{n}^2 \cdot \frac{t}{\overline{n}^2}
	= (d\pm 2\beta)t|U_j^R| \pm 2k t^{1/2}
	= (d\pm 2\beta \pm \frac{2k}{t^{1/2}|U_j^R|})t|U_j^R| 
	= (d\pm 3\beta )t|U_j^R| 
\end{align*}
since $|U_j^R|= \frac{k}{s}$ and $1/t\ll \beta$.
}
\begin{align*}
d_{R_t,U_j^R}(i)= (d\pm 3\beta)t|U_j^R|,
\end{align*}
which implies (VI).

To prove (VII), let $W,W' \subseteq V(R_t)$ with $|W|,|W'|\geq 2\beta |R_t|=2\beta k$.
Thus
$|\bigcup_{i\in W} V_i|, |\bigcup_{i\in W'} V_{i}|\geq \beta n$.
Since $G$ is $(\beta,\alpha)$-dense, this implies that
\begin{align}\label{eq: e G 1}
 e_G\left(\bigcup_{i\in W} V_i, \bigcup_{i\in W'} V_{i}\right) 
\geq \alpha \left|\bigcup_{i\in W} V_i\right|\left|\bigcup_{i\in W'} V_i\right| 
 \geq \alpha |W||W'|\overline{n}^2.
\end{align}
On the other hand, 
\begin{eqnarray}\label{eq: e G 2}
e_G\left(\bigcup_{i\in W} V_i, \bigcup_{i\in W'} V_{i}\right) 
&\stackrel{(IV)}{\leq}& \sum_{i\in W, i'\in W'} d_{ii'}'\overline{n}^2 + \epsilon k |W| \overline{n}^2 + \frac{k|W|\overline{n}^2}{t^{1/2}}\nonumber \\
 &\leq &\sum_{i\in W, i'\in W'} ( \lfloor t d_{ii'}' \rfloor + 1)\frac{\overline{n}^2}{t}  + \frac{2 k |W|\overline{n}^2}{t^{1/2}} \nonumber \\
 & \leq& e_{R_t}(W,W') \frac{\overline{n}^2}{t} + \frac{|W||W'|\overline{n}^2}{t^{1/3}}.
\end{eqnarray}
By combining \eqref{eq: e G 1} and \eqref{eq: e G 2}, we obtain
\begin{align*}
e_{R_t}(W,W') 
&\geq \alpha t |W||W'|  - t^{2/3}|W||W'| \geq t\alpha|W||W'|/2.
\end{align*}
Thus $R_t$ is $(2\beta,t\alpha/2)$-dense.
\end{proof}

In Section~\ref{sec:2.5} and Section~\ref{sec:final} we will use the following version of the Blow-up lemma of Koml\'{o}s, S\'{a}rk\"ozy and Szemer\'{e}di (see Remark 8 in \cite{KSS97}).

\begin{theorem} \label{blow up target sets}
Suppose $r,n, \Delta\in \N$ and $1/n \ll \epsilon \ll d,d_0,1/\Delta,1/r\leq 1$.
Let $H$ be an $r$-partite graph with $\Delta(H)\leq \Delta$ with vertex partition $(X_1,\dots, X_r)$ and 
let $G$ be an $r$-partite graph with vertex partition $(V_1,\dots, V_r)$ such that $|X_i|=|V_i| \in \{n,n+1\}$ for all $i\in [r]$.
Suppose that $G[V_i,V_j]$ is $(\epsilon,d)$-super-regular for each $i\neq j\in [r]$. 
Let $X'$ be a subset of $V(H)$ with $|X'|\leq \epsilon n$ and for each $i\in [r]$ and $x\in X'\cap X_i$, 
let $A_x\subseteq V_i$ with $|A_x|\geq d_0 n$. Then there exists an embedding $\phi : H\rightarrow G$ such that $\phi(x) \in A_{x}$ for all $x\in X'$.
\end{theorem}

\subsection{Graph packing and decomposition tools.}
\label{sec2.4}
In this subsection we will collect some results about packings and decompositions of graphs that we will use later on.
We start with three simple results about splitting an $(\epsilon,d)$-regular or a $(\beta,\alpha)$-dense graph into suitable edge-disjoint subgraphs.

\begin{proposition}\label{prop: part  eps reg}
Let $k,n\in \N$ and $1/n \ll\epsilon \ll  d,d',1/k \leq 1$.
Let $G$ be a bipartite graph with vertex partition $(A,B)$ such that $|A|=|B|=n$.
If $G$ is $(\epsilon,d)$-regular and $d'k\leq d$,
then there exist edge-disjoint $(2\epsilon,d')$-regular spanning subgraphs $G_1,\ldots, G_k$ of $G$ 
such that $d_{G'}(u)\leq (1-kd'/d)n + n^{2/3}$ for all $u\in A\cup B$, where $G':=G-\bigcup_{i=1}^k E(G_i)$.

\end{proposition}
\begin{proof}
We colour each edge of $G$ with colour $i\in [k]$ with probability $d'/d$ and with colour $0$ with probability $1-kd'/d$ at random independently from all other edges. 
Let $G_i$ be the graph consisting of the edges of colour $i$.
Straightforward applications of Lemma~\ref{lem: chernoff} show that $G_1,\ldots,G_k$ are $(2\epsilon,d')$-regular
and  $d_{G'}(u)\leq (1-kd'/d)n + n^{2/3}$ for all $u\in A\cup B$ with positive probability, 
in particular, such graphs exist.
\end{proof}

\begin{proposition}\label{prop: decomp multigraph}
Suppose $k,n\in \N$ with $1/n\ll 1/k,\alpha, \beta,d \leq 1$, and $\beta \ll \alpha, d$. 
Let $G$ be a $(\beta,k\alpha)$-dense multigraph on $n$ vertices 
whose edge multiplicity is at most $k$.
If $d_{G}(v) = k(d\pm \beta)n$ for all $v\in V(G)$,
then $G$ can be decomposed into $k$ spanning edge-disjoint (simple) graphs $G_1,\ldots, G_k$
such that for each $i\in[k]$, the graph $G_i$ is $(\beta,\alpha -\beta)$-dense and $d_{G_i}(v) = (d\pm 2\beta) n$ for all $v\in V(G)$.
\end{proposition}
\begin{proof}
For every set $E$ of $\ell$ parallel edges of $G$, 
we choose an $\ell$-subset of $[k]$ uniformly at random and label every edge in $E$ with a distinct member of this $\ell$-subset. 
For each $i\in [k]$, 
let $G_i$ be the spanning subgraph of $G$ consisting of all edges with label $i$.
Lemma~\ref{lem: chernoff} shows that the $G_i$ have the desired properties with probability at least $1/2$, in particular, such a decomposition exists.
\COMMENT{ 
Let $c$ be a positive real number such that $c\ll \beta$.
Note that if there are $m(uv)$ edges between $u$ and $v$, then $G_i$ has an edge $uv$ with probability $m(uv)/k$.
Fix $i\in [k]$. For a pair $u,v$, let $X_{uv}=1$ if $uv \in E(G_i)$ and $X_{uv}=0$ otherwise.

Thus for two sets $A,B$ with $|A|,|B|\geq \beta n$ and vertex $u$, 
 $\mathbb{E}[d_{G_i}(u)] = \sum_{v\neq u} m(uv)/k = \frac{d_{G}(u)}{k}$ and $\mathbb{E}[e_{G_i}(A,B)] = \sum_{v\in A, w\in B} m(vw)/k = \frac{e_{G}(A,B)}{k} \geq \alpha|A||B|.$.

Note that $d_{G_i}(u) = \sum_{v\in N_{G}(u)} X_{uv}$ is sum of at least $d_{G}(u)/k \geq d n/2$ independent random variables, and $e_{G_i}(A,B) = \sum_{ v\in A, w\in B} X_{vw}$ is sum of at least $e_{G}(A,B)/(2k) \geq \alpha \beta^2 n^2 /(4k)$ independent random variables.

Thus Lemma~\ref{lem: chernoff} shows that
$$\mathbb{P}[ d_{G_i}(u) = (d\pm 2\beta)n ] \geq 1 - (1-c)^n,$$
and 
$$\mathbb{P}[ e_{G_i}(A,B) \geq  \alpha|A||B| - \beta|A||B|] \geq 1- (1-c)^{n^2}.$$ 
Since there are $n$ vertices and at most $2^{2n}$ pairs $(A,B)$ of sets, union bounds gives that for all $i\in [k]$ $G_i$ is $(\beta,\alpha -\beta)$-dense and $d_{G_i}(v) = (d\pm 2\beta) n$ for all $v\in V(G)$ with probability at least $1 - k n (1-c)^n - k 2^{2n} (1-c)^{n^2} \geq 1/2$. Thus the conclusion holds.
}
\end{proof}

In Section~\ref{sec: decomp to cycle}, we will need to find an approximate decomposition of an almost regular graph into long but not quite spanning cycles. As an intermediate step, the following lemma gives a decomposition into ``almost'' spanning almost regular subgraphs.

\begin{lemma}\label{lem: D partition}
Suppose $D,k,n\in \N$ with $1/n\ll d,1/k,\alpha,\beta, 1/D \leq 1$, and $\beta\ll \alpha,d,1/D$, and $D\geq 6$.
Let $G$ be a $(\beta,k\alpha)$-dense multigraph on $n$ vertices whose edge multiplicity is at most $k$. 
Suppose that $(U_1,\dots, U_D)$ is an equitable partition of $V(G)$ such that $d_{G,U_i}(v)= (d\pm \beta)k|U_i|$ for all $v\in V(G)$ and $i\in [D]$.
Then we can decompose $G$ into $\{ G_S : S\in \binom{[D]}{D-3}\}$ such that for all $S\in \binom{[D]}{D-3}$
\begin{enumerate}[label=(\Roman*)]
\item $G_S$ is a $(2\beta,k\alpha D^{-3})$-dense multigraph with $V(G_S)=\bigcup_{i\in S}U_i$, and
\item $d_{G_S}(v) = (d\pm 2\beta)k\binom{D-1}{3}^{-1}n$ for all $v\in \bigcup_{i\in S}U_i$.
\end{enumerate}
In particular,
$e(G_S)=(1\pm \beta^{2/3})\binom{D}{3}^{-1}e(G)$.
\end{lemma}
\begin{proof}
We show that a suitable random decomposition produces a decomposition as desired with positive probability.
We randomly label every edge of $G$ with a set from $\binom{[D]}{D-3}$ according to the following rules.
For each $i\in [D]$ and each edge in $G[U_i]$,
choose a set uniformly at random from $\{S: i \in S \in \binom{[D]}{D-3}\}$.
For $i,j\in [D]$ with $i\neq j$ and each edge in $G[U_i,U_j]$,
choose a set uniformly at random from $\{S: i,j \in S \in \binom{[D]}{D-3}\}$.

For each $S\in \binom{[D]}{D-3}$, 
let $G_S$ be the random (multi-)subgraph of $G$ with vertex set $\bigcup_{i\in S} U_i$ whose edge set is the set of all edges with label $S$. Therefore,
for all $S\in \binom{[D]}{D-3}$ and $u\in U_i$ with $i\in S$, we obtain
\begin{align*}
\mathbb{E}[d_{G_S}(u)] 
&= \binom{D-1}{D-4}^{-1} d_{G,U_i}(u)+ \binom{D-2}{D-5}^{-1} d_{G,\bigcup_{j\in S\setminus\{i\}}U_j}(u )\\
&= \frac{6(d\pm \beta)kn}{D(D-1)(D-2)(D-3)} +\frac{6(D-4)(d\pm \beta)kn}{D(D-2)(D-3)(D-4)}\\
&=\binom{D-1}{3}^{-1}(d\pm \beta)kn.
\end{align*}
Let $S\in \binom{[D]}{D-3}$ and $W,W'\subseteq \bigcup_{i\in S} U_i$ with $|W|,|W'|\geq 2\beta |\bigcup_{i\in S} U_i| \geq \beta n$.
Observe that every edge in $G[\bigcup_{i\in S} U_i]$ is contained in $G_S$ with probability at least $2D^{-3}$.
Thus, as $G$ is $(\beta,k\alpha)$-dense,
\begin{align*}
\mathbb{E}[e_{G_S}(W,W')]  \geq 2D^{-3} k\alpha|W||W'|.
\end{align*}
Straightforward applications of Lemma~\ref{lem: chernoff}
lead to the desired result.
\end{proof}

We call a path of length $2$ a \emph{seagull}
and a graph consisting of the vertex-disjoint union of $k$ seagulls a \emph{flock of seagulls of size $k$}
or just a \emph{flock of seagulls} or a \emph{flock}.
Moreover, let the vertices of degree one of a seagull be its \emph{wings}. The following simple proposition guarantees a decomposition of a suitable bipartite graph into edge-disjoint flocks of seagulls.

\begin{proposition}\label{prop: bip seagull}
Let $G$ be a bipartite graph with vertex partition $(A,B)$.
Suppose every vertex in $B$ has even degree.
Then $G$ can be decomposed into at most $3\Delta(G)$ edge-disjoint flocks of seagulls with wings in $A$.
\end{proposition}
\begin{proof}
Since every vertex $u\in B$ has even degree, 
we can decompose the edges incident to $u$ into $d_{G}(u)/2$ edge-disjoint seagulls with wings in $A$. 
Let $\{S_1,\dots, S_t\}$ be a corresponding decomposition of $G$ into seagulls.
Let $\cH$ be the graph with vertex set $\{S_1,\dots, S_t\}$ and edge set $\{S_iS_j: V(S_i)\cap V(S_j)\neq \emptyset\}$. 
Thus we have $\Delta(\cH) \leq 3\Delta(G)-1$.\COMMENT{Actually, $5\Delta(G)/2-3$.} 
Hence we can properly colour $\cH$ by using $3\Delta(G)$ colours, and each colour class gives rise to a flock of seagulls.
Thus there exists a decomposition of $G$ into $3\Delta(G)$ edge-disjoint flocks of seagulls.
\end{proof}

The following observation will be used to find suitable partitions (e.g. of flocks of seagulls).

\begin{proposition}\label{prop: weight partition}
Let $X$ be a set of size $n$ and let $w:X\to [0,M]$. 
Let $w(X'):= \sum_{x\in X'}w(x)$ for all $X'\subseteq X$.
For each $m \leq n$, 
there is a partition of $X$ into $X_1,\dots, X_m$ 
such that $|X_i|\leq \lceil2n/m\rceil$ and $w(X_i)\leq 2w(X)/m+ M$ for all $i\in [m]$.
\end{proposition}
\begin{proof}
Let $X'_1,\dots, X'_m$ be disjoint subsets of $X$ such that 
$|X'_i|\leq \lceil 2n/m \rceil$
and $w(X'_i)\leq 2w(X)/m+M$ for all $i\in [m]$ 
and subject to these conditions such that $|\bigcup_{i=1}^{m}X'_i|$ is maximal. 
Observe that such sets exist as the collection of $m$ empty sets satisfy the first two conditions.

If $ \bigcup_{i=1}^{m}X'_i = X$, there is nothing to show. 
Otherwise, let $x\in X\sm \bigcup_{i=1}^{m}X'_i$. 
Since $\sum_{i=1}^{m}|X'_i| \leq |X|-1$, there are more than $m/2$ indices $i\in[m]$ such that $|X'_i|< \lceil2n/m\rceil$. 
Since $\sum_{i=1}^{m}w(X'_i) \leq w(X)$, there are at least $m/2$ indices $i\in[m]$ such that $w(X'_i)\leq 2w(X)/m$.
Thus there exists $i\in[m]$ such that $|X'_i\cup\{x\}|\leq \lceil 2n/m \rceil$ and $w(X'_i\cup\{x\})\leq 2w(X)/m+M$. 
However, then $X'_1,\dots, X'_{i-1}, X'_i\cup\{x\}, X'_{i+1},\dots, X'_m$ contradicts the choice of $X'_1,\dots, X'_m$.
\end{proof}

\subsection{Extending a partial embedding}
\label{sec:2.5}
In this section we provide three results that allow us to extend partial embeddings to proper embeddings.
The first one is only for forests and requires only very mild assumptions.
The second one is based on the Blow-up lemma and requires stronger assumptions but works for general graphs.
The third one provides a tool for embedding a collection of edges where one endpoint has already been embedded.

\begin{lemma}\label{lem: embed forests}
Suppose $n,k\in\N$.
Let $G$ be a graph on $n$ vertices 
such that $d_G(u,v) \geq k$ for every pair of vertices $u, v$.
Let $F$ be a forest on $k$ vertices. 
Let $I\subseteq V(F)$ be a $3$-independent set in $F$.
Then for any injection $\phi':I\to V(G)$, there exists a function $\phi$ consistent with $\phi'$ which embeds $F$ into $G$.
\end{lemma}
\begin{proof}
Let $v_1,\dots, v_k$ be an ordering of $V(F)$ 
such that for every $i$, the vertex $v_i$ has at most one neighbour in $\{v_1,\dots, v_{i-1}\}$,
say $v_{a(i)}$, if this neighbour exists. 
Note that for all $i\in [k]$, we have $d_{F,I}(v_i)\leq 1$, since $I$ is $3$-independent.
 We define $\phi(v):=\phi'(v)$ for every $v\in I$.  
We sequentially determine $\phi(v_i)$ for the remaining vertices according to the order $v_1,\dots, v_k$.
So assume that $v_i\notin I$ and we have already determined $\phi(v)$ for each $v\in \{v_1,\ldots,v_{i-1}\}\cup I$.

If $d_{F,I}(v_i)=0$, 
then we let $\phi(v_i)$ be some vertex in $V(G)\setminus \phi(\{v_1,\dots, v_{i-1}\}\cup I)$
if $v_{a(i)}$ does not exist and otherwise let $\phi(v_i)$ be some vertex in
$N_{G}(\phi(v_{a(i)}))\setminus \phi(\{v_1,\dots, v_{i-1}\}\cup I)$,
which exists as $\delta(G)\geq k$.

If $N_{F}(v_i)\cap I = \{u\}$ for some vertex $u$, 
then we let $\phi(v_i)$ be some vertex in $N:=N_{G}(\phi(u)) \setminus \phi(\{v_1,\dots, v_{i-1}\}\cup I)$ if $v_{a(i)}$ does not exist and otherwise let $\phi(v_i)$ be some vertex in $N':=N_{G}(\phi(v_{a(i)}),\phi(u)) \setminus \phi(\{v_1,\dots, v_{i-1}\}\cup I)$. 
Such a choice is always possible, because $\min\{ |N|, |N'|\} \geq k - (k-1) \geq 1$.
\end{proof}

\begin{lemma}\label{blow up with pre-embedding}
Suppose $n,\Delta\in \N$ and $1/n\ll\epsilon\ll p,1/\Delta \leq 1$.
Let $G$ be an $(\epsilon,p)$-quasi-random graph on $n$ vertices. 
Let $H$ be a graph on $n$ vertices with $\Delta(H)\leq \Delta$.
Let $I\sub V(H)$ be a $3$-independent set in $H$ with $|I|\leq \epsilon n$.
Then for any injection $\phi':I\to V(G)$, 
there exists a function $\phi$ consistent with $\phi'$ which embeds $H$ into $G$.
\end{lemma}
\begin{proof}
Let $U:=\phi'(I)$. 
As $I$ is a $3$-independent set,  $d_{H,I}(x)\leq 1$ for any $x\in V(H)$ and $d_{H,I}(x)=1$ for at most $\Delta|I|\leq \Delta \epsilon n$ vertices $x$.
Apply the Hajnal-Szemer\'edi theorem to obtain an equitable partition $(X_1,\dots, X_{\Delta+1})$ of $H-I$ such that $X_i$ is an independent set of $H$ for all $i\in [\Delta+1]$.
Consider a partition $(V_1,\dots, V_{\Delta+1})$ of $V(G)\setminus U$ such that $|X_i|=|V_i|$ for $i\in [\Delta+1]$ chosen uniformly at random.
Straightforward applications of Lemma~\ref{lem: chernoff} and Theorem~\ref{thm: almost quasirandom} show
that the following hold with probability at least $1/2$:
\begin{itemize}
\item[(GV1)] $G[V_i,V_j]$ is $(\epsilon^{1/7},p)$-super-regular for all distinct $i,j\in [\Delta+1]$, and
\item[(GV2)] $d_{G,V_i}(u) \geq pn/(2(\Delta+1))$ for all $u\in U$ and $i\in [\Delta+1]$.
\end{itemize}
In particular, there exists a partition $(V_1,\dots, V_{\Delta+1})$ of $V(G)\setminus U$ such that $|X_i|=|V_i|$ and both (GV1) and (GV2) hold. 

For each $x\in X_i$ for which $x\in N_{H}(y)$ for some $y\in I$, 
we let $A_x:=N_{G}(\phi'(y))\cap V_i$.
Thus $|A_x|\geq pn/(2(\Delta+1))$. Apply the Blow-up lemma (Theorem~\ref{blow up target sets}) to find an embedding $\phi''$ of $H-I$ into $G-U$ such that $\phi''(x) \in A_x \subseteq  N_{G}(\phi'(y))$ 
whenever $x\in N_H(y)$ for some $y\in I$. 
Then we obtain the desired embedding $\phi$ by defining $\phi:=\phi'\cup \phi''$.
\end{proof}

The next result will be used to embed edges of trees where one endpoint has already been embedded. The role of the graph $H$ is to help us avoid collisions between edges belonging to the same tree.
\begin{proposition}\label{prop: sparse edge embedding}
Suppose $\Delta, m, s \in \N$.
Let $G$ be a graph, let $A\subseteq V(G)$, and 
let $u_1,\dots, u_m$ be a sequence of (not necessarily distinct) vertices of $G$. Let $W_1,\dots, W_m$ be sets of vertices and let $H$ be a graph on $[m]$ satisfying the following:
\begin{itemize}
\item[(i)] $d_{G,A}(u_i)-|W_i|\geq 3\Delta + m/s+s$ for all $i\in [m]$,
\item[(ii)] $|\{i\in [m]: v=u_i\}| \leq \Delta$ for any $v\in V(G)$, and
\item[(iii)] $\Delta(H)\leq \Delta$.
\end{itemize}
Then we can choose distinct edges $u_1v_1,\dots, u_mv_m$ in $G$ such that $|\{i\in [m]: v=v_i\}|\leq s$ for all $v\in V(G)$ 
and $v_i\in A\setminus W_i$ for all $i\in [m]$ and $v_i\notin \{u_j, v_{j}\}$ whenever $ij\in E(H)$.
\end{proposition}
\begin{proof}
We sequentially choose $v_i$. 
Assume that we have chosen $v_1,\dots v_{i-1}$ such that 
\begin{itemize}
\item[(a)$_{i-1}$] $v_j\in (N_{G}(u_j)\cap A)\setminus W_j$ and $u_jv_j\neq u_{j'}v_{j'}$ for any $j\neq j' \leq i-1$,
\item[(b)$_{i-1}$] $|\{j\leq i-1: v=v_j\}|\leq s$ for any $v \in V(G)$, and
\item[(c)$_{i-1}$] $v_j\notin\{ u_{j'}, v_{j'}\}$ whenever $jj'\in E(H)$ and $j\neq j'\leq i-1$.
\end{itemize}
Consider the set $(N_{G}(u_i)\cap A)\setminus W_i$, which contains at least $3\Delta+ m/s + s$ vertices by (i). 
Among those are 
\begin{itemize}
	\item at most $\Delta-1$ vertices $w$ such that $w=v_j, u_i=u_j $ for some $j\leq i-1$, by (ii),
	\item at most $s$ vertices $w$ such that $w=u_j,u_i=v_j$ for some $j\leq i-1$, by (b)$_{i-1}$,
	\item at most $m/s$ vertices $w$ such that $|\{ j\leq i-1: w=v_j \}|= s$, and
	\item at most $2\Delta$ vertices $w$ such that $w\in\{u_j,v_{j}\}$ for $j\in N_{H}(i)$, by (iii). 
\end{itemize}
Therefore, by (i), we can choose a vertex $v_i$ satisfying (a)$_{i}$, (b)$_{i}$ and (c)$_{i}$. 
By repeating this selection process, we obtain distinct edges $u_1v_1,\dots, u_mv_m$  satisfying (a)$_{m}$, (b)$_{m}$ and (c)$_{m}$. It is clear that they form a desired collection of edges.
\end{proof}

\section{Approximate cycle decomposition}
\label{sec: approx cycle decomp}
The aim of this section is to show that every $(\beta,\alpha)$-dense almost regular graph on an odd number of vertices
has an approximate decomposition into Hamilton cycles and a few very long odd cycles (see Lemma~\ref{lem: decomp Eulerian}). 
In Section~\ref{sec: decomp to cycle}, we will apply this result to the reduced graph obtained from Szemer\'{e}di's regularity lemma.

The main idea for the proof is based on that of Theorem~4.1 in \cite{GKO16}.
We will use results from \cite{KO13},\cite{KO14} and \cite{KOT10} which imply that robustly expanding graphs enjoy very strong Hamiltonicity properties. To state these, we first need to introduce the concept of robust expansion. For $0<\nu\leq \tau<1$, a graph $G$ on $n$ vertices and $S\subseteq V(G)$, 
we define the {\em $\nu$-robust neighbourhood $RN_{\nu,G}(S)$} of $S$ 
as the set of all those vertices of $G$ which have at least $\nu n$ neighbours in $S$. 
We call $G$ a {\em robust ($\nu,\tau$)-expander} if
$$ |RN_{\nu,G}(S)| \geq |S|+\nu n \text{ for all } S\subseteq V(G) \text{ with } \tau n \leq |S| \leq (1-\tau) n. $$

We will use the following two results about Hamilton cycles in robust expanders.

\begin{theorem}[\cite{KOT10}]\label{thm: expander ham}
Suppose $n\in \N$ and $1/n \ll \nu \leq \tau \ll \alpha \leq 1$. 
Then any robust $(\nu,\tau)$-expander $G$ on $n$ vertices with $\delta(G)\geq \alpha n$ contains a Hamilton cycle.
\end{theorem}
The following result guarantees a Hamilton decomposition in an even-regular robust expander. 
It is derived in \cite[Theorem~1.2]{KO14} from a digraph version in \cite{KO13}.
\begin{theorem}\label{thm: decomp regular}
Suppose $n\in \N$ and $1/n\ll \nu\leq \tau \ll \alpha\leq 1$.
If $G$ is an $\alpha n$-regular robust $(\nu,\tau)$-expander on $n$ vertices such that $\alpha n$ is even,
then $G$ can be decomposed into Hamilton cycles.
\end{theorem}

The next observation shows that in a $(\beta,\alpha)$-dense graph
every not too small set of vertices induces a robust expander.

\begin{proposition}\label{prop: dense induced sub}
Suppose $n\in \N$ and $1/n \ll \beta \ll \alpha, \gamma\leq 1$. 
Let $G$ be a $(\beta,\alpha)$-dense graph on $n$ vertices. 
If $U\sub V(G)$ with $|U|\geq \gamma n$,
then $G[U]$ is a robust $(\alpha\beta,2\beta/\gamma)$-expander.
\end{proposition}
\begin{proof}
Consider a set $S \subseteq U$ with $2\beta n \leq (2\beta/\gamma) |U| \leq |S| \leq (1-2\beta/\gamma) |U| \leq |U|-2\beta n$. 
Let $T:= \{ v \in U: d_{G,S}(v) < \alpha\beta n\}$.
This implies that
$$\den_G(S,T) \leq \frac{\alpha\beta n|T|}{|S||T|} \leq \frac{\alpha \beta n}{2 \beta n} \leq \frac{\alpha}{2}.$$
As $G$ is $(\beta,\alpha)$-dense, we conclude that $|T|< \beta n$.
Thus 
$$|RN_{\alpha\beta,G[U]}(S)|
=|U\setminus T| 
\geq |U| - \beta n \geq 
|S|+\beta n
\geq |S|+ \alpha \beta |U|.$$
Therefore, $G[U]$ is a robust $(\alpha\beta, 2\beta/\gamma)$-expander.
\end{proof}

Next, we show that every $(\beta,\alpha)$-dense graph can be ``approximated'' by an Eulerian graph.

\begin{proposition}\label{prop: making Eulerian}
Suppose $n\in \N$ and $1/n \ll \beta \ll \alpha \leq 1$.
Let $G$ be a $(\beta,\alpha)$-dense graph on $n$ vertices with $\delta(G)\geq \alpha n$.
Then $G$ contains an Eulerian subgraph $G'$ such that $\Delta(G-E(G'))\leq 3$.
In particular, $e(G-E(G'))\leq 3n/2$.
\end{proposition}
\begin{proof}
Let $M$ be a maximal matching in the subgraph of $G$ induced by the vertices of odd degree.
Let $X$ be the set of all vertices of odd degree which are not covered by $M$.
Thus $X$ is an independent set (of even size)
and hence $|X|< \beta n$ as $G$ is $(\beta,\alpha)$-dense.
We write $X=\{x_1,y_1,\ldots,x_k,y_k\}$. Since $\delta(G)\geq \alpha n$ and $G$ is $(\beta,\alpha)$-dense, 
for every $i\in [k]$, there are at least $\alpha^3 n^2/3$ edges joining $N(x_i)$ and $N(y_i)$ which do not lie in $M$.\COMMENT{$\alpha |N(x_i)||N(y_i)|/2 - |M| \geq \alpha^3 n^2/2 - n \geq \alpha^3 n^2/3$.} Thus there are at least $\alpha^3n/6$ internally vertex-disjoint paths of length $3$ between $x_i$ and $y_i$ which are edge-disjoint from $M$.\COMMENT{None of these paths has an inner vertex in $X$ since $X$ is independent set.}
Since $\alpha^3n/6 \geq \beta n \geq 2k$, 
we can choose vertex-disjoint paths $P_1,\dots, P_k$ in $G-M$ such that $P_i$ is a path of length $3$ between $x_i$ and $y_i$. 
Then $G':=G-(M\cup \bigcup_{i=1}^{k} E(P_i))$ is an Eulerian subgraph of $G$ with  $\Delta(G-E(G'))\leq 3$.
\end{proof}

\begin{lemma}\label{lem: decomp Eulerian}
Suppose $n,r\in \N$ and $1/n \ll \beta \ll \alpha , 1/r< 1$ and $n$ is odd.\COMMENT{since $1/r<1$, we have $r\geq 2$.}
Let $G$ be a $(\beta,\alpha)$-dense Eulerian graph on $n$ vertices with $\delta(G)\geq \alpha n$ and $\Delta(G)-\delta(G) \leq \beta n$.
Then $G$ can be decomposed into Hamilton cycles and at most $2r\beta n$ odd cycles of length at least $(1- \frac{1}{r-1})n$.
\end{lemma}
To prove Lemma~\ref{lem: decomp Eulerian}, we will sequentially remove the edges of suitable odd cycles until the remaining graph is regular. We can then obtain a Hamilton decomposition of the remaining graph via Theorem~\ref{thm: decomp regular}. 
The argument builds on ideas from \cite{GKO16}.
\begin{proof}
Given any $V'\subseteq V(G)$ such that $|V'|$ is even, for $k\in \mathbb{N}$ we say $(V_1,\dots, V_k)$ is a {\em fair $k$-partition} of $V'$ with respect to $G$ if the following hold:
\begin{enumerate}[label=(\alph*)]
\item $|V_i|$ is even for every $i\in [k]$,
\item $||V_i|-|V_j||\leq 2$, and
\item  $d_{G,V_i}(v) \geq d_{G,V'}(v)/k-n^{2/3}$ for all $v\in V(G)$ and $i\in [k]$.
\end{enumerate}

First, note that any set $V'\subseteq V(G)$ such that $|V'|$ is even has a fair $k$-partition. 
(To see this, choose a partition satisfying (a) and (b) uniformly at random and apply Lemma~\ref{lem: chernoff} to show that (c) holds with probability at least $1/2$.)
Now we sequentially construct nested collections $\cC_0:=\es\subseteq \cC_1\subseteq  \dots \subseteq \cC_s$ of edge-disjoint odd cycles and 
spanning subgraphs $G_0:=G\supseteq G_1 \supseteq\dots\supseteq G_s$ of $G$ such that $G_s$ is regular and for all $0\leq i<s$ the following hold:
\begin{enumerate}[label=(\roman*)$_i$]
	\item\label{item:euler1} $E(G_{i})-E(G_{i+1}) = \bigcup_{C\in \cC_{i+1}\setminus \cC_i} E(C)$ and $|\cC_{i+1}\setminus \cC_i|\leq 2r$,
	\item\label{item:euler2} $d_{G_i}(v)- d_{G_{i+1}}(v) \leq 4r-2$ for every $v\in V(G)$,
	\item\label{item:euler3} $(\Delta(G_i)-\delta(G_i))-(\Delta(G_{i+1})-\delta(G_{i+1}))= 2$, and
	\item\label{item:euler4} $|C|\geq (1-\frac{1}{r-1})n$ for every $C\in \cC_{i+1}$.
\end{enumerate}
Suppose that for some $j\geq 0$ 
we have already defined collections $\cC_0\subseteq \cC_1\subseteq \dots \subseteq \cC_j$ of edge-disjoint odd cycles and a sequence $G_0\supseteq G_1\supseteq \dots\supseteq G_j$ of graphs satisfying \ref{item:euler1}--\ref{item:euler4} for all $0\leq i<j$. 
If $G_j$ is regular, then we may set $s:=j$, so assume that $G_j$ is not regular. 
Since $\Delta(G)-\delta(G) \leq \beta n$, 
\ref{item:euler3} for $i<j$ implies that $j\leq \beta n/2$. 
This together with \ref{item:euler2} for $i<j$ implies that for every $v\in V(G)$
\begin{align}\label{eq: Gj deg}
d_{G}(v) - d_{G_j}(v) \leq 2r \beta n.
\end{align}
By Proposition~\ref{prop: dense deletion}, 
the graph $G_j$ is $(2r\beta^{1/2},\alpha/2)$-dense. 
Thus Proposition~\ref{prop: dense induced sub} implies that
\begin{equation}\label{eq: Gj expand}
\begin{minipage}[c]{0.8\textwidth}\em
for any set $V'\subseteq V(G)$ with $|V'|\geq \alpha n/5$ and any graph $G'\subseteq G_j[V']$  with  
$\Delta(G_j[V'] - E(G'))\leq 4r$, the graph $G'[V']$ is a robust $(\alpha r \beta^{1/2}/2,\beta^{1/3})$-expander.
\end{minipage}\ignorespacesafterend 
\end{equation} 

Let $M_j$ be the set of all vertices of maximum degree in $G_j$ and let $Z_j:=V(G)\sm M_j$. 
Note that $d_{G_j}(u) - d_{G_j}(v)\geq 2$ for every pair of vertices $u\in M_j, v\in Z_j$, 
since $G$ (and thus $G_j$) is Eulerian. Pick $x\in M_j$ and let 
$$Z':= \left\{\begin{array}{ll} 
Z_j & \text{ if } |M_j| \text{ is odd,}\\
Z_j\cup \{x\} &\text{ if } |M_j| \text{ is even.}
 \end{array} \right.$$

Note that $|Z'|$ is even, because $n$ is odd. 
Let $(V_1,\dots, V_r)$ be a fair $r$-partition of $Z'$ with respect to $G$.  Next we show how to define $G_{j+1}$ and $\cC_{j+1}$.
We will choose the cycles in $\cC_{j+1}\sm \cC_j$ one by one. 
Assume we have already defined edge-disjoint odd cycles $C_1,\dots, C_{t-1}$  in $G_j$ for some $t\in [r]$.
Consider the graph
$$G_{j,t}:= G_j[ V(G)\setminus V_{t}] - \bigcup_{i=1}^{t-1} E(C_i).$$ 
By \eqref{eq: Gj deg}, part (c) in the definition of an $r$-fair partition, and $\delta(G)\geq \alpha n$, we conclude that $d_{G_j[V(G)\setminus V_{t}]}(v) \geq \alpha n/3$.
Hence $\delta(G_{j,t})\geq \alpha n/3-2r$. 
This together with~\eqref{eq: Gj expand} implies that we can use Theorem~\ref{thm: expander ham} to obtain a Hamilton cycle $C_{t}$ in $G_{j,t}$.
We repeat this procedure until we have defined $C_1,\dots, C_r$. Note that $|C_t|\geq (1-1/r)n-2 \geq (1- 1/(r-1))n$ for each $t\in [r]$.

If $|M_j|$ is odd, then we let $\cC_{j+1}:= \cC_j \cup \{C_1,\dots, C_{r}\}$, and $G_{j+1}:= G_j - \bigcup_{i=1}^{r} E(C_i)$. 
Then we have for every $v\in V(G)$
$$d_{G_j}(v)-d_{G_{j+1}}(v) = \left\{ \begin{array}{ll}
2r-2 & \text{ if } v\notin M_j, \\
2r & \text{ if } v\in M_j,
  \end{array} \right.$$
so each of (i)$_j$--(iv)$_j$ holds. 

If $|M_j|$ is even, we will define additional cycles $C_{r+1},\dots, C_{2r}$ as follows. Let $(U_1,\dots, U_r)$ be a fair $r$-partition of $V(G)\setminus \{x\}$.
Assume for some $r+1\leq t\leq 2r$ we have already defined additional edge-disjoint odd cycles $C_{r+1},\dots, C_{t-1}$ in $G_j-\bigcup_{i=1}^{r} E(C_i)$.
Consider the graph
$$G'_{j,t}:=G_j[ V(G)\setminus U_{t-r}] - \bigcup_{i=1}^{t-1} E(C_i).$$ 
By (c), we obtain $\delta(G'_{j,t})\geq \alpha n/3-4r$.
As before we can use Theorem~\ref{thm: expander ham} to obtain a Hamilton cycle $C_{t}$ in $G'_{j,t}$. 
We repeat this process until we defined $C_1,\dots, C_{2r}$. 
Let $\cC_{j+1}:= \cC_j \cup \{C_1,\dots, C_{2r}\}$, and $G_{j+1}:= G_j - \bigcup_{i=1}^{2r} E(C_i)$. 
Thus
$$d_{G_j}(v)-d_{G_{j+1}}(v) = \left\{ \begin{array}{ll}
4r-4 & \text{ if } v\notin M_j, \\
4r-2 & \text{ if } v\in M_j,
  \end{array} \right.$$
so each of (i)$_j$--(iv)$_j$ holds. 

Hence, by repeating this procedure, we can obtain a regular graph $G_s$ 
and a collection $\cC_s$ of edge-disjoint odd cycles of length at least $(1-\frac{1}{r-1})n$. 
Also $|\cC_s|\leq 2rs \leq 2r\beta n$, since $s \leq \beta n/2 +1$. 
Observe that $G_s$ is Eulerian as $G$ is Eulerian.
By \eqref{eq: Gj expand} and Theorem~\ref{thm: decomp regular}, $G_s$ can be decomposed into a collection $\cC'$ of Hamilton cycles. 
Now $\cC := \cC_s\cup \cC'$ is the desired decomposition of $G$.
\end{proof}

\section{Decomposing a graph into cycle blow-ups}
\label{sec: decomp to cycle}
The main result in the current section (Lemma~\ref{lem: decomp blow ups}) guarantees a near-optimal packing of blow-ups of long cycles into a $(\beta,\alpha)$-dense graph which is close to being regular. Results in Section~\ref{sec: trees} will allow us to obtain a near-optimal packing of trees into blow-ups of long cycles. So together these results guarantee a near-optimal packing of trees into a $(\beta,\alpha)$-dense graph which is close to being regular. 
The fact in \ref{item:R1} that all cycle blow-ups are only almost spanning will be important when we apply Lemma~\ref{lem: decomp blow ups} in the iteration lemma in Section~\ref{sec:iteration}.
Indeed, whenever we embed a tree $T$, our aim is to embed all or almost all vertices of $T$ into a given cycle blow-up $C$.
But for some of the trees $T$, 
we will need to embed a small part of $T$ outside $C$, as this enables us to cover given sets of ``exceptional'' edges. 

\begin{lemma}\label{lem: decomp blow ups}
Suppose $D,M,M',n,t\in \N$ with $1/n \ll 1/M\ll 1/M'\ll\epsilon\ll 1/t \ll \beta \ll \alpha \leq d,1/D\leq 1$ and let $D$ be even with $D\geq 6$. 
Let $\cI:= \binom{[D]}{D-3}$.
Suppose that $G$ is a $(\beta,\alpha)$-dense graph on $n$ vertices
with equitable partition $(U_1,\dots, U_D)$ such that $d_{G,U_i}(v) = (d \pm \beta)D^{-1}n$ for all $v\in V(G)$ and all $i\in [D]$. 
Then there exist a set $V_0\sub V(G)$ with $|V_0|\leq \epsilon n$, an integer $\Gamma$ with $M' \leq \Gamma \leq M$, and 
for each $S\in \cI$ a set $\cC_{S}$ such that the following properties hold (writing $\cC:=\bigcup_{S\in \cI}\cC_S$):
\begin{enumerate}[label=(R\arabic*)]
	\item\label{item:R1} for every $S\in \cI$, 
	the set $\cC_{S} $ is a set of $(2\epsilon,1/t)$-super-regular cycle $(1- \epsilon)\frac{n}{\Gamma}$-blow-ups
	such that $C\sub G$ for each $C\in \cC_S$,
	the length of each $C\in \cC_S$ is odd and at least $(1- {7}/{(2D)})\Gamma$
	and such that $V(C)\sub (\bigcup_{i\in S}U_i)\sm V_0$ for each $C\in \cC_S$,
	\item\label{item:R2a} the cycle blow-ups in $\cC$ are pairwise edge-disjoint,
	\item\label{item:R2} $e(\cC_S) = (1\pm \beta^{1/2})\binom{D}{3}^{-1} e(G)$,
	\item\label{item:R3} 	$d_G(v)-\sum_{C\in\cC}d_{C}(v)\leq 3n/t^{1/2}$ for every $v\in V(G)\sm V_0$. 
	In particular, $ e(\cC)\geq e(G) - 2n^2/t^{1/2}$.
\end{enumerate}
\end{lemma}
To prove Lemma~\ref{lem: decomp blow ups}, we apply Szemer\'{e}di's regularity lemma to obtain a reduced multigraph $R^m$. We then apply Lemma~\ref{lem: decomp Eulerian} to obtain an approximate decomposition of $R^m$ into long odd cycles. This then translates into the existence of the desired super-regular cycle blow-ups in $G$. The reliance on Szemer\'{e}di's regularity lemma has the drawback that we have a small exceptional set $V_0$ which is not part of the cycle blow-ups. On the other hand the packing is very efficient: by \ref{item:R3} the leftover density can be chosen to be much smaller than the input parameters $\alpha$, $\beta$ of $G$. 
\begin{proof}
In this proof,
given a graph $H$,
we set $d_H(v):=0$ whenever $v\notin V(H)$.\COMMENT{Do we use it somewhere else?}
We apply the regularity lemma (Lemma~\ref{lem: rl}) to the graph $G$ with $\epsilon^2$ playing the role of $\epsilon$
to obtain a partition $(V_0',V_1',\ldots,V_{\Gamma}')$ of $V(G)$ such that,
writing $d_{ij}:=\den_G(V_i',V_j')$,
\begin{enumerate}[label=(\roman*)]
	\item $M'\leq \Gamma \leq M$,
	\item $|V_0'|\leq \epsilon^2 n$,
	\item $|V_1'|=\ldots=|V_{\Gamma}'|=:n'$,
	\item for each $i\in [\Gamma]$, the graph $G[V'_i,V'_j]$ is $(\epsilon^2,d_{ij})$-regular except for at most $\epsilon^2 \Gamma$ indices $j \in [\Gamma]$, and		
	\item for each $i\in [\Gamma]$, there exists a unique $j \in [D]$ such that $V_i'\subseteq U_j$
	and, writing $U_j^R:=\{i:V_i' \sub U_j\}$, we have $|U_1^R|=|U_{j}^R| $ for all $j\in [D]$ and
	$|U_1^R|$ is odd.
\end{enumerate}
Moreover, 
let $R$ be the graph with vertex set $[\Gamma]$ where two vertices $i,j$ are joined by an edge 
if $d_{ij}\geq 1/t^{1/2}$ and $G[V_i',V_j']$ is $(\epsilon^2,d_{ij})$-regular.
Construct a multigraph $R^m$ from $R$ by
replacing each edge $ij$ in $R$ by $\lfloor d_{ij}t \rfloor$ edges.
Observe that $(U_1^R,\ldots,U_D^R)$ is an equitable partition  of $V(R^m)$.
Then
\begin{enumerate}
	\item [(vi)] $d_{R^m,U_j^R}(i)= (d\pm 3\beta)t|U_j^R|$ for all $i\in [\Gamma]$ and $j\in [D]$ and
	\item [(vii)] $R^m$ is $(2\beta,t\alpha/2)$-dense.
\end{enumerate}

For every $ij \in E(R)$, we use Proposition~\ref{prop: part  eps reg} to find $\lfloor d_{ij}t \rfloor$ disjoint edge sets $E_{ij}^1,\ldots,E_{ij}^{\lfloor d_{ij}t \rfloor}$ of $G[V_i',V_j']$
such that $E_{ij}^k$ induces a $(2\epsilon^2, 1/t )$-regular graph
for every $k\in [\lfloor d_{ij}t \rfloor]$.

Let $E'_{ij}:= E(G[V_i',V_j'])\sm \bigcup_{k=1}^{\lfloor d_{ij}t \rfloor} E_{ij}^k$ for any $ij \in E(R)$. 
Note that, also by Proposition~\ref{prop: part  eps reg},
for every edge $ij\in E(R)$
\begin{align}\label{eq: E' deg}
	 d_{E'_{ij}}(v) \leq \frac{3n'}{2t^{1/2}}.
\end{align}
\COMMENT{
Get $d_{E'_{ij}}(v) \leq (1- \lfloor d_{ij}t \rfloor/td_{ij})n' + n'^{2/3}$. But $\lfloor d_{ij}t\rfloor/(td_{ij}) \geq  1- 1/(td_{ij}) \geq 1- 1/t^{1/2}$.
}
Let $\phi:E(R^m)\to \{E_{ij}^k: ij\in E(R), k\in \lfloor d_{ij}t \rfloor\}$ be a bijection
such that for every $e\in E(R^m)$ which joins $i$ and $j$, we have $\phi(e)= E_{ij}^k$ for some $k\in [\lfloor d_{ij}t \rfloor]$.

The proof strategy is now as follows.
We first use Lemma~\ref{lem: D partition} to decompose $R^m$ into a collection $\{R_S^m\}_{S\in \cI}$ of simple graphs such that 
$V(R^m_S)=\bigcup_{i\in S}U_i^R$ for each $S\in \cI$.
For each $S\in \cI$,
we then apply Lemma~\ref{lem: decomp Eulerian} to find a set $\cD_S$ of edge-disjoint long odd cycles which together cover almost all edges of $R^m_S$.
Via the map $\phi$ each cycle $C\in \cD_S$ corresponds to a $(2\epsilon^2,1/t)$-regular cycle blow-up $\phi(C)$.
The set $\cC_S$ is then obtained from $\{\phi(C):C\in \cD_S\}$ by modifying each $\phi(C)$ slightly to obtain a cycle blow-up which is super-regular
and in which all the clusters have equal size.
In this last step we have to be careful
since we need to ensure that \ref{item:R3} holds.

By Lemma~\ref{lem: D partition} and (vi) as well as (vii), 
we can decompose the multigraph $R^m$ into a collection of graphs $\{R^m_S\}_{S\in \cI}$,
such that for each $S\in \cI$ the following holds:
\begin{itemize}
\item[(a)] $R^m_S$ is  $(6\beta, D^{-3}t\alpha/2)$-dense with $V(R^m_S)=\bigcup_{i\in S} U_i^R$,
\item[(b)] $d_{R^m_S}(j) =  (d\pm 6\beta)\binom{D-1}{3}^{-1}t\Gamma= (d\pm 6\beta)\binom{D-1}{3}^{-1}tD |R_S^m|/(D-3)$
for every $j\in V(R^m_S)$, and
\item[(c)] $e(R^m_S)=(1\pm 3\beta^{2/3})\binom{D}{3}^{-1}e(R^m)$.
\end{itemize}

We next use Proposition~\ref{prop: decomp multigraph} to decompose $R^m_S$ into $t$ simple spanning graphs $R_{S,1},\ldots,R_{S,t}$
such that for all $S\in \cI$ and $i\in[t]$ the following hold:
\begin{itemize}
\item[(a$'$)] $R_{S,i}$ is $(6\beta,D^{-3}\alpha/4)$-dense, and
\item[(b$'$)] $d_{R_{S,i}}(j)=(d\pm \beta^{2/3})\binom{D-1}{3}^{-1}D|R_{S,i}|/(D-3)$ for every $j\in V(R_{S,i})$.
\end{itemize}

Using Proposition~\ref{prop: making Eulerian} for each $S\in \cI$,
we obtain Eulerian graphs $R_{S,1}',\ldots,R_{S,t}'$ such that $R'_{S,i}\subseteq R_{S,i}$ for all $i\in [t]$ and
\begin{align}\label{eq: RR deg diff}
\Delta(R_{S,i}- E(R'_{S,i})) \leq 3.
\end{align}
Note that 
since $D$ is even and $|U_1^R|=\Gamma/D$ is odd (by (v)),
it follows that
$|R'_{S,i}|=(D-3)\Gamma/D$ is odd for all $S\in \cI$ and $i\in [t]$.
Observe that $\delta(R_{S,i}')\geq D^{-3}\alpha|R_{S,i}|/5$ for every $i\in[t]$ by (b$'$), and that $R'_{S,i}$ is still $(\beta^{3/4}, D^{-3}\alpha/5)$-dense.
Thus for every $S\in \cI$, we 
can apply Lemma~\ref{lem: decomp Eulerian} with $2D+1$, $\beta^{3/5}$ and $D^{-3}\alpha/5$ playing the roles of $r$, $\beta$ and $\alpha$
to decompose $R_{S,1}',\ldots,R_{S,t}'$ into a set $\cD_S$ of odd cycles
of length at least $(1- 1/(2D))(D-3)|U_1^R|\geq (1- 7/(2D))\Gamma$.
Let $\cD:= \bigcup_{S\in \cI} \cD_S$.
Thus all the cycles in $\cD$ are pairwise edge-disjoint.

Recall that for every $e\in R^m$ joining $i$ and $j$, we have $\phi(e)= E_{ij}^k$ for some $k\in [\lfloor d_{ij} t\rfloor]$. 
For every $C \in \cD$, let $\phi(C) := \bigcup_{e \in E(C)}\phi (e)$ be the cycle blow-up in $G$ induced by the edges in $\bigcup_{e \in E(C)}\phi (e)$. 
Then $\phi(C)$ is a $(2\epsilon^2, 1/t)$-regular cycle $n'$-blow-up. Moreover, for all $S\in \cI$ and $v\in V(G)$,
we have
\begin{align}\label{eq:degDS}
	\sum_{C\in \cD_S}d_{\phi(C)}(v)
	\stackrel{(\ref{eq: RR deg diff})}{=}
	\sum_{e\in R^m_S}d_{\phi(e)}(v)\pm 3tn'.
\end{align}

In the remainder of the proof we show how to modify the cycle blow-ups $\phi(C)$ (for $C\in \cD$) so that they satisfy \ref{item:R3}. First observe that 
\begin{align}\label{eq: cD size}
 |\cD| \leq 2t\Gamma.
\end{align} 
Indeed, every $C\in \cD$ has length at least $(1-7/(2D))\Gamma \geq 5\Gamma/12$ and 
thus $\phi(C)$ contains at least $n^2/(3t\Gamma)$ edges
while $e(G)\leq n^2/2$.

By Proposition~\ref{prop: making super reg},
for each $C\in \cD$, 
the graph $\phi(C)$ contains a $(12\epsilon^2, 1/t$)-super-regular cycle $(1-8\epsilon^2)n'$-blow-up $\phi^s(C)$.
For every $i\in[\Gamma]$, let 
$$W_i:=\{  u\in V'_i: |\{C \in \cD: u\in V(\phi(C)) \sm V(\phi^s(C)) \}|\geq \epsilon^{1/2} \Gamma  \}.$$
Thus the sets $W_i$ contain all those vertices in $V'_i$ that get deleted too often when making the cycle blow-ups super-regular. 
Then
\begin{align*}
	|W_i|\epsilon^{1/2}\Gamma \leq 8\epsilon^2 n'|\cD| \stackrel{\eqref{eq: cD size}}{\leq} 16 \epsilon^2t\Gamma n'
\end{align*}
and hence $|W_i|\leq \epsilon n/(2\Gamma)$.
Let $V_0:=W_1\cup \ldots \cup W_\Gamma \cup V_0'$ and let $V_i:=V_i'\sm V_0$.
Thus $|V_0|\leq \epsilon n$, 
and by construction for each $v\in V(G)\sm V_0$, we obtain
\begin{align}\label{eq: not disregard too many}
|\{C \in \cD : v\in \phi(C)  \}| - |\{C\in \cD: v\in V(\phi^s(C))\setminus V_0\}| \leq \epsilon^{1/2}\Gamma.
\end{align}
Also for all $C\in \cD$ and $v\in V(\phi^s(C))\sm V_0$, we have
\begin{align}\label{eq:deg phis}
	d_{\phi(C)}(v)= d_{\phi^s(C)-V_0}(v)\pm \frac{2\epsilon n}{\Gamma}.
\end{align}
Write $\cD=\{C_1,\dots, C_m\}$.
For each $r\in [m]$, let $C'_{r}:=\phi^s(C_r)- V_0$. 
Note that for each $j\in V(C_{r})$, we obtain
\begin{align}\label{eq: C'-V0 cluster size}
|V(C'_{r})\cap V_j| 
\geq |V(\phi^s(C_{r}))\cap V'_j| -|W_j| 
\geq (1- 8\epsilon^2) n' - \frac{\epsilon n}{2\Gamma} 
> (1-\epsilon)\frac{n}{\Gamma}.
\end{align}
We claim that we can find cycle $(1-\epsilon)\frac{n}{\Gamma}$-blow-ups $C_1'',\ldots, C_{m}''$ such that 
for all ${r}\in [m]$ the following hold:
\begin{itemize}
\item[($\alpha$1)$_r$] $C_{r}''$ is an induced subgraph of $C_{r}'$,
\item[($\alpha$2)$_r$] $C_{r}''$ is a $(2\epsilon,1/t)$-super-regular cycle blow-up, and
\item[($\alpha$3)$_r$] $|\{r'\in [r]: u\in V(C'_{r'})\sm V(C_{r'}'')\}|\leq 10\epsilon t \Gamma$ for all $u\in V(G)\sm V_0$.
\end{itemize}


Suppose that for some $r\in [m]$ we have already constructed $C_1'',\dots, C_{r-1}''$ such that
($\alpha$1)$_{r'}$--($\alpha$3)$_{r'}$ hold for all $r'<r$.
Next we construct $C_{r}''$. 
For each $i\in V(C_{r})$, consider the set $B^r_{i}\subseteq V_{i}$ defined by 
$$B^{r}_{i}:=\{u \in V_{i}: |\{ r'< r :u\in V(C'_{r'})\sm V(C_{r'}'') \}|\geq 8\epsilon t\Gamma\}.$$
Since 
$$
|B_{i}^{r}|\cdot 8 \epsilon t \Gamma \leq 
\sum_{r'=1}^{r-1} |(V(C'_{r'})\setminus V(C''_{r'})) \cap V_{i}| 
{\leq} 
\sum_{r'=1}^{r-1} \epsilon n/\Gamma 
\stackrel{\eqref{eq: cD size}}{\leq} 2t\epsilon n,$$
we have $|B^{r}_{i}|\leq n/(4\Gamma)$.
Thus we can choose any set $B'_{i}$ of size $|V(C_r')\cap V_i| - (1-\epsilon)n/\Gamma$ in $(V(C_r')\cap V_i) \sm  B^r_{i}$ and 
define 
$$C_r'':= C'_r - \bigcup_{i \in V(C_r') } B'_{i}.$$
Then $C_r''$ satisfies ($\alpha$1)$_r$--($\alpha$3)$_r$. 
(In order to check ($\alpha$2)$_r$, use that $|B'_i|\leq \epsilon n/\Gamma$ and $|W_i|\leq \epsilon n/ (2\Gamma)$.)
By repeating this procedure, we obtain a collection $\{C_1'',\dots, C_m''\}$ satisfying ($\alpha$1)$_{r}$--($\alpha$3)$_{r}$ for all $r\in [m]$.

Let $\cC:=\{ C_1'',\dots, C_m''\}$ and let $\cC_S:=\{C_r'': C_r\in \cD_S\}$. 
We now verify that $\cC$ and $\cC_S$ with $S\in \cI$ satisfy \ref{item:R1}--\ref{item:R3}.
Note that $C_1'',\dots, C_m''$ are edge-disjoint subgraphs of $G$ and
that for each $r\in [m]$ and every $v\in V(G)$,
we trivially have
$d_{C_r''}(v) \leq d_{C_r'}(v) \leq  d_{\phi(C_r)}(v) \leq   3 n/\Gamma$. 
Thus, for fixed $S\in \cI$ and $v\in V(G)\sm V_0$, we obtain
\begin{eqnarray}\nonumber
\sum_{C_r''\in \cC_S}d_{C_r''}(v)  
&\stackrel{(\alpha1)_r,(\alpha3)_r}{=}& \sum_{C_r''\in \cC_S} \left(d_{C'_r}(v) \pm \frac{\epsilon n}{\Gamma}\right)\pm 10\epsilon t\Gamma \cdot \frac{3 n}{\Gamma}\\
&\stackrel{\eqref{eq: cD size},\eqref{eq: not disregard too many},\eqref{eq:deg phis}}{=}& \sum_{C_r\in \cD_S}\left(d_{\phi(C_r)}(v) \pm \frac{2\epsilon n}{\Gamma}\right)
\pm \left(\epsilon^{1/2} \Gamma \cdot \frac{3 n}{\Gamma} + 2\epsilon tn
+ 30\epsilon tn\right)  \nonumber \\
&\stackrel{\eqref{eq:degDS},\eqref{eq: cD size}}{=} &  \sum_{e\in R_{S}^m}d_{\phi(e)}(v) 
\pm \left(3 tn'
+  \frac{7\epsilon^{1/2}n}{2}\right)  \nonumber\\
& = & \sum_{e\in R_{S}^m}d_{\phi(e)}(v) 
\pm 4\epsilon^{1/2} n.\label{eq: cC_S deg}
\end{eqnarray}
Hence for each $v\in V(G)\setminus V_0$, we obtain
\begin{eqnarray*}
\sum_{C''_r\in \cC}d_{C''_r}(v) &= &  
\sum_{S\in \cI} \sum_{C''_r\in \cC_S}d_{C''_r}(v) 
\stackrel{\eqref{eq: cC_S deg}}{=}   \sum_{S\in \cI}\left(\sum_{e\in R_{S}^m}d_{\phi(e)}(v)\pm 4\epsilon^{1/2} n \right) \\
& = & \sum_{e\in R^m}d_{\phi(e)}(v) \pm 4 |\cI| \epsilon^{1/2}n.
\end{eqnarray*}
Given $v\in V(G)\setminus V_0$,
let $i(v)\in \Gamma$ be such that $v\in V_{i(v)}'$.
Then
\begin{eqnarray}
d_{G}(v) - \sum_{C''_r\in \cC}d_{C''_r}(v) 
& \leq & d_{G}(v) - \sum_{e\in R^m}d_{\phi(e)}(v) + 4 |\cI| \epsilon^{1/2}n \nonumber \\
&\leq & d_{G,V'_0 \cup V_{i(v)}'}(v) 
+ \sum_{j\in [\Gamma]\sm N_R[i(v)]}d_{G,V_{j}'}(v)\nonumber\\
&&+ \sum_{j\in N_R(i(v))} d_{E'_{ij}}(v) + 4|\cI|\epsilon^{1/2}n \nonumber \\
&\stackrel{ ({\rm ii}),({\rm iv}), \eqref{eq: E' deg}}{\leq} & 
2\epsilon^2 n + \Gamma n'/t^{1/2} + \epsilon^2 \Gamma \cdot n'  + \Gamma \cdot 3n'/(2t^{1/2}) +4|\cI|\epsilon^{1/2}n\nonumber\\
&\leq& 3n/t^{1/2} \label{eq:deg outside}.
\end{eqnarray}
This shows \ref{item:R3}.

Since $|V_0|\leq \epsilon n$ and by \eqref{eq: cC_S deg},
for every $S\in \cI$ we conclude that 
$e(\cC_S)=(1\pm \beta)e(\phi(R^m_S))$.
Furthermore,
by (c), we obtain that
$e(\phi(R^m_S))=(1\pm 4\beta^{2/3})\binom{D}{3}^{-1}e(\phi(R^m))$.
Since $e(G) = e(\phi(R^m)) \pm 3n^2/t^{1/2}$ by \eqref{eq:deg outside}, altogether this implies \ref{item:R2}.\COMMENT{
$e(G)
\geq e(\phi(R^m))
\geq \sum_{C\in \cC}e(C)\geq e(G)- 3n^2/t^{1/2}$.
}
Note that \ref{item:R1}, \ref{item:R2a} follow directly from our construction of the cycle blow-ups.
\end{proof}

\section{Packing trees into cycle blow-ups} \label{sec: trees}
The main result of this section is Lemma~\ref{lem: blow up advanced}, which states that we can obtain a near-optimal packing of a given collection of bounded degree forests into an $\epsilon$-regular cycle blow-up. The key ingredient for this is Theorem~\ref{thm: blow up} which is (a special case of) the main result in \cite{KKOT16}. Theorem~\ref{thm: blow up} achieves a near-optimal packing of suitable regular bounded degree graphs. In order to be able to apply this, we need to pack bounded degree trees into regular bounded degree subgraphs of cycle blow-ups. This is achieved by Lemma~\ref{lem: tree in a cycle} and \ref{packing into regular}.

At the beginning of this section, we also make some simple observations about trees that are used in later sections.
Let us start with some notation.
We write $(T,r)$ for a rooted tree $T$ with root $r$.
For a vertex $v$ in $(T,r)$, 
let $T(v)$ be the subtree of $T$ containing all vertices below $v$ (including $v$), that is, all vertices $u$ such that the path between $u$ and $r$ in $T$ contains $v$.
Moreover, let $T(v,t)$ be the subtree of $T(v)$ induced by all vertices in distance at most $t$ from $v$. 

\subsection{Simple results about trees}
In this subsection we collect simple results concerning $k$-independent sets in trees and partitions of trees into certain subtrees.

\begin{proposition}\label{prop: number of leaves}
Suppose $n,t,k,\Delta \in \N\sm \{1\}$. 
Suppose that $T$ is a tree on $n$ vertices with at most $t$ leaves, $\Delta(T)\leq \Delta$,
and $X\subseteq V(T)$ is a $k$-independent set in $T$.
Then
\begin{enumerate}[label=(\Roman*)]
\item $T$ contains at least $n-2t$ vertices of degree $2$, and
\item $T$ contains a $k$-independent set $Y\supseteq X$ of size at least $(n-2t)/\Delta^{k}$ such that $Y\sm X$ consists of vertices of degree $2$.
\end{enumerate}
\end{proposition}
\begin{proof}
To prove (I), note that if $T$ has at most $t$ leaves, it has at most $t-2$ vertices of degree at least $3$.
To prove (II), we greedily add vertices of degree $2$ in $T$ to $X$ to obtain a maximal $k$-independent set $Y$ with this property. 
Clearly, $|Y|\geq \max\{ (n-2t - |X| \Delta^k)/\Delta^k, 0\}+|X|\geq (n-2t)/\Delta^{k}$.
\end{proof}

The next proposition follows easily from a greedy argument.
\begin{proposition}\label{prop: k-independent set}
Suppose $n,k,\Delta \in \N\sm \{1\}$ 
and let $G$ be a graph on $n$ vertices with $\Delta(G)\leq \Delta$ and $\delta(G)\geq 1$. Suppose $X$ is a $k$-independent set of $G$ and $X\subseteq Z\subseteq V(G)$. Then the following hold.
\begin{enumerate}[label=(\Roman*)]
\item  There exists a $k$-independent set $Y$ such that 
$X\subseteq Y \subseteq Z$ and $|Y|\geq |Z|/\Delta^{k}$.
\item $G$ contains a $k$-independent matching $M$ with $|M|\geq e(G)/(2\Delta^{k})\geq n/(4\Delta^{k})$.
\end{enumerate}
\end{proposition}

The next result guarantees a subtree $T(y)$ of roughly prescribed order inside a rooted tree $(T,x)$ such that the distance between $T(y)$ and $x$ is not too small. We omit the proof as it is easy.

\begin{proposition}\label{prop: tree 2 parts}
Suppose $n, \Delta,k \in \N$ and $1/n\leq \alpha\leq  \Delta^{-k}/2$.
Suppose $(T,x)$ is a rooted tree on $n$ vertices with $\Delta(T)\leq \Delta$. 
Then there exists a vertex $y\in V(T)$ such that the subtree $T(y)$ of $T$ satisfies
\begin{enumerate}[label=(\Roman*)]
	\item $\alpha n \leq |T(y)|\leq \alpha \Delta n $, and
	\item the distance between $x$ and $y$ is at least $k$.
\end{enumerate}
\end{proposition}
Note that the definition of $T(y)$ implies that $T-V(T(y))$ is always a tree.
\COMMENT{
Let $x$ be the root of $T$. Note that for $u\in V(T)$, there exists a child $u'$ of $u$ such that $|T(u)|/\Delta \leq |T(u')| < |T(u)|$. 
Let $x=u_0$, and for each $u_i$ we choose $u_{i+1}$ such that $|T(u_i)|/\Delta \leq |T(u_{i+1})| < |T(u_{i})|$. 
Then there exists $t$ such that $T(u_{t})$ satisfies (I).  
Since $\alpha \leq \Delta^{-k}/2$, $t\geq k$. Thus we have (II).
}
\begin{proposition}\label{prop: taking subtree collection}
Suppose $\Delta,n\in \N$ with $1/n \ll \beta \ll \alpha\leq 1/\Delta$ and $c\in [2]$. 
Let $\cF$ be a collection of forests such that
\begin{enumerate}[label=(\roman*)]
\item each $F\in \cF$ consists of two components $K_F^1, K_F^2$,
\item $K_F^i$ has a root $x_F^i$ for every $i\in [2]$ and every $F\in \cF$,
\item $|K_F^1|, |K_F^2|\geq \alpha n$, and
\item $|F|\leq n$ and $\Delta(F)\leq \Delta$.
\end{enumerate}
Then for every $F\in \cF$, there exists a subforest $T_F$ of $F$ such that 
\begin{enumerate}[label=(\Roman*)]
\item $T_F$ consists of $c$ components $T_F^1,\dots,T_F^c$, 
\item for each $i\in [c]$, we have $T_F^i=K^i_F(y_F^i)$ for some $y_F^i\in V(K_F^i)$, 
in particular, $F-V(T_F)$ is a forest consisting of two components,
\item $T_F$ has distance at least $5$ from $\{x_F^1, x_F^2\}$,
\item $\Delta^{-1} \beta n/2 \leq |T_F^{i}|\leq \Delta\beta n$ for all $F\in \cF$ and $i\in [c]$, and
\item $\sum_{F\in \cF} e(T_F) = \beta  n |\cF| \pm n$.
\end{enumerate}
\end{proposition}
\begin{proof}
For each $F\in \cF$ and $j\in [c]$, 
we apply Proposition~\ref{prop: tree 2 parts} to $K^j_F$ to obtain 
vertices $y_F'^j, \hat{y}_F^j\in V(K_F^j)$ and subtrees $T_F'^j=K_F^j(y_F'^j)$ and $\hat{T}_F^j=K_F^j(\hat{y}_F^j)$
such that
\begin{itemize}
\item[(1)] $c^{-1}\Delta^{-1}\beta n \leq |T'^j_{F}|\leq c^{-1}\beta n$ and $c^{-1}\beta n \leq |\hat{T}^j_{F}|\leq c^{-1}\Delta\beta n$, and
\item[(2)] both $T'^j_F$ and $\hat{T}^j_F$ have distance at least $5$ from $x^j_F$.
\end{itemize} 
Let $T'_F:= \bigcup_{j=1}^{c} T'^j_F$ and $\hat{T}_F:=\bigcup_{j=1}^{c}\hat{T}^j_F$.
We write $\cF=\{F_1,\dots, F_{|\cF|}\}$ and consider $S_i:= \sum_{j=1}^{i} e(\hat{T}_{F_j}) + \sum_{j=i+1}^{|\cF|} e(T'_{F_j})$. Then $0\leq S_{i+1}-S_i = e(\hat{T}_{F_{i+1}})- e(T'_{F_{i+1}}) \leq n$ and $S_0\leq \beta n|\cF| \leq S_{|\cF|}$ by (1).
Thus there exists $0\leq t\leq |\cF|$ such that $S_t = \beta n |\cF|\pm n$. 
Let 
$$
T_{F_i}:= \left \{ \begin{array}{ll}
\hat{T}_{F_i} & \text{ if } i\leq t, \\
T'_{F_i} & \text{ if } i\geq t+1.
\end{array}\right.
$$
Then $\sum_{F\in \cF} e(T_F) = S_t = \beta  n |\cF| \pm n$, so we obtain (V).
Conditions (I)--(IV) follow from~(1) and~(2).
\end{proof}

\subsection{Trees in a cycle blow-up}
In this subsection, we find a near-optimal packing of a set $\cF$ of bounded degree forests into an $\epsilon$-regular cycle blow-up. The first step is to find a suitable embedding of a single tree (see Lemma~\ref{lem: tree in a cycle}). To analyze the corresponding embedding process, we consider a suitable partition into subtrees, which is given by Proposition~\ref{prop: tree decomp}.

\begin{proposition}\label{prop: tree decomp}
Suppose $n,\Delta \in \N\sm \{1\}$ and $n\geq t\geq 1$.
Then for any rooted tree $(T,r)$ on $n$ vertices with $\Delta(T)\leq \Delta$,
there exists a collection $\cS$ of pairwise vertex-disjoint rooted subtrees 
such that 
\begin{enumerate}[label=(\Roman*)]
	\item $S\subseteq T(s)$ for every $(S,s)\in \cS$,
	\item $t \leq |S|\leq 2\Delta t$ for every $(S,s)\in \cS$, and
	\item $\bigcup_{(S,s)\in\cS}V(S)=V(T)$.
\end{enumerate}
\end{proposition}
\begin{proof}
We use induction on $n$ for fixed $t$. 
If $t\leq n \leq 2\Delta t$, then $\cS:=\{ (T,r)\}$ is as desired. 
So let $m>2\Delta t$ and assume the proposition holds for all $n\leq m-1$. 
Consider a rooted tree $(T,r)$ on $m$ vertices.  
Choose $y$ at maximal distance from $r$ in $T$ subject to $|T(y)|\geq t$.
This implies that $t\leq |T(y)|\leq \Delta t$. 
Consider $T':= T- V(T(y))$. 
Since $t\leq m-\Delta t \leq |T'|\leq m-1$, 
by induction, we obtain a collection $\cS'$ of pairwise vertex-disjoint rooted subtrees satisfying (I)--(III) with $T'$ playing the role of $T$. 
Then it is clear that $\cS:=\cS'\cup \{(T(y),y)\}$ satisfies (I)--(III). 
Thus the proposition holds for all $n\geq t$.
\end{proof}

Recall from Section~\ref{sec: preliminaries} that we write $C(\ell,m)$ for an $\ell$-cycle $m$-blow-up. The next result describes how to embed a bounded-degree tree into a cycle blow-up in a uniform way, i.e. we embed approximately the same number of edges into every ``blown-up edge'' of $C(m,\ell)$. 
We will achieve this by considering a symmetric random walk on the cycle of length $\ell$.
Note that for Lemma~\ref{lem: tree in a cycle} it is crucial that $\ell$ is odd.

\begin{lemma}\label{lem: tree in a cycle}
Suppose $n,\Delta,\ell \in \N$ with $1/n\ll 1/\Delta, 1/\ell<1,$ where $\ell$ is odd. Let $m:= \frac{n}{\ell} + \frac{n}{\log^2 n}$.
Let $T$ be a tree on $n$ vertices with $\Delta(T)\leq \Delta$ and let $G:=C(\ell,m)$. 
Then there is an embedding $\phi$ of $T$ into $G$ such that $\phi(T)$ contains at most $m$ edges between any two clusters of $G$.
\end{lemma}
\begin{proof}
Choose an additional constant $\delta$ such that $1/n \ll \delta \ll 1/\ell$.
Pick $u\in V(T)$ and apply Proposition~\ref{prop: tree decomp} to $(T,u)$ with $n^{1-\delta}$ playing the role of $t$ to obtain a collection $\cS=\{(S_1,s_1),\dots, (S_{k},s_{k})\}$ of vertex-disjoint rooted subtrees such that 
\begin{itemize}
	\item $S_i\subseteq T(s_i)$ for every $i\in [k]$,
	\item $n^{\delta}/(2\Delta)\leq k\leq n^{\delta}$, 
	\item $|S_i|\leq 2\Delta n^{1-\delta}$, and
	\item $\bigcup_{i=1}^kV(S_i)=V(T)$.
\end{itemize}
We may assume that the $(S_i,s_i)$ are labelled in such a way that the distance of $s_1,\dots, s_k$ from the root $u$ of $T$ is non-decreasing. We will assign the vertices of $T$ one by one to some cluster of $G$ in such a way that the vertices of $S_i$ are assigned before the vertices of $S_{i+1}$, and such that within $S_i$ we assign the vertices in breadth-first order, starting with $s_i$. 
To choose our assignments, we consider the following random process. For every edge $e$ of $T$ we pick a label $X_e \in \{-1,1\}$ independently and uniformly at random. 
We assign $s_1$ to a cluster of $G$ uniformly at random. 
Assume that we have already assigned some vertices of $T$ to clusters of $G$ and next wish to assign $x$ to a cluster. Let $V_i$ be the cluster assigned to the ancestor $y$ of $x$. Assign $x$ to $V_{i-1}$ if $X_{xy}=-1$ and to $V_{i+1}$ otherwise. 
Note that this assignment of vertices of $T$ to clusters induces an assignment of the edges of $T$ to the pairs $(V_i,V_{i+1})$ of clusters. 
We will show that with positive probability for all $i\in [\ell]$ both the number of vertices of $T$ assigned to $V_i$ and the number of edges of $T$ assigned to the pair $(V_i,V_{i+1})$ is at most $m$. 
This then implies that the assignment corresponds to the required embedding of $T$ into $G$.

We first show that with probability at least $1-n^{-1}$ the number of vertices assigned to $V_r$ is at most $m$ for fixed $r\in[\ell]$. 
(Then a union bound over all $r\in [\ell]$ completes the proof for the statement about the vertices.)
We denote by $\phi: V(T)\rightarrow \{V_1,\dots, V_{\ell}\}$ the assignment of the vertices produced by the random process described above. 
For each $i\in [k]$, let $S'_i:= S_i- V(T(s_i,(\log_{\Delta}n)/2))$. 
Thus $S'_i$ is a subforest of $S_i$ and $|S_i|- |S'_i|\leq \Delta^{(\log_{\Delta}n)/2}= n^{1/2}$ and 
\begin{align}\label{eq: sum S'}
\sum_{i=1}^{k} |S'_i| = n \pm k n^{1/2} = n \pm n^{1/2+\delta}.
\end{align} 
Let $X$ be the total number of vertices assigned to $V_r$ and for each $i\in [k]$ define $X_i$ to be the number of vertices of $S'_i$ assigned to $V_r$. Note that
\begin{align}\label{eq: X Si inequality}
X\leq \sum_{i=1}^{k} |S_i\setminus S'_i|+ \sum_{i=1}^{k}X_i\leq n^{1/2+\delta} + \sum_{i=1}^{k}X_i.
\end{align}

Consider the exposure martingale $Y_i:= \mathbb{E}[\sum_{j=1}^{k} X_j \mid X_1,\dots, X_i]$. 
Thus $Y_0=\Exp[\sum_{j=1}^{k} X_j ]=\sum_{j=1}^{k} |S'_j|/\ell$.
Given any assignment $\phi(s_1),\ldots,\phi(s_i)$ of $s_1,\ldots,s_i$, 
by Lemma~\ref{lem: random walk}, 
there exists $0<\gamma=\gamma(\ell) <1$ such that for each $v\in \bigcup_{i'=i}^{k} S'_{i'}$, 
$$\mathbb{P}[\phi(v)=V_r \mid \phi(s_1),\ldots,\phi(s_i)] 
= \frac{1 \pm \gamma^{(\log_{\Delta}n)/2}}{\ell}.$$
Note that since $\delta \ll 1/\ell$ and $\gamma=\gamma(\ell)$, we have $\gamma^{(\log_{\Delta}n)/2} = n^{\frac{\log{\gamma}}{2\log{\Delta}}}\leq n^{-\delta}$.
Hence for any $i,j\in [k]$ with $i<j$,
$$
\mathbb{E}[X_j\mid  X_1,\dots, X_i] = \frac{(1 \pm n^{-\delta})|S'_j|}{\ell}.
$$
Let $i\in [k]$ and $x_1,\dots, x_i, x'_i\in \N\cup \{0\}$ be
such that $x_j\leq |S_j'|$ for each $j\in [i]$ and $x_i'\leq |S_i'|$.
Let $\cE'$ be the event that $X_1=x_1,\dots, X_{i}=x_i$ occurs, and let $\cE''$ be the event that $X_1=x_1,\dots, X_{i-1}=x_{i-1}, X_i=x'_i$.
Then
\begin{align*}
 \left|\mathbb{E}\left[\sum_{j=1}^{k}X_j \mid \cE'\right] - \mathbb{E}\left[\sum_{j=1}^{k}X_j\mid \cE''\right]\right| 
&\leq |x_i-x'_i| + \sum_{j=i+1}^{k} \left|\mathbb{E}[X_j \mid \cE'] - \mathbb{E}[X_j \mid \cE''] \right| \\
 &\leq |S'_i| + \sum_{j=i+1}^{k} \frac{1}{\ell} \cdot \left|(1 \pm n^{-\delta})|S'_j| - (1 \pm n^{-\delta})|S'_j|\right| \\
&\leq 2 \Delta n^{1-\delta} +  n^{-\delta} \sum_{j=i+1}^{k} |S'_j|
 \stackrel{\eqref{eq: sum S'}}{\leq} 4\Delta n^{1-\delta}. 
 \end{align*}
Therefore, $Y_i$ is a $4\Delta n^{1-\delta}$-Lipschitz martingale.
Thus by Lemma~\ref{Azuma} and \eqref{eq: sum S'}, 
we conclude that
\begin{align*}
\Pro\left[\sum_{i=1}^{k} X_i \neq  \frac{n}{\ell} \pm \frac{n}{2\log^2{n}} \right]
\leq   e^{-\frac{ (n/(3\log^2{n}))^{2} }{2k (4\Delta n^{1-\delta})^2}} \leq \frac{1}{n}.
\end{align*}
Together with \eqref{eq: X Si inequality}, this implies that with probability at least $1- n^{-1}$, the assignment $\phi$ uses at most $m$ vertices in $V_r$. 

In order to prove the corresponding statement involving the edges of $T$, 
observe that by Lemma~\ref{lem: random walk} for any edge $xy\in \bigcup_{i'=i}^k E(S'_{i'})$ such that $x$ is an ancestor of $y$, we have 
\begin{align*}
&\mathbb{P}[\phi(e) \in E(V_r,V_{r+1}) \mid \phi(s_1),\ldots,\phi(s_i) ] \\
=\,\, & \mathbb{P}[\phi(x) =V_r, X_{xy}=1 \mid \phi(s_1),\ldots,\phi(s_i) ] + \mathbb{P}[\phi(x) =V_{r+1}, X_{xy}=-1 \mid\phi(s_1),\ldots,\phi(s_i) ] \\ 
=\,\, & \frac{1 \pm \gamma^{(\log_{\Delta}n)/2}}{\ell}.
\end{align*}
Similar arguments as before lead to the desired statement.\COMMENT{The edges between different $S_i$ don't matter as there at at most $k\leq n^\delta$ of them.}
\end{proof}

The previous lemma shows that a tree can be embedded ``uniformly'' into a blow-up of an odd cycle. The next lemma allows us to combine such ``uniform'' embeddings of several trees and pack them together into a graph which is an internally regular blow-up of a cycle. 
The lemma is a special case of Lemma~7.1 in \cite{KKOT16}. 
Recall that a spanning subgraph $H$ of $C(\ell,m)$ is internally $k$-regular if $H[V_i,V_{i+1}]$ is $k$-regular for all $i\in [\ell]$ (where $V_1,\dots, V_\ell$ are the clusters of $C(\ell,m)$).

\begin{lemma}[\cite{KKOT16}] \label{packing into regular}
Suppose $n,\Delta,\ell,k,s\in \N$ with $1/n\ll 1/s \ll  1/\Delta$ and $1/n\ll 1/k$, and $\ell$ divides $n$. Suppose that $0<\xi<1$ is such that $s^{2/3} \leq \xi k$.
Let $G=C(\ell,n/\ell)$ and let $V_1,\dots, V_\ell$ be the clusters of $G$. 
Suppose that for each $j\in [s]$, 
the graph $L_j$ is a subgraph of $C(\ell,n/\ell)$ with clusters $X_1^j,\dots,X_\ell^j$ such that $\Delta(L_j)\leq \Delta$ and for each $i\in[\ell]$, we have
\begin{align}
\sum_{j=1}^{s} e(L_j[X_i^j,X_{i+1}^j]) = (1-3\xi\pm \xi)\frac{kn}{\ell}
\end{align}
where $X_{\ell+1}^j:=X_1^{j}$.
Then there exist an internally $k$-regular subgraph $H$ of $G$ and 
a function $\phi$ which packs $\{L_1,\ldots, L_s\}$ into $H$ such that $\phi(X^j_i)\sub V_i$.
Moreover, writing $J_{i}:=H[V_i,V_{i+1}]-\phi(L_1\cup\dots\cup L_s)$ for each $i\in[\ell]$, 
we have $\Delta(J_{i})\leq 4\xi k  +2s^{2/3}$. 
\end{lemma}
Theorem~\ref{thm: blow up} is a special case of Theorem~6.1 in \cite{KKOT16}. Theorem~\ref{thm: blow up} guarantees a near-optimal packing of the internally regular graphs given by Lemma~\ref{packing into regular} into a super-regular cycle blow-up. An important feature is that Theorem~\ref{thm: blow up} (iii) also allows for a small proportion of vertices to be embedded into a given ``target set''.
\begin{theorem}[\cite{KKOT16}]\label{thm: blow up}
Suppose $n,k,\ell,s \in \N$ with $1/n \ll  \epsilon \ll \alpha, d,d_0,1/k \leq 1$ and $1/n \ll 1/\ell$, and $\ell$ divides $n$. Suppose $s\leq \frac{d}{k\ell}(1-\alpha/2)n$ and the following hold.
\begin{enumerate}[label=(\roman*)]
\item $G$ is a $(\epsilon,d)$-super-regular $\ell$-cycle $\frac{n}{\ell}$-blow-up with clusters $V_1,\ldots, V_\ell$.
\item $\cH=\{H_1,\dots, H_{s}\}$, where each $H_i$ is an internally $k$-regular subgraph of $C(\ell,n/\ell)$
with clusters $X_1,\ldots,X_\ell$.
\item For all $j \in [s]$ and $i\in [\ell]$, 
there is a set $Y_i^j\subseteq X_i$ with $|Y_i^j|\leq \epsilon n/\ell$ and for each $y\in Y_i^j$, 
there is a set $A^j_{y}\subseteq V_i$ with $|A^j_y|\geq d_0n/\ell$. 
\end{enumerate}
Then there is a function $\phi$ packing $\cH$ into $G$ such that for all $j\in[s]$ and $i\in [\ell]$ the vertices of $H_j$ in $X_i$ are mapped to $V_i$ and such that $\phi(y)\in A^j_y$ for all $y\in Y^j_1\cup \dots\cup Y^j_{\ell}$.
\end{theorem}

We deduce the next lemma from Lemma~\ref{lem: tree in a cycle}, Lemma \ref{packing into regular} and Theorem~\ref{thm: blow up}. It guarantees a near-optimal packing of a set of rooted forests into a cycle blow-up which is compatible with given embeddings of the roots into a ``root set'' $R$.

\begin{lemma}\label{lem: blow up advanced}
Suppose $n,\ell,\Delta \in \N$ with $1/n\ll 1/\ell \ll \epsilon \ll \alpha,\eta, 1/\Delta,d, d_0\leq 1$, and $\ell$ is odd and divides $n$. 
Suppose that $G$ is a graph, $\cF$ is a collection of forests, and $R$ is a set with $V(G)\cap R=\emptyset$ satisfying the following:
\begin{enumerate}[label=(a\arabic*)]
\item\label{item:A1} $G$ is an $(\epsilon,d)$-super-regular $\ell$-cycle $\frac{n}{\ell}$-blow-up with clusters $V_1,\ldots, V_\ell$,
\item\label{item:A2} $G'$ is a graph with $R\subseteq V(G')\subseteq V(G)\cup R$ such that $|N_{G'}(u)\cap V(G)|\geq d_0 n$ for each $u\in R$,

\item\label{item:A3} for every $F\in \cF$, we have $\eta n \leq |F| \leq (1-\eta)n$, $F$ has $c(F)$ components where $c(F)\leq 2\Delta$ and $\Delta(F)\leq \Delta$,
\item\label{item:A4} $e(\cF)= (1-2\alpha\pm \alpha)e(G)$,
\item\label{item:A5} for every $F\in \cF$, there is a set $r(F)=\{r^1_F,\dots, r^{c(F)}_F\}$ of vertices belonging to distinct components of $F$,
\item\label{item:A6} there is a set $U'\subseteq V(G)$ with $|U'|\leq n^{1/2}$ and 
for every $F\in \cF$, there is a set $X'_F \subseteq V(F)$ with $|X'_F|\leq n^{1/2}$, and
\item\label{item:A7} there is a function $\phi':\{r^c_F: c\in [c(F)], F\in \cF\}\to R$ such that $|\phi'^{-1}(v)|\leq \ell^{1/2}$ for every $v\in R$ and $\phi'(r^c_F)\neq \phi'(r^{c'}_F)$ whenever $c\neq c'$.
\end{enumerate}
Then there exists a function $\phi$ which is consistent with $\phi'$ and packs $\cF$ into $G \cup G'$ satisfying the following:
\begin{enumerate}[label=(A\arabic*)]
\item\label{item:B1} $\Delta( G-E(\phi(\cF)) ) \leq 10\alpha d n/\ell$,
\item\label{item:B2} $\phi(V(F)\setminus r(F)) \cap R =\emptyset$ for every $F\in \cF$, and
\item\label{item:B3} $\phi(X'_F) \cap U' =\emptyset$ for every $F\in \cF$.
\end{enumerate}
\end{lemma}
To prove Lemma~\ref{lem: blow up advanced} we apply Lemma~\ref{lem: tree in a cycle} separately to all the components of $F-r(F)$ for all $F\in \cF$ to obtain a ``uniform'' embedding into a blown-up cycle. We also split $\cF$ into $\cF_1,\dots,\cF_t$ and use Lemma~\ref{packing into regular} for all $i\in [t]$ to pack together these uniform embeddings of $F\in \cF_i$ into internally regular graphs $H_i$. The graphs $H_i$ are then packed into $G$ via Theorem~\ref{thm: blow up}.
\begin{proof}
For each $F\in \cF$, let $F^*:= F-r(F)$. Then the forest $F^*$ has at most $2\Delta^2$ components. Choose an integer $k$ and a constant $\xi$ such that $\epsilon \ll 1/k \ll \xi \ll \alpha, \eta$. Let 
\begin{align}\label{eq: t def size}
t:= \frac{\sum_{F\in \cF}e(F^*) }{(1-3\xi)k n}.
\end{align}
We partition $\cF$ into $\cF_1,\dots, \cF_t$ such that $\sum_{F\in \cF_i} e(F^*) = (1-3\xi\pm \xi/2)kn$ for all $i\in [t]$. Note that 
$\eta n |\cF_i|\leq e(\cF_i) \leq (1-\xi)kn$ by \ref{item:A3} and so
\begin{align}\label{eq: size cFi}
|\cF_i|\leq k \eta^{-1}.
\end{align}

Note that by \ref{item:A4}, we have\COMMENT{$4\alpha/3$ because of adding roots.}
\begin{align}\label{eq: t play s}
t = \frac{(1-2\alpha\pm 4\alpha/3)e(G)}{(1-3\xi)kn}\leq \frac{ (1-2\alpha/3)(1+\epsilon)dn^2/\ell}{(1-3\xi)kn}\leq \frac{d}{k\ell}(1-\alpha/2)n.
\end{align} 
 We claim that for each $r^c_F$ there is an index $j(F,c)\in [t]$ such that
 \begin{enumerate}[label=(\roman*)]
 \item $|N_{G'}(\phi'(r^c_F))\cap V_{j(F,c)}| \geq  d_0|V_{j(F,c)}|/2$ and
 \item $j(F,c)\neq j(F',c')$ if one of the following holds:
 \begin{itemize}
 \item[($j$1)] $F,F'\in \cF_i$ for some $i\in [\ell]$ and $(F,c)\neq (F',c')$,
 \item[($j$2)] $\phi'(r^c_F)=\phi'(r^{c'}_{F'})$ and $(F,c)\neq (F',c')$.
 \end{itemize}
 \end{enumerate}
Indeed, we can choose the $j(F,c)$ greedily. When choosing $j(F,c)$ for $F\in \cF_i$, the number of choices excluded by ($j$1) is at most $2\Delta|\cF_i|$ and the number of choices excluded by ($j$2) is at most $\ell^{1/2}$ by \ref{item:A7}. Hence the total number of excluded choices is at most $2\Delta|\cF_i| + \ell^{1/2} \leq 2\Delta k \eta^{-1} + \ell^{1/2} \leq d_0\ell/3$ by \eqref{eq: size cFi}. On the other hand, by \ref{item:A2}, for each $r^c_F$ there are at least $d_0 \ell/2$ indices $j\in [\ell]$ such that $|N_{G'}(\phi'(r^c_F))\cap V_{j}|\geq d_0 |V_j|/2$. This proves the claim.

For each $F\in \cF_i$ and each component $J$ of $F^*$, we apply Lemma~\ref{lem: tree in a cycle} to $J$ in order to find a partition $(X^J_1,\dots, X^J_\ell)$ of $V(J)$ such that for each $j\in [\ell]$, 
\begin{itemize}
\item $|X^J_{j}| \leq \frac{|J|}{\ell} + \frac{n}{\log^2{n}}$.
\item $e(J[X^J_{j},X^J_{j'}]) = \left\{\begin{array}{ll}
\frac{e(J)}{\ell} \pm \frac{n}{\log^2{n}} & \text{ if } |j-j'|=1~(\text{mod}~\ell),\\
0 & \text{ otherwise.}
\end{array}\right.$
\end{itemize}
From this, by rotating the indices and taking disjoint unions of the components, we obtain a partition $(X^F_1,\dots, X^F_{\ell})$ of $V(F^*)$ such that 
\begin{itemize}
\item[($\alpha$1)] $|X^F_{j}|\leq n/\ell$, \COMMENT{Note that $|F|\leq (1-\eta)n$, thus for each $j\in[\ell]$ $\sum_{J} |X^J_j| \leq (1-\eta)n/\ell + 2\Delta n/\log(n) \leq n/\ell$.} 
\item[($\alpha$2)] $ e(F^*[X^F_{j},X^F_{j'}]) = \left\{\begin{array}{ll}
(1\pm \epsilon) \frac{e(F^*)}{\ell} & \text{ if $|j-j'|=1$ } ({\rm mod} ~\ell),  \\
0 & \text{ otherwise,}\\ 
\end{array}\right.
$
\item[($\alpha$3)] $N_{F}(r^c_F)\subseteq X^F_{j(F,c)}$.
\end{itemize}
Note that for each $j\in [\ell]$,
$$\sum_{F\in \cF_i} e(F^*[X_j^F, X_{j+1}^F])\stackrel{(\alpha2)}{=} \sum_{F\in \cF_i}(1\pm \epsilon) \frac{e(F^*)}{\ell} = (1-3\xi\pm \xi)\frac{kn}{\ell}.$$
Moreover, $|\cF_i|^{2/3}\leq \xi k$ by \eqref{eq: size cFi}. Thus we can apply Lemma~\ref{packing into regular} with $|\cF_i|$ playing the role of $s$, in order to obtain a function $\tau_i$ packing $\{F^*: F\in \cF_i\}$ into $H_i$, where $H_i$ is an internally $k$-regular subgraph $H_i$ of $C(\ell,n/\ell)$, such that $\tau_i$ maps $X_j^F$ into the $j$th cluster $X_j$ of $C(\ell,n/\ell)$ for every $F\in \cF_i$ and every $j\in [\ell]$. Also for all $x\in V(H_i)$ we have 
\begin{align}\label{real edges in Hi}
d_{H_i}(x) - \sum_{F\in \cF_i} d_{\tau_i(F^*)}(x) \leq 8\xi k + 4 |\cF_i|^{2/3} \leq 12\xi k.
\end{align} 
Note that $\tau_i(N_{F}(r^c_F)) \cap \tau_{i}(N_{F'}(r^{c'}_{F'}))= \emptyset$ for all $F\neq F' \in \cF_i$ and $c\in [c(F)], c'\in [c(F')]$ by ($j$1).\COMMENT{We automatially have that $\tau_i(N_{F}(r^c_F))\cap \tau_i(N_F(r^{c'}_F))=\emptyset$ if $c\neq c'$ since $\tau_i(F^*)$ is a proper embedding of $F^*$.} For all $F \in \cF_i$, $c\in [c(F)]$ and $x\in N_{F}(r^c_F)$, let 
$$A^c_F(x) := N_{G'}(\phi'(r^c_F))\cap V_{j(F,c)}.$$ Thus $|A^c_F(x)| \geq d_0 n/(2\ell)$ by (i). For all $i\in [t]$, let 
$$Y_i:= \bigcup_{F\in \cF_i} \left(\tau_i(X'_F\setminus r(F))\cup \bigcup_{c\in [c(F)]} \tau_i(N_{F}(r^c_F)) \right).$$
Then $Y_i\subseteq V(H_i)$ and $|Y_i|\leq |\cF_i| (n^{1/2} + 2\Delta^2) \leq \epsilon n/\ell$ by \ref{item:A6} and \eqref{eq: size cFi}.

For all $i\in[t]$, $j\in [\ell]$ and $y\in Y_i\cap X_j$, we let 
$$A^i_y =\left\{\begin{array}{ll} A^c_F(\tau^{-1}_i(y))\setminus U' &\text{ if }y\in\tau_{i}(N_{F}(r^c_F)) \text{ for }F\in \cF_i, c\in [c(F)], \\
 V_j\setminus U' &\text{ otherwise.}
\end{array}\right.$$

Note that ($\alpha$3) implies $\tau_{i}(N_{F}(r^c_F))\subseteq X_{j(F,c)}$. So the sets $A^i_y$ are well-defined since the pair $(F,c)$ is unique by ($j$1). Moreover, $|A^i_{y}|\geq d_0n/(3\ell)$ for each $y\in Y_i$. Together with \ref{item:A1} and \eqref{eq: t play s},
this means that we can apply Theorem~\ref{thm: blow up} to $G$ and $\cH:= \{ H_1,\dots, H_t\}$ with $d_0/3$ playing the role of $d_0$ in order to find a function $\phi''$ which packs $\cH$ into $G$ such that for all $i\in [t]$ and $j\in [\ell]$ the vertices of $H_i$ in $X_j$ are mapped to $V_j$ and such that $\phi''(y) \in A^i_y$ for all $y\in Y_i$. 
For $x\in \bigcup_{F\in \cF} V(F)$, we let 
$$\phi(x) = \left\{\begin{array}{ll} \phi''(\tau_i(x)) & \text{ if } x\in V(F)\setminus r(F) \text{ for some }F\in \cF_i, i\in [t],\\
 \phi'(x) & \text{ if } x\in r(F) \text{ for some }F\in \cF_i, i\in [t].\end{array}\right.$$
 Then for all $F\in \cF_i$ and $x\in N_F(r^c_F)$,
 $$\phi(x)= \phi''( \tau_{i}(x)) \in A^i_{\tau_i(x)} \subseteq A^c_F(x)\subseteq N_{G'}(\phi(r^c_F)).$$  
 Moreover, if $F,F'\in \cF$, $c\in [c(F)], c'\in [c(F')]$, $(F,c)\neq (F',c')$ and $\phi(r^c_F)=\phi(r^{c'}_{F'})$, then ($j$2) implies
 $$\phi(N_{F}(r^c_F) )\cap \phi(N_{F'}(r^{c'}_{F'}))= \emptyset.$$  Thus $\phi$ packs $\cF$ into $G\cup G'$ such that $\phi$ is consistent with $\phi'$ and \ref{item:B2} holds. Again, \ref{item:B3} also holds since $A^i_{\tau_i(x)} \cap U'=\emptyset$ for all $x\in X'_F$. 
 To show \ref{item:B1}, let $G'':=G-E(\phi(\cF))$. 
Then for each vertex $v\in V(G)$ 
\begin{eqnarray*}
d_{G''}(v) 
&=& d_{G}(v) - \sum_{F\in \cF} d_{\phi(F)}(v) \stackrel{\eqref{real edges in Hi}}{\leq} \frac{2(d+\epsilon)n}{\ell} - \sum_{i=1}^{t} (2k-12\xi k) \\ 
&\stackrel{\eqref{eq: t def size}}{\leq}& \frac{2(d+\epsilon)n}{\ell} - \frac{(2-12\xi)e(\cF)}{(1-2\xi)n}
 \stackrel{\ref{item:A4}}{\leq} \frac{10\alpha dn}{\ell}.
\end{eqnarray*}
Thus \ref{item:B1} also holds.
\end{proof}

\section{Covering parts of graphs with trees}\label{sec:cleaning}

In this section we present four results which will be used in Section~\ref{sec:iteration}.
In each of these lemmas we use a collection of trees or forests to cover certain given sets of edges of a graph $G$. 
In the first lemma, we are given a small ``exceptional'' set $V_0\sub V(G)$ and 
a collection $\cF$ of small forests, each consisting of two components.
The lemma guarantees a packing of $\cF$ which covers almost all edges of $G$ incident to $V_0$, but not too many edges incident to any vertex in $V(G)\setminus V_0$.

Throughout the section, we write $S(c)$ for a star with centre $c$.

\begin{lemma}\label{lem: clear rl waste}
Suppose $n,\Delta \in \N$ and $1/n \ll \epsilon \ll \eta \ll \delta_2 \ll  1/\Delta\leq 1/2$.
Let $G$ be a graph on $n$ vertices and let $V_0\sub V(G)$ with $|V_0|\leq \epsilon n$.
Let $\cF$ be a collection of forests with $|\cF|=(1/2 \pm \epsilon )n$ and $\Delta(F)\leq \Delta$ for every $F\in \cF$. 
Let $\{W_F\}_{F\in \cF}$ be a collection of vertex sets with $W_F\sub V(G)$ for every $F\in \cF$. 
Suppose 
\begin{enumerate}[label=(c1.\arabic*)]
\item\label{item:L11} $d_{G}(u,v)\geq \delta_2^2 n/2$ for all $u,v\in V(G)\sm V_0$,
\item\label{item:L12} for every $F\in \cF$, the forest $F$ consists of two components $T_F^1, T_F^2$ such that $\eta n/\Delta \leq |T_{F}^i| < \eta^{3/4} n$ for all $i\in[2]$,\COMMENT{we only need one, but it is more convenient for the proof in Lemma 8.1.}
\item\label{item:L13} for every $F\in \cF$, 
there is a $5$-independent set $R_F:=\{y^1_F, y^2_F, z^1_F, z^2_F\}$ such that $y^i_F,z^i_F\in V(T_F^i)$
and there is an injective function $\tau'_F:R_F\to V(G)\sm (V_0 \cup W_F)$ such that $|\{F : u\in \tau'_F(R_F)\}| \leq \epsilon n$ for each $u\in V(G)$,
\item\label{item:L14}	for every $F\in \cF$, we have $|W_F|\leq \Delta^2$, and
\item\label{item:L15}  $|\{F \in \cF : u\in W_F\}|\leq \epsilon n$ for every $u\in V(G)$.
\end{enumerate}
Then there is a  function $\tau$ packing $\cF$ into $G$ which is consistent with $\{\tau'_F\}_{F\in \cF}$
such that 
\begin{enumerate}[label=(C1.\arabic*)]
\item\label{item:Q11} $\tau(F)\cap W_F=\emptyset$ for all $F\in \cF$,
\item\label{item:Q12} $d_{\tau(\cF)}(v)\leq \eta^{1/3}n$ for all $v\in V(G)\sm V_0$, and
\item\label{item:Q13} $d_{\tau(\cF)}(v) \geq d_{G}(v) - \epsilon^{1/2} n$ for all $v\in V_0$.
\end{enumerate}
\end{lemma}
When we apply Lemma~\ref{lem: clear rl waste} in Section~8, $F$ will be a part of a larger forest. The remaining parts of this larger forest are attached to $F$ at the ``root set'' $R_F$.
\begin{proof}
We write $\cF =\{F_1,\dots, F_{t}\}$, 
$\tau'_i:= \tau'_{F_i}$,
and $p_R(u,i):=|\{i'\in[i]:u\in \tau_{i'}'(R_{F_{i'}})\}|$ for all $i\in [t]$ and $u\in V(G)$. 
Thus \ref{item:L13} implies that for all $i\in [t]$ and $u\in V(G)$, 
\begin{align}\label{eq: pr en}
p_R(u,i)\leq \epsilon n.
\end{align}
We use an algorithmic approach and iteratively find an embedding of $F_i\in \cF$ into $G$.

To be precise, we will construct embeddings $\tau_1,\dots, \tau_{t}$
and define sets $A_i,X_i\sub V(G)\sm (V_0 \cup W_{F_i})$ for all $i\in [t]$
such that for all $i\in [t]$ the following hold:
\begin{itemize}
\item[(A1)$_i$] $|A_i| =\eta^{2/3}n$.
\item[(A2)$_i$] For all $u\in V(G)\setminus V_0$, let 
$$p_A(u,i) := |\{i'\in[i]: u\in A_{i'}\sm (X_{i'}\cup\tau_{i'}'( R_{F_{i'}}))\}| \text{ and } p_X(u,i) := |\{i'\in[i]: u\in X_{i'}\}|.$$
Then $p_A(u,i)\leq \eta^{1/2}n$ and $p_X(u,i)\leq \epsilon n$. 

\item[(A3)$_i$] Let $G_{i}:= G- \bigcup_{j=1}^{i-1}E(\tau_{j}(F_{j}))$. Then
$d_{G_i,A_i}(u,v)  \geq 2\eta^{3/4} n$ for every $u,v\in V(G)\setminus V_0$.
\item[(A4)$_i$]$\tau_{i}$ embeds $F_i$ into $G_i[(V_0 \sm W_{F_i}) \cup A_i \cup X_i ]$.
\item[(A5)$_i$] Let $W'_i:= \{v\in V_0\setminus W_{F_i}: d_{G_i, V(G_i)\setminus(V_0 \cup W_{F_i})}(v) \geq \Delta \epsilon n +4\}$.
Then $W_i'\sub \tau_i(F_i)$ and for every $w\in W_i'$,
the vertex $\tau_i^{-1}(w)$ is a non-leaf of $F_i$.
\item[(A6)$_i$] $X_i=N_{\tau_i(F_i)}(W_i')$.
\end{itemize}
We will then show that $\tau:=\bigcup_{i=1}^t\tau_i$ satisfies \ref{item:Q11}--\ref{item:Q13}.
In particular, (A5)$_1$--(A5)$_t$ will ensure that \ref{item:Q13} holds.

Assume that for some $i\in [t]$ and all $j<i$ we have constructed $\tau_j, A_j, X_j$ satisfying (A1)$_j$--(A6)$_j$. 
Note that for $u,v\in V(G)\sm V_0$, we have
\begin{eqnarray}
d_{G_i}(u,v)  
&\stackrel{\text{(A4)$_{1}$--(A4)$_{i-1}$}}{\geq} & d_{G}(u,v)  -\sum_{w\in \{u,v\}} \Delta (p_A(w,i-1)+p_X(w,i-1)+p_R(w,i-1))  \nonumber \\
&\stackrel{\text{\ref{item:L11}},\text{(\ref{eq: pr en})}, ({\rm A}2)_{i-1}}{\geq}& 
\delta_2^2n/2 - 2\Delta\eta^{1/2} n - 4\Delta\epsilon n \label{Gi codeg}
\geq \delta_2^2 n/3.
\end{eqnarray}
Next we show how to choose an embedding $\tau_i$ and sets $A_i, X_i\subseteq V(G)\setminus (V_0\cup W_{F_i})$ satisfying (A1)$_i$--(A6)$_i$.
Let $B:=\{w \in V(G)\setminus V_0: p_A(w,i-1) \geq \eta^{1/2}n/2\}$. 
Then $|B| \leq 2(\eta^{1/2}n)^{-1}\sum_{j<i} |A_j| \leq 2\eta^{1/6} n$, 
because (A1)$_j$ holds for all $j<i$ and since $i\leq t\leq n$. 
Consider a set $A'$ of $\eta^{2/3}n-4$ vertices in $V(G)\sm (V_0\cup B \cup W_{F_i} \cup \tau_i'(R_{F_i}))$ chosen uniformly at random. 
Note that for any two distinct vertices $u,v \in V(G_i)\setminus V_0$, we have
\begin{eqnarray*}
d_{G_i-(V_0\cup B\cup W_{F_i}\cup \tau_i'(R_{F_i}))}(u,v)
&\geq& d_{G_i}(u,v)- 2\eta^{1/6}n-|V_0 \cup W_{F_i}\cup \tau_i'(R_{F_i})|\\
& \stackrel{\eqref{Gi codeg},\text{\ref{item:L14}}}{\geq}& \delta_2^2 n/3 - 2\eta^{1/6}n - \epsilon n-\Delta^2 -4\\
&\geq& \delta_2^2 n/4.
\end{eqnarray*} 
A straightforward application of Lemma~\ref{lem: chernoff} shows that for any two distinct vertices $u,v \in V(G)\setminus V_0$,
\begin{align}\label{eq:deg G_i}
	d_{G_i,A'}(u,v)\geq \eta^{2/3}\delta_2^2 n/8>2\eta^{3/4}n
\end{align}
 with probability at least $1-n^{-3}$.  
Thus there is a choice of $A'$ that satisfies \eqref{eq:deg G_i} for all $u,v\in V(G)\sm V_0$.
Let $A_i:= A'\cup \tau_{i}'(R_{F_i})$ for this choice. 
Then clearly (A1)$_i$ and (A3)$_i$ hold.

We write $W'_i=\{w_1, \ldots, w_r\}$. Note that $r\leq |V_0|\leq \epsilon n$.
Consider the set $V_{\mathrm{nl}}(F_i)$ of all non-leaf vertices of $F_i$. 
Then $|V_{\mathrm{nl}}(F_i)|
\geq |T_{F_i}^1|/\Delta 
\geq \eta n/\Delta^2
\geq \Delta^5 (r+4)$ by \ref{item:L12}. 
Recall that $R_{F_i}$ is a $5$-independent set in $F_i$. 
So by Proposition~\ref{prop: k-independent set}~(I), 
there exists a $5$-independent set $I\cup R_{F_i}\subseteq V_{\mathrm{nl}}(F_i)\cup R_{F_i}$ of $F_i$ with $|I|=r$ and $I\cap R_{F_i}=\es$.
Write $I=\{v_1,\ldots, v_r\}$ and
note that no vertex in $I$ is a leaf of $F_i$. 

To ensure (A5)$_i$, we will construct $\tau_i$ in such a way that $\tau_i(I)=W_i'$.
To see that this can be done, recall that the definition of $W'_i$ gives that $d_{G_i,V(G_i)\setminus(V_0\cup W_{F_i})}(w)\geq \Delta \epsilon n+4$ for all $w\in W'_i$. 
Thus there is a collection of vertex-disjoint stars $S(w_1),\ldots,S(w_r)$ in $G_i$, 
where $S(w_j)$ has $d_{F_i}(v_j)$ leaves which avoid $V_0 \cup W_{F_i}\cup \tau'_{i}(R_{F_i})$.
Let $X_i$ be the set of all these leaves of $S(w_1),\ldots,S(w_r)$. 
Let $\phi''$ be the bijection between $I$ and $W_i'$ defined by $\phi''(v_j):=w_j$ for every $j\in [r]$. 
Let $J := N_{F_i}(I)$ and let $\phi': J\cup R_{F_i} \to X_i \cup \tau'_{i}(R_{F_i})$ be a bijective function consistent with $\tau'_{i}$ such that for every $j\in[r]$,
$$\phi'(N_{F_i}(v_j)) = V(S(w_j))\setminus \{w_j\}.$$ 

Let $F'_i:= F_i- I$. 
Note that $J\cup R_{F_i}$ is a $3$-independent set in $F_i'$. 
Moreover, (A3)$_i$ and \ref{item:L12} imply that any two vertices $u,v \in V(G)\sm V_0$ satisfy 
$d_{G_i,A_i\cup X_i}(u,v) \geq 2\eta^{3/4}n > |F_i'|$. 
So we can apply Lemma~\ref{lem: embed forests} with $|A_i\cup X_i|, |F_i'|$ playing the roles of $n, k$,
to obtain a function $\phi_i$ embedding $F_i'$ into $G_i[A_i\cup X_i]$ which is consistent with $\phi'$. 
Let $\tau_{i}:= \phi_i\cup \phi''$.
Thus $\tau_{i}$ embeds $F_i$ into $G_i[(V_0 \sm W_{F_i}) \cup A_i \cup X_i ]$ and satisfies $\tau_i^{-1}(W_i')=I$ and $X_i=N_{\tau_i(F_i)}(W_i')$. 
Hence (A4)$_i$--(A6)$_i$ hold.

Note that by definition of $A'$, for any $u\in V(G)\setminus V_0$, 
we obtain that $p_{A}(u,i)\leq
\eta^{1/2} n$. 
Since $E(\tau_{j}(F_j))\cap E(\tau_{j'}(F_{j'})) = \emptyset$ whenever $1\leq j \leq j'\leq i$, 
(A6)$_i$ implies
$p_X(u,i)\leq d_{G,V_0}(u)\leq |V_0|\leq \epsilon n$ for each $u\in V(G)\sm V_0$. Thus (A2)$_i$ holds.
This shows that we can construct $\tau_i$ such that (A1)$_i$--(A6)$_i$ hold.

\medskip
Let $\tau := \bigcup_{i=1}^{t} \tau_{i}$. 
Then $\tau$ packs $\cF$ into $G$ and is consistent with $\{\tau'_F\}_{F\in \cF}$.
Clearly \ref{item:Q11} holds by (A4)$_1$--(A4)$_t$.
Moreover, by (A2)$_{t}$, (A4)$_1$--(A4)$_t$ and \eqref{eq: pr en}, 
we have $d_{\tau(\cF)}(v) \leq \Delta (p_A(v,t) +  p_X(v,t)+ p_R(v,t)) \leq \eta^{1/3}n$ for all $v\in V(G)\setminus V_0$
and
thus \ref{item:Q12} holds.

Note that (A5)$_i$ implies for every $w\in W_i'$ that
\begin{align}\label{eq: degree decrease by 2}
d_{G_{i}}(w)-d_{G_{i+1}}(w)\geq 2.
\end{align}
 To show \ref{item:Q13}, assume for a contradiction that there is a vertex $v \in V_0$ which does not satisfy \ref{item:Q13}, 
i.e.~that $d_{G_{t+1}}(v) > \epsilon^{1/2}n$. 
Then
$$|\{i\in[t]: v\in W'_i\}| \geq t - |\{i\in[t]: v\in W_{F_i}\}| \stackrel{\text{\ref{item:L15}}}{\geq} n/2 -2\epsilon n.$$ 
Together with \eqref{eq: degree decrease by 2} this yields that $d_{G_{t+1}}(v) \leq d_{G}(v) - n + 4\epsilon n$, 
which is a contradiction to the assumption that $d_{G_{t+1}}(v) > \epsilon^{1/2}n$.
\end{proof}
Given a collection of trees $\cT$, 
a graph $G$ with vertex partition $(A,B)$ such that $G[B]$ is sparse,
every vertex in $B$ has many neighbours in $A$,
and $G[A]$ is quasi-random,
the following lemma allows us to find a packing of $\cT$ into $G$ which covers all edges of $G[B]$.

\begin{lemma}\label{lem: clear delta2}
Suppose $n,\Delta\in \N\sm\{1\}$ and $1/n\ll\epsilon \ll\delta_2 \ll\delta_1 \ll  \gamma \ll 1/\Delta,p \leq 1$.
Let $G$ be a graph on $n$ vertices and $A$ be a set of at least $\delta_1 n$ vertices of $G$.
Let $B:=V(G)\sm A$, $H:=G[A,B]$, and 
let $\cT$ be a collection of trees with $\Delta(T)\leq \Delta$ for every $T\in \cT$. 
Suppose 
\begin{enumerate}[label=(c2.\arabic*)]
	\item\label{item:L21} $G[A]$ is $(\gamma^{1/10},p)$-quasi-random, 
	\item\label{item:L22} $\Delta(G[B])\leq 2\delta_2 n$,
	\item\label{item:L23} $d_H(u)\geq \delta_1 |A|/20$ for every $u\in B$,
	\item\label{item:L24} $\delta_2^{1/2} n \leq |\cT| \leq \gamma |A|$, 
	\item\label{item:L25} every $T\in \cT$ satisfies $\delta_1 n\leq |T|\leq p^2|A|/2$ and has a root $y_T$;
	moreover, there is a set $W_T\sub V(G)$ with $|W_T|\leq \Delta^2$ and
	 a function $\tau_T':\{y_T\}\to A \sm W_T$, and
  \item\label{item:L26} $|\{T\in \cT: u\in W_{T}\}|\leq \epsilon n$ for every $u\in V(G)$.
\end{enumerate}
Then there is a function $\tau$ packing $\cT$ into $G$ which is consistent with $\{\tau_T'\}_{T\in \cT}$ such that 
\begin{enumerate}[label=(C2.\arabic*)]
	\item\label{item:Q21} $E(G[B])\subseteq E(\tau(\cT))$ and
	\item\label{item:Q22} $W_{T}\cap  \tau(T)=\es$ for each $T\in \cT$.
\end{enumerate}
\end{lemma}

\begin{proof}
We write $\cT=\{T_1,\ldots,T_t\}$ and let $\tau'_i:=\tau'_{T_i}$ for every $i\in [t]$.
Our proof strategy is as follows.
We first decompose $G[B]$ into $t$ edge-disjoint small matchings $M_1,\ldots, M_t$ such that
$W_{T_i}\cap V(M_i)=\es$ for all $i\in [t]$.
We then choose a suitable matching $N$ in $T_i-\{y_{T_i}\}$ and embed $T_i$ into $G$
in such a way that $N$ is mapped to $M_i$ and $y_{T_i}$ is mapped to $\tau_i'(y_{T_i})$.

For every edge $e\in E(G)$, let $w(e):= |\{i\in [t]: W_{T_i} \cap V(e)\neq \es\}|$. 
Hence $w(e)\leq 2\epsilon n$ by \ref{item:L26}.
For every set of edges $E\sub E(G)$, 
we define $w(E):=\sum_{e\in E}w(e)$.
Observe that \ref{item:L24} and \ref{item:L25} imply that $w(E)\leq \Delta^2 \gamma n$ for any matching $E$.

By \ref{item:L22}, the graph $G[B]$ has an edge colouring using at most $3\delta_2 n$ colours. 
This gives rise to a partition of the edge set of $G[B]$ into matchings $M_1',\ldots,M_{3\delta_2 n}'$.
We apply Proposition~\ref{prop: weight partition} to $M_i'$ in order to split $M_i'$ into $q:=\delta_2^{-1/2}/3$ possibly empty matchings $M_1^i,\dots,M_q^i$ 
such that $|M_j^i|\leq 3\delta_2^{1/2}n$ and 
$w(M_j^i) \leq 3\delta_2^{1/2} \cdot 2\Delta^2\gamma n + 2\epsilon n\leq \delta^{1/2}_2 n/2$ for all $i\in[3\delta_2 n]$ and $j\in [q]$.

Note that $q\cdot 3\delta_2n \leq t$ by \ref{item:L24}. 
So by adding empty matchings to $\{M_j^i\}_{i\in[3\delta_2 n],j\in [q]}$ if necessary, we obtain a collection $\cM=\{M_1,\ldots,M_{t}\}$ of matchings partitioning $E(G[B])$ such that for all $i\in [t]$, we have
\begin{align}\label{eq: Mi size weight}
|M_i|\leq 3\delta_2^{1/2}n\text{ and } w(M_i) \leq \delta^{1/2}_2n/2.
\end{align}
By \ref{item:L22}, observe that for every $v\in V(G)$
\begin{align}\label{eq:match}
	|\{ i\in [t] : v\in M_i\}| \leq 2\delta_2 n.
\end{align}
Also recall $|\cM|= |\cT|=t$.
Consider an auxiliary bipartite graph $\cH$ with vertex partition $(\cT,\cM)$ such that $T_i M_j \in E(\cH)$ if $W_{T_i}\cap V(M_j)=\es$. 
Then, for every tree $T\in\cT$, 
we conclude 
$$d_{\cH}(T)\geq |\cM|- |\{ i\in[t]: W_{T} \cap V(M_i)\neq \es\}|
\stackrel{\eqref{eq:match}}{\geq} t - 2\delta_2 n |W_T|
\stackrel{\ref{item:L25}}{>} 
t/2.$$ 
Furthermore, $d_{\cH}(M_i) = |\cT| - w(M_i) \geq t/2$ for every $i\in [t]$ by \eqref{eq: Mi size weight}. 
Thus $\cH$ has a perfect matching. By relabelling, we may assume that $T_1M_1,\dots ,T_{t}M_{t}$ is a perfect matching in $\cH$ and 
hence $W_{T_i}\cap V(M_{i})=\es$ for every $i\in[t]$.

Next, we iteratively construct edge-disjoint embeddings $\tau_i$ so that for all $i\in [t]$ the following hold:
\begin{itemize}
\item[(P1)$_i$] $\tau_i$ embeds $T_i$ into $G[(A\cup V(M_i))\sm W_{T_i}]$ and
\item[(P2)$_i$] $E(\tau_i(T_i))\cap E(G[B])= M_i$.
\end{itemize}

Assume for some $i\in [t]$ we have defined $\tau_1,\dots,\tau_{i-1}$ so that $\tau_j$ satisfies (P1)$_j$ and (P2)$_j$ for all $j<i$. Let $G_i:= G - \bigcup_{j=1}^{i-1} E(\tau_j(T_j))$. 
Consider any $v\in B$.
If
$v \in \tau_j(T_j)$,
then (P1)$_j$ implies that $v \in M_j$. 
Hence,
\begin{align}\label{5.2 Gi deg}
d_{G_i,A}(v) 
\stackrel{(\ref{eq:match})}{\geq} d_{H,A}(v) - 2\delta_2 \Delta n 
\stackrel{\ref{item:L23}}{\geq} \delta_1|A|/25 .
\end{align}

We write $m_i:=|M_i|$ and $M_i= \{u_{1}v_{1},\ldots,u_{m_i}v_{m_i}\}$. 
By \eqref{eq: Mi size weight} and \eqref{5.2 Gi deg},
there is a collection of stars $\{S(u_{j}), S(v_{j})\}_{j\in [m_i]}$ in $G_i$ 
so that all these stars are disjoint from each other and each star has $\Delta$ leaves in $A\sm (W_{T_i}\cup \{\tau'_{i}(y_{T_i})\})$.
Let $u_{j}^1,\ldots,u_{j}^\Delta$ be the leaves of $S(u_j)$ and let $v_{j}^1,\ldots,v_{j}^\Delta$ be the leaves of $S(v_{j})$.

Let $F_i$ be the forest obtained from $T_i$ by first removing all vertices of distance at most $4$ from $y_{T_i}$
and secondly removing all isolated vertices. 
Note that $|F_i| \geq \delta_1 n  - 2\Delta^5 \geq 4\Delta^5 m_i$ by \ref{item:L25} and \eqref{eq: Mi size weight}.
Thus by Proposition~\ref{prop: k-independent set}, $F_i$ has a $5$-independent matching $N=\{x_{1}z_{1},\ldots,x_{m_i}z_{m_i}\}$.
We will construct $\tau_i$ in such a way that edges $x_jz_j$ are mapped to $u_jv_j$ for each $j\in [m_i]$.
To achieve this, 
let $\tau''_i$ be an injective function mapping $N_{T_i}(x_{j})\sm \{z_{j}\}$ to the leaves of $S(u_{j})$ and $N_{T_i}(z_{j})\sm \{x_{j}\}$ to the leaves of $S(v_{j})$ for every $j\in [m_i]$ and so that $\tau''_{i}(y_{T_i}) :=\tau'_{i}(y_{T_i})$.
Let $$J:= \bigcup_{j=1}^{m_i} (N_{T_i}(x_{j})\setminus \{z_{j}\}) \cup \bigcup_{j=1}^{m_i} (N_{T_i}(z_{j})\setminus \{x_{j}\}.$$ 
Note that $J\cup \{y_{T_i}\}$ forms a $3$-independent set in $F_i':=T_i-V(N)$ by the choice of $N$.

For each $v\in A$, by \ref{item:L24}, we obtain
\begin{align*}
d_{G}(v)-d_{G_i}(v) \leq i \Delta \leq t\Delta \leq \Delta\gamma |A|.
\end{align*}
This together with \ref{item:L21} and Proposition~\ref{prop: quasi-random subgraph} implies that $G_i[A \sm W_{T_i}]$ is $(\gamma^{1/11},p)$-quasi-random. 
In particular, $d_{G_i,A\sm W_{T_i}}(u,v) \geq p^2|A|/2 \geq |F'_i| $ for any $u,v\in A\sm W_{T_i}$
(where the final inequality follows from \ref{item:L25}). 
Thus Lemma~\ref{lem: embed forests} implies that there exists a function $\phi_i$ embedding $F_i'$ into $G_i[A \sm W_{T_i}]$
consistent with $\tau_i''$.
Let $\tau_i: V(T_i) \to G_i$ be 
$$\tau_{i}(z):=
\left\{\begin{array}{ll} u_j &\text{ if } z=x_j \text{ for some }j\in[m_i],\\
v_j &\text{ if } z=z_j \text{ for some }j\in[m_i],\\
\phi_i(z) &\text{ otherwise.}\end{array} \right.$$ 
Then $\tau_i$ satisfies (P1)$_i$ and (P2)$_i$.
Let $\tau:=\bigcup_{i=1}^{t} \tau_i.$ By construction, \ref{item:Q21} and \ref{item:Q22} hold.
\end{proof}

Given a collection $\cT$ of trees and a graph $G$ with a vertex partition $(A,B)$
such that $B$ is independent
and $G[A,B]$ has very few edges, 
the next lemma guarantees a packing of $\cT$ into $G$ which covers all edges of $G$ between $A$ and $B$. The purpose of the lemma will be to adjust the parity of the leftover degrees of the vertices in $B$ prior to applying Lemma~\ref{lem: clear delta_1}.

\begin{lemma}\label{lem: clear parity}
Suppose $n,\Delta \in \N\sm\{1\}$ and $1/n \ll \epsilon \ll \delta_1  \ll\gamma\ll 1/\Delta,p\leq 1$.
Let $G$ be a graph on $n$ vertices and $A$ be a set of at least $\gamma^2n$ vertices of $G$.
Let $B:=V(G)\sm A$, $H:=G[A,B]$, and let $\cT$ be a collection of trees with $\Delta(T)\leq \Delta$ for every $T\in \cT$. 
Suppose 
\begin{enumerate}[label=(c3.\arabic*)]
	\item\label{item:L31} $G[A]$ is $(\gamma^{1/10},p)$-quasi-random, 
	\item\label{item:L32} $B$ is an independent set of $G$,
	\item\label{item:L33} $|E(H)| \leq \delta_1 n$ and $|E(H)| +2\epsilon n \leq |\cT| \leq \gamma |A|$,
	\item\label{item:L34} every $T\in \cT$ satisfies $2\Delta^4 \leq |T|\leq p^2 |A|/2$ and has a root $y_T$;
	moreover, there is a set $W_T\sub V(G)$ with $|W_T|\leq \Delta^2$
	and a function $\tau_T':\{y_T\}\to A\sm W_T$, and
  \item\label{item:L35} $|\{T\in \cT: u\in W_{T}\cup \{\tau_{T}'(y_T)\}\}|\leq \epsilon n$ for every $u\in V(G)$.
\end{enumerate}
Then there is a function $\tau$ packing $\cT$ into $G$ which is consistent with $\{\tau_T'\}_{T\in \cT}$ such that 
$E(H)\sub E(\tau(\cT))$ 
and $W_T\cap \tau(T)=\es$ for every $T\in \cT$.
\end{lemma}
\begin{proof}
We write $\cT=\{T_1,\ldots,T_t\}$ and
$E(H)=\{u_1v_1,\dots, u_k v_k\}$ with $u_i\in A$ for all $i\in [k]$. 
Let $\tau'_{i}:=\tau'_{T_i}$.
For each $i\in  [k]$, we choose $j_i\in[t]$ such that $v_i\notin W_{T_{j_i}}, u_i\notin W_{T_{j_i}}\cup \{ \tau'_{j_i}(y_{T_{j_i}})\}$ and $j_1,\dots, j_{k}$ are all distinct. 
This is possible using a greedy approach, since $t \geq k + \epsilon n + \epsilon n$ and by \ref{item:L35}. 
By relabeling, we may assume that $v_i\notin W_{T_i}$ and $u_i\notin W_{T_{i}}\cup \{ \tau'_{i}(y_{T_{i}})\}$. For $k< i \leq t$, let $u_i$ be an arbitrary vertex in $A\sm (W_{T_{i}}\cup \{\tau'_{T_{i}}(y_{T_i})\})$. 

We sequentially construct $\tau_1,\dots, \tau_{t}$
such that for all $i\in [t]$
 \begin{itemize}
 \item[(E1)$_i$] $\tau_i$ is consistent with $\tau'_i$ and embeds $T_i$ into $G_i- W_{T_i}$ where $G_i:= G- \bigcup_{j=1}^{i-1}E(\tau_{j}(T_j))$,
 \item[(E2)$_i$] if $i\in[k]$, then $E(\tau_i(T_i)\cap H) = \{u_iv_i\}$,
 \item[(E3)$_i$] if $i\in[k]$, then $\tau_i(T_i)\setminus \{v_i\}\subseteq A$, and
 if $i>k$, then $\tau_i(T_i)\sub A$.
 \end{itemize}

Assume that for some $i\in [t]$ we have constructed $\tau_1,\dots,\tau_{i-1}$ such that (E1)$_j$--(E3)$_j$ hold for all $j<i$. 
By \ref{item:L33}, $d_{G}(v)-d_{G_i}(v)\leq \Delta t\leq \Delta \gamma |A|$ for all $v\in V(G)$, 
and hence $G_i[A\sm W_{T_i}]$ is $(\gamma^{1/11},p)$-quasi-random by \ref{item:L31} and Proposition~\ref{prop: quasi-random subgraph}.
Let $\ell_i$ be a leaf of $T_i$ which has distance at least $4$ from $y_{T_i}$, and let $x_i$ be the unique neighbour of $\ell_i$ in $T_i$. 
Such a leaf exists, since $|T_i|\geq 2\Delta^4$. 
Let 
$$T'_i:=\left\{ \begin{array}{ll} T_i - \{\ell_i\} &\text{ if }i\in [k], \\
 T_i &\text{ if }i\in [t]\sm [k]. \end{array}\right.$$
Let $\tau''_i(y_{T_i}):=\tau'_{i}(y_{T_i})$ and $\tau''_i(x_i):=u_i$.

Observe that $\{y_{T_i},x_i\}$ forms a $3$-independent set of $T_i'$ and 
for any $u,v\in A\sm W_{T_i}$, 
we have $d_{G_i,A\sm W_{T_i}}(u,v) \geq p^2 |A|/2 \geq |T'_i|$ by \ref{item:L34}.
Apply Lemma~\ref{lem: embed forests} to $T'_i, G_i[A\sm W_{T_i}]$ and $\tau''_{T_i}$ to obtain a function $\phi_i$ 
embedding $T'_i$ into $G_i[A\sm W_{T_i}]$ which is consistent with $\tau_i''$.
Let $\tau_i: V(T_i)\to G_i$ be defined by
\begin{align*}
\tau_i(x) := \left\{ \begin{array}{ll}
v_i &\text{ if } x= \ell_i \text{ and } i\in [k], \\
\phi_i(x)  &\text{ otherwise.}\\
\end{array}\right.
\end{align*}
Then (E1)$_i$--(E3)$_i$ hold. 
Define $\tau := \bigcup_{i=1}^{t} \tau_i$, then $\tau$ is consistent with $\{\tau'_T\}_{T\in \cT}$ and $E(H)\subseteq E(\tau(\cT))$.
\end{proof}

The following lemma is a variant of Lemma~\ref{lem: clear parity} in which $G[A,B]$ is allowed to have more edges than in Lemma~\ref{lem: clear parity}. 
It guarantees a packing of $\cT$ into $G$ which covers all edges between $A$ and $B$ apart from precisely one edge at every vertex of odd degree in $B$. 
In particular, once we have adjusted the parity of the leftover degrees of the vertices in $B$ via Lemma~\ref{lem: clear parity}, 
we can cover all remaining leftover edges via Lemma~\ref{lem: clear delta_1}. 
The proof of Lemma~\ref{lem: clear delta_1} is more difficult than that of Lemma~\ref{lem: clear parity} since the covering is more ``efficient'': while in Lemma~\ref{lem: clear parity} we used one tree for each leftover edge, here each tree covers a linear number of leftover edges.

\begin{lemma}\label{lem: clear delta_1}
Suppose $n,\Delta \in \N\sm\{1\}$ and $1/n\ll \epsilon\ll \delta_1 \ll \gamma \ll 1/\Delta, p\leq 1$.
Let $G$ be a graph on $n$ vertices and $A$ be a set of at least $\gamma^2 n$ vertices of $G$.
Let $B:=V(G)\sm A$, $H:=G[A,B]$, and 
let $\cT$ be a collection of trees with $\Delta(T)\leq \Delta$ for every $T\in \cT$. 
Suppose 
\begin{enumerate}[label=(c4.\arabic*)]
	\item\label{item:L41} $G[A]$ is $(\gamma^{1/10},p)$-quasi-random, 
	\item\label{item:L42} $B$ is an independent set of $G$,
	\item\label{item:L43} $\Delta(H)\leq 2\delta_1 n$,
	\item\label{item:L44} $\delta_1^{1/2} n \leq |\cT| \leq \gamma |A|$, 
	\item\label{item:L45}	every $T\in \cT$ satisfies $\delta_1^{1/4} n\leq |T|\leq p^2|A|/4$ and has a root $y_T$;
	moreover, there is a set $W_T\sub V(G)$ with $|W_T|\leq \Delta^2$ and
	a function $\tau_T':\{y_T\}\to A \sm W_{T}$, and 
	\item\label{item:L46} $|\{T\in \cT: u\in W_{T}\cup \{\tau_T'(y_T)\}\}|\leq \epsilon n$ for every $u\in V(G)$.
\end{enumerate}
Then there is a function $\tau$ packing $\cT$ into $G$ which is consistent with $\{\tau_T'\}_{T\in \cT}$ such that 
\begin{enumerate}[label=(C4.\arabic*)]
	\item\label{item:Q41} writing $H^* := H -  E(\tau(\cT))$, we have $d_{H^*}(v)\leq 1$ for every $v \in B$,
	\item\label{item:Q42} $W_{T}\cap  \tau(T)=\es$ for every $T\in \cT$, and
	\item\label{item:Q43} $d_{\tau(\cT)}(v)$ is even for all $v\in B$.
	\end{enumerate}
\end{lemma}
\begin{proof}
We define 
$B^{\rm odd} := \{v \in B, d_{H}(v) \text{ is odd}\}$ and $t:=|\cT|$.
For each $v\in B^{\rm odd}$, we select an edge $e_v$ joining $v$ and $A$, 
and let $E^{\rm odd}:=\{e_v : v\in B^{\rm odd}\}$. 
We define $G':=G-E^{\rm odd}$.
Our packing $\tau$ of $\cT$ will cover all edges in $G'[A,B]$ (but no edges from $E^{\rm odd}$).
This will ensure that \ref{item:Q41} and \ref{item:Q43} are satisfied.

We let $\cT^2=\{ \tilde{T}_1, T^{\diamond}_1, \tilde{T}_2, T^{\diamond}_2,\dots, \tilde{T}_{t_2}, T^{\diamond}_{t_2}\}$ be a maximal collection of an even number of trees in $\cT$ which have at least $\delta_1^{1/3} n$ leaves.
Let $\cT^{1}=\{ T_1,\dots, T_{t_1}\}:=\cT\sm \cT^2$ such that $T_i$ has at most $\delta_1^{1/3} n$ leaves for all $i\leq t_1-1$.
(The tree $T_{t_1}$ does not play a special role; if it has at least $\delta_1^{1/3}n$ leaves,
we simply cannot include it into $\cT^2$ since $|\cT^2|$ must be even.)
Let $f:= t_1+t_2-\min\{1,t_1\}$.
Next we define a collection $\cF=\{F_1,\ldots,F_f\}$ of forests.
For every $i\in [t_2]$, let
$$F_i:= \tilde{T}_{i} \cup T^{\diamond}_{i}, \enspace W_i':= W_{\tilde{T}_{i}}\cup W_{T^{\diamond}_{i}}, 
\enspace W_i:= W_i' \cup \{\tau'_{\tilde{T}_{i}}(y_{\tilde{T}_{i}}), \tau'_{T^{\diamond}_{i}}(y_{{T}_i^{\diamond}}) \}.$$
For every $t_2< i\leq f$, let
$$F_i:= T_{i-t_2}, \enspace W_i':=W_{i-t_2}, \enspace W_i := W'_i \cup \{\tau'_{T_{i-t_2}}(y_{T_{i-t_2}})\}.$$

Thus if $t_1\geq 1$, then $T_{t_1}$ is the only tree which is not part of any forest in $\cF$.
The proof strategy now is as follows.
We first decompose $G'[A,B]$ into a collection $\cS$ of edge-disjoint flocks of seagulls (recall that these were defined in Section~\ref{sec2.4}).
For every flock of seagulls in $\cS$, 
we use one forest $F\in\cF$ to cover it.
Note that either $F$ has many vertices of degree $2$
or $F$ is the union of two trees both having many leaves.
Accordingly we will cover the vertex of degree $2$ of every seagull by a vertex of degree $2$ in $F$
or by two leaves from two distinct components in $F$.

For every seagull $Z$ in $G$, let $w(Z):=|\{i\in [f]: W_i\cap V(Z)\neq \es\}|$.
By \ref{item:L46}, we conclude $w(Z)\leq 3\epsilon n$.
For every flock $S$ of seagulls, let $w(S) := \sum_{Z\in S} w(Z)$.
Recall that in $G'$ every vertex in $B$ has even degree.
Thus by \ref{item:L43} and Proposition~\ref{prop: bip seagull}, there is a partition of $E(G'[A,B])$ into $6\delta_1 n$ edge-disjoint flocks $S^1,\dots, S^{6\delta_1 n}$ of seagulls with wings in $A$. 
For each $i\in [6\delta_1 n]$, by \ref{item:L45},
we have $w(S^i) \leq (\Delta^2+1)t$.\COMMENT{$w(S_i) \leq \sum_{v\in V(G)}w(v) = \sum_{T\in \cT} |W_{T}\cup \{y_T\}| \leq (\Delta^2+1)|\cT|$.}
Thus for each $i\in [6\delta_1 n]$, we can use Proposition~\ref{prop: weight partition} to partition $S^i$ into 
$q:=\delta_1^{-1/2}/60$ (possibly empty) disjoint flocks $S^i_1,\dots, S^i_{q}$ 
such that $|S^i_{j}|\leq 40\delta_1^{1/2}n$ and $w(S^i_j) \leq 120\delta_1^{1/2}(\Delta^2+1)t + 3\epsilon n \leq \delta_1^{1/2}n/100$
for all $j\in[q]$. 

Note that $q\cdot 6\delta_1 n \leq t/10$ by \ref{item:L44}. By adding empty flocks to $\{S_j^i\}_{i\in [6\delta_1 n],j\in [q]}$ if necessary,
we obtain a collection of edge-disjoint flocks $\cS:=\{S_1,\dots, S_{t/10}\}$ partitioning $E(G'[A,B])$ with
\begin{align}\label{eq:seagulls}
|S_j|\leq 40\delta_1^{1/2} n, \text{ and } w(S_j) \leq \delta_1^{1/2}n/100
\end{align}
for all $j\in [t/10]$.

Let $\cH$ be an auxiliary bipartite graph with vertex partition $(\cF, \cS)$ 
such that $F_iS_j \in E(\cH)$ if $W_i \cap V(S_j)=\es$. 
Since $(t-1)/2 \leq f \leq t$, we have
$$d_{\cH}(S_i) \geq |\cF| - w(S_i) \geq (t-1)/2 - \delta_1^{1/2}n/100 \geq t/10 =  |\cS|.$$ 
Thus $\cH$ contains a matching covering $\cS$. 
By relabelling we may assume that $F_iS_i\in E(\cH)$ for all $i\in [t/10]$. 
Let $\hat{\cS}$ arise from $\cS$ by adding empty flocks $S_{t/10+1},\dots, S_{f}$.

We will greedily construct embeddings $\tau_1,\dots, \tau_f$ such that for each $i\in [f]$ the following hold:
\begin{itemize}
\item[(P1)$_i$] $\tau_i$ packs the components of $F_i$ into $G_i-W_i'$ where $G_i:= G'-\bigcup_{j=1}^{i-1} E(\tau_j(F_j))$,
\item[(P2)$_i$] $E(\tau_i(F_i))\cap E(G'[A,B])= E(S_i)$, and
\item[(P3)$_i$] if $i\in [t_2]$, then $\tau_i$ is consistent with $\tau'_{\tilde{T}_i} \cup \tau'_{T^{\diamond}_i}$ and 
if $t_2<i\leq f$, then $\tau_i$ is consistent with $\tau'_{T_i}$.
\end{itemize}
Note that 
if $i\in[t_2]$, i.e.~if $F_i$ consists of two components $\tilde{T}_i,T_i^\diamond$, 
then $\tau_i(\tilde{T}_i)$ and $\tau_{i}(T_i^\diamond)$ are not necessarily vertex-disjoint, 
thus $\tau_i$ may not be an embedding of $F_i$. 

Suppose for some $i\in [f]$, we have constructed $\tau_1,\dots, \tau_{i-1}$ satisfying (P1)$_j$--(P3)$_j$ for all $j<i$.
We will now construct $\tau_i$.
Since $d_{G, A\sm W_i'}(v)-d_{G_i, A\sm W_i'}(v) \leq \Delta i \leq \Delta t \leq \Delta \gamma |A|$ for each $v\in A$ by \ref{item:L44} 
and $|W_i'|\leq 2\Delta^2$ by \ref{item:L45},
we can apply Proposition~\ref{prop: quasi-random subgraph} to conclude from \ref{item:L41} that $G_i[A\sm W_i']$ is $(\gamma^{1/11},p)$-quasi-random
and so for any two vertices $u,v\in A\sm W_i'$
\begin{align}\label{eq:Giquasirandom}
	d_{G_i,A\sm W_i'}(u,v)\geq p^2 |A|/2 \geq |F_i|.
\end{align}

Let $s:=|S_i|$ and write $S_i = \{u_1v_1w_1, \dots, u_{s}v_{s}w_{s}\}$. 
Let us first consider the case when $s=0$. 
Let $c\in [2]$ be the number of components in $F_i$ and $T^1,\dots,T^c$ be these components.
By \eqref{eq:Giquasirandom}, we can apply Lemma~\ref{lem: embed forests} to each component $T^j$ of $F_i$ to find a function $\phi_{T^j}$ which embeds $T^j$ into $G_i[A\sm W_i']$ 
which is consistent with $\tau'_{T^j}(y_{T^j})$ for each $j\in [c]$ such that the sets $E(\phi_{T^j}(T^j))$ are pairwise edge-disjoint for different $j\in[c]$.
(If $c=2$, then in order to find $\phi_{T^2}$ we apply Lemma~\ref{lem: embed forests} to $G_i[A\sm W_i']-E(\phi_{T^1}(T^1))$.)
Let $\tau_i:= \phi_{T^1}\cup \dots \cup \phi_{T^c}$. 
Then $\tau_i$ satisfies (P1)$_i$--(P3)$_i$.

Suppose next that $s>0$ and $F_i$ is a tree, i.e.~$t_2<i\leq f$.
Thus $F_i$ has at most $\delta_1^{1/3} n$ leaves.
By \eqref{eq:seagulls} and \ref{item:L45}, we have $s\leq 40\delta_1^{1/2} n \leq (|F_i|- 2\delta_1^{1/3}n)/\Delta^{5}$. Thus Proposition~\ref{prop: number of leaves} implies that there exists a $5$-independent set $I:=\{y_{F_i},y_1,\dots, y_{s}\}$ in $F_i$
such that $I\sm \{y_{F_i}\}$ is a set of vertices of degree $2$ in $F_i$.
For each $j\in [s]$, let $\{x_j,z_j\}:=N_{F_i}(y_j)$. 
Then $J:= \{x_1,z_1,\dots,x_{s}, z_{s}\}\cup \{y_{F_i}\}$ is a $3$-independent set in $F_i-(I\sm \{y_{F_i}\})$. 
Let 
\begin{align*}
\phi(v):= \left\{ \begin{array}{ll}
u_j &\text{ if } v=x_j \text{ for some } j\in[s],\\
w_j &\text{ if } v=z_j \text{ for some } j\in[s],\\
\tau'_{F_i}(y_{F_i}) &\text{ if } v= y_{F_i}.
\end{array}\right.
\end{align*}
By \eqref{eq:Giquasirandom}, we can apply Lemma~\ref{lem: embed forests} to find 
a function $\phi'$ which is consistent with $\phi$ and embeds $F_i-(I\sm \{y_{F_i}\})$ into $G_i[A\sm W_i']$.
For every $v\in V(F_i)$,
let
\begin{align*}
\tau_i(v):= \left\{ \begin{array}{ll}
v_j &\text{ if } v=y_j \text{ for some } j\in[s],\\
\phi'(v) &\text{ if } v\in V(F_i)- (I\setminus\{y_{F_i}\}).
\end{array}\right.
\end{align*}
Then $\tau_i$ satisfies (P1)$_{i}$--(P3)$_{i}$.

Suppose next that $s>0$ and $F_i$ is a forest with two components $T^1,T^2$, i.e. $i\in [t_2]$ and $\{T^1,T^2\}=\{\tilde{T}_i,T_i^\diamond\}$.
Thus $T^1,T^2$ both have at least $\delta_1^{1/3} n$ leaves.
For each $c\in[2]$, let 
$$Z_c := \{z \in V(T^c): z\text{ is a neighbour of a leaf}\}\cup \{y_{T^{c}}\}.$$
Hence $|Z_c|\geq \delta_1^{1/3} n/\Delta \geq \Delta^{3} (s+1)$ by \eqref{eq:seagulls}.
Thus Proposition~\ref{prop: k-independent set}
implies that each $Z_c$ contains a $3$-independent set $I_c:=\{y_{T^c},y^c_1,\dots, y^c_{s}\}$. 
For each $j\in [s]$, let $x_j^c$ be a leaf of $T^c$ adjacent to $y_j^c$ and write $X^c:=\{x_1^c,\dots,x_s^c\}$.
Let 
\begin{align*}
\phi(v):= \left\{ \begin{array}{ll}
u_j &\text{ if } v=y^1_j \text{ for some } j\in[s],\\
w_j &\text{ if } v=y^2_j \text{ for some } j\in[s],\\
\tau'_{T^c}(y_{T^c}) &\text{ if } v= y_{T^c} \text{ for some } c\in[2].
\end{array}\right.
\end{align*}
If $\tau'_{T^1}(y_{T^1})= \tau'_{T^2}(y_{T^2})$, then let $F'_i$ be the tree we obtain from $F_i$ by identifying $y_{T^1}$ and $y_{T^2}$. 
Otherwise let $F'_i:=F_i$.  
By \eqref{eq:Giquasirandom}, 
we can apply Lemma~\ref{lem: embed forests} to find
a function $\phi'$ which is consistent with $\phi$ and embeds $F'_i-(X^1\cup X^2)$ into $G_i[A\sm W_i']$.
Let
\begin{align*}
\tau_i(v):= \left\{ \begin{array}{ll}
v_j &\text{ if } v\in \{x_j^1,x_j^2\}\text{ for some } j\in[s],\\
\phi'(v) &\text{ if } v\in V(F_i)- (X^1 \cup X^2).
\end{array}\right.
\end{align*}
Then $\tau_i$ satisfies (P1)$_{i}$--(P3)$_{i}$.

Suppose now that we have defined $\tau_1,\dots,\tau_f$.
If $t_1= 0$, 
then $\tau:= \bigcup_{i=1}^{f} \tau_i$ satisfies \ref{item:Q41}--\ref{item:Q43}.
If $t_1\geq 1$, then there is a single tree $T\in \cT$ which is not yet embedded.
In this case 
let $G_{f+1}:= G'-\bigcup_{j=1}^{f} E(\tau_j(F_j))$.
We once again use Lemma~\ref{lem: embed forests} to find a function $\tau_{f+1}$ which embeds $T$ into $G_{f+1}[A\sm W_T]$ and is consistent with $\tau'_{T}(y_T)$.
Then $\tau:= \bigcup_{i=1}^{f+1} \tau_i$ satisfies \ref{item:Q41}--\ref{item:Q43}.
\end{proof}

\section{Iteration Lemma}
\label{sec:iteration}
In this section we state and prove the key lemma of this paper.
Given a suitable graph $G^*$ on a vertex set $V$, 
a small set $A\sub V$, and a suitable set of forests,
we can cover all edges of $G^*$ incident to the vertices in $V\sm A$ with these forests without using many edges of $G^*$ inside $A$.
In Section~\ref{sec:final} we will apply this lemma iteratively to obtain our main result. More precisely, the above graph $G^*$ will be the union of graphs $G, G', H_1$ and $H_2$, where $G$ will be the ``leftover'' from the previous iteration step, and $H_1$ and $H_2$ are graphs we set aside at the start to ensure that each iteration can be carried out successfully. We will cover $G\cup H_1\cup H_2$ entirely (see \ref{item:Phi1}). Roughly speaking the leftover of $G'$ from the current iteration step will play the role of $G$ in the next iteration step, which is why we aim to use it as little as possible in the current iteration (see \ref{item:Phi2}).

\begin{lemma}\label{lem: iteration}
Suppose $D,\Delta,n\in \N$ with  $1/n \ll \epsilon \ll  \delta_2 \ll \delta_1 \ll \gamma_* \ll  \gamma \ll \beta\ll \alpha, 1/D, 1/\Delta,p,d\leq 1$ 
such that $D\geq 6$ is even and $\Delta\geq 3$. 
Suppose $G, G', H_1, H_2$ are four edge-disjoint graphs on a set $V$ and $A,R \sub V$ satisfying the following:
\begin{enumerate}[label=(g\arabic*)]
\item\label{item:G1} $|V|=n$, $|A|=\gamma n$, $|R|= \epsilon n$, $A\cap R=\emptyset$, and $A$ is an independent set in $G \cup H_1 \cup H_2$,
\item\label{item:G2} $G$ is $(\beta,\alpha)$-quasi-random with $V(G)=V$ and $G'$ is $(\gamma^{1/3},p)$-quasi-random with $V(G')= A$,
\item\label{item:G3} $H_1$ is a $(4\epsilon,\delta_1)$-quasi-random bipartite graph with vertex partition $(A,V\setminus A)$, and
\item\label{item:G4} $H_2$ is a graph on $V$ with $\Delta(H_2)\leq 3\delta_2 n/2$ such that $d_{H_2}(u,v)\geq 2\delta_2^2 n/3$ for any two distinct vertices $u,v \in V$.
\end{enumerate}
Suppose $\cF$ is a collection of rooted forests so that each $F\in \cF$ consists of two components $(T_F^1,r^1_F), (T_F^2,r^2_F)$ and $\cF$ satisfies the following:
\begin{enumerate}[label=(f\arabic*)]
\item\label{item:T1} $|\cF| \geq  (1/2 + d)n $,
\item\label{item:T2} $|F|\leq (1- {4}/{D}) n$, $\Delta(F)\leq \Delta$, and $|T_F^1|, |T_F^2| \geq \gamma n$ for all $F\in \cF$,
\item\label{item:T3} $ e(\cF) = e(G)+e(H_1)+e(H_2) + (3\gamma_*^2 \pm \gamma_*^2 )n^2$, and
\item\label{item:T4} there is a function $\phi':\{r^c_F: F\in \cF, c\in [2]\}\to R$ such that $|\phi'^{-1}(v)| \leq \epsilon^{-2}$ for any $v\in R$ and $\phi'(r^1_F)\neq \phi'(r^2_F)$ for any $F\in \cF$.
\end{enumerate}
Then there exists a function $\phi$ which is consistent with $\phi'$ and which packs $\cF$ into $G\cup G'\cup H_1\cup H_2$ such that
\begin{enumerate}[label=($\Phi$\arabic*)]
\item\label{item:Phi1} $E(G)\cup E(H_1)\cup E(H_2) \sub E(\phi(\cF))$ and
\item\label{item:Phi2} for every $v\in A$, we have $d_{G'\cap \phi(\cF)}(v) \leq \gamma_*^{1/2}|A|$.
\end{enumerate}
\end{lemma}

We structure the proof into several steps.
In the ``preparation step'', Step~\ref{step1}, 
we partition $\cF$ into several subsets. 
We also apply Lemma~\ref{lem: decomp blow ups} to obtain an approximate decomposition of $G$ into suitable cycle blow-ups.
Unfortunately, Lemma~\ref{lem: decomp blow ups} also returns a small exceptional set $V_0$ of vertices which is not part of the cycle blow-ups.
Furthermore,
we partition each forest $F\in \cF$ into a large subforest $\tilde{F}$ and a small subforest $F-V(\tilde{F})$.

In Step~\ref{step2}, 
we apply Lemma~\ref{lem: blow up advanced} to the cycle blow-ups obtained in Step~\ref{step1} 
to approximately cover the edges in $E(G)$ with the set of large subforests. 
Crucially, the density of the leftover of $G$ (i.e.~the uncovered part) is insignificant compared to the density of $H_2$.
In this step we will also use some edges of $H_1$ in order to satisfy the restrictions on the packing given by \ref{item:T4}.
In the following steps we cover the remaining edges with the small subforests.
In order to avoid an overlap between the embeddings of the small and large subforests, 
we equip $V$ with an equitable partition $(U_1,\ldots,U_D)$ and make sure that for each $F\in \cF$ the large subforest $\tilde{F}$ of $F$ uses at most $D-3$ parts $U_i$
while the remaining three parts are reserved for the small subforest $F-V(\tilde{F})$ of $F$.
In order to maintain a symmetric setting, 
we actually partition $\cF$ into $\binom{D}{3}$ subsets and carry out the above procedure for every choice of $D-3$ sets in $(U_1,\ldots,U_D)$.

In Step~\ref{step3},
we prepare a link between the large and the small subforests by embedding the roots of each small subforest $F-V(\tilde{F})$ 
(i.e.~those vertices of $F-V(\tilde{F})$ which attach to $\tilde{F}$) into some $U_i$ 
which is not used by its corresponding large subforest $\tilde{F}$.

In the remaining steps we make use of the small subforests in the  following way.
In Step~\ref{step4}, we use Lemma~\ref{lem: clear rl waste} to cover almost all edges incident to the exceptional set $V_0$.
In Step~\ref{step5},
we cover all the edges induced by $V\sm A$ via Lemma~\ref{lem: clear delta2}.
In Step~\ref{step6}, we make use of Lemma~\ref{lem: clear parity} to ensure that the number of uncovered edges incident to each vertex in $V\sm A$ is even.
In Step~\ref{step7}, we use Lemma~\ref{lem: clear delta_1} to cover all remaining edges between $A$ and $V\sm A$. In each of these four steps we make use of the fact that the leftover of $G$ forms an insignificant part of the total current leftover.

\begin{proof}[Proof of Lemma~\ref{lem: iteration}]
We start with the preparation step.

\step{Preparation}\label{step1}
In this step, we partition $G,G',H_1,H_2$, 
the underlying vertex sets, 
and $\cF$ into structures which are suitable for applications of the packing and embedding lemmas proved in Section~\ref{sec: trees} and \ref{sec:cleaning}.
We additionally choose $M,M'\in \N$ and a constant $\eta$ such that $1/n\ll 1/M \ll 1/M' \ll \epsilon\ll \eta \ll \delta_2$. Let $\hat{D}:=\binom{D}{3}$ and
let $\cI:= \binom{[D]}{D-3} =\{ S_1,\dots, S_{\hat{D}}\}$. Thus $|\cI|=\hat{D}$.

We select an equitable partition $(U_1,\dots, U_D)$ of $V\sm R$ satisfying the following:
\begin{enumerate}[label=(U\arabic*)]
\item\label{item:U1} $(U_1\cap A,\ldots,U_D\cap A)$ is an equitable partition of $A$, and
\item\label{item:U2} for every $J\in \{G,G',H_1,H_2\}$, every $u,v\in V$ and $i,j\in [D]$, we have 
\begin{itemize}
\item[(i)] $|d_{J,U_i}(u)- d_{J,U_j}(u)|\leq \epsilon n$, 
\item[(ii)] $|d_{J,U_i\cap A}(u)- d_{J,U_j\cap A}(u)|\leq \epsilon n$, and
\item[(iii)] $|d_{J,U_i}(u,v)- d_{J,U_j}(u,v)|\leq \epsilon n$.
\end{itemize}
\end{enumerate}
Such an equitable partition exists as an equitable partition of $V\setminus R$
chosen uniformly at random subject to \ref{item:U1} satisfies \ref{item:U2} with probability at least $1/2$ by Lemma~\ref{lem: chernoff}.%
\COMMENT{We take $U_1,\dots, U_D$ satisfying (U1) uniformly at random.
For any set $S$ of size at least $\epsilon n$, $$\mathbb{E}[|U_i\cap S|] = D^{-1}|S|,$$ and $\mathbb{P}[|U_i\cap S| = D^{-1}|S|\pm \epsilon n/2] \geq 1-(1-c)^n$. We take union bound for $S= N_{H'}(v)$ or $S=N_{H'}(u,v)$ or $N_{H'}(v)\cap A, N_{H'}(u,v)\cap A$ then we get this with probability at least $1- 16 D n^2(1-c)^n$.} 
For a set $S\subseteq [D]$, we define 
$$U_{S}:=\bigcup_{i\in S} U_i.$$
For each $\I\in [\hat{D}]$, let 
\begin{align}\label{eq:defAB}
	A_\I:= A\cap U_{[D]\setminus S_\I}\text{ and }B_\I:= (U_{[D]\setminus S_\I}\cup  R)\setminus A.
\end{align}
We make the following observation
(which follows from \ref{item:G3}, \ref{item:U1}, \ref{item:U2} and the fact that $|R|=\epsilon n$).
\begin{equation}\label{eq: H1 quasi-random}
\begin{minipage}[c]{0.9\textwidth}\em
For all $i\in [D]$ and $\I\in [\hat{D}]$,
the bipartite graphs $H_1[U_i]$, $H_1[U_i\cup R]$, $H_1[ A_\I\cup B_\I]$ and $H_1[A_\I\cup (B_\I\setminus R)]=H_1[U_{[D]\setminus S_I}]$ are $( \epsilon^{1/2},\delta_1)$-quasi-random.
\end{minipage}
\end{equation}
Observe that $G- R$ is $(2\beta,\alpha)$-quasi-random by \ref{item:G2}, and thus $(\beta^{1/7},\alpha/2)$-dense by Proposition~\ref{prop: quasi-random implies dense}. 
Thus we can apply Lemma~\ref{lem: decomp blow ups} with the following graphs and parameters. \newline

\noindent
{
\begin{tabular}{c|c|c|c|c|c|c|c|c|c|c}
object/parameter & $G- R$ & $\epsilon/(1- \epsilon)$ & $\delta_2^{-4}$ & $\beta^{1/7}$ & $\alpha/2$ & $\alpha$ & $D$ & $(U_1,\dots, U_D)$&
$M$ &$M'$\\ \hline
playing the role of & $G$ & $\epsilon$ & $t$ & $\beta$ & $\alpha$ & $d$ & $D$& $(U_1,\dots,U_D)$  &$M$&$M'$
\end{tabular}
}\newline \vspace{0.2cm}

\noindent
Therefore, 
there exist an exceptional set $V_0\sub V\sm R$, with 
\begin{align}\label{eq: V0 size}
|V_0|\leq 2\epsilon n,
\end{align}
a positive integer $\Gamma$ with $M'\leq \Gamma \leq M$, 
and collections of cycle blow-ups $\cC_S$ (one for each $S\in \cI$) such that the following properties hold, 
where we write $\cC:= \bigcup_{S\in \cI}\cC_S$:
\begin{enumerate}[label=(R$'$\arabic*)]
	\item\label{item:R'1} for every $S\in \cI$, 
	the set $\cC_{S} $ is a set of $(3\epsilon,\delta_2^4)$-super-regular cycle $(1- 2\epsilon)\frac{n}{\Gamma}$-blow-ups
	such that $C\sub G$ for each $C\in \cC_S$,
	the length of each $C\in \cC_S$ is odd and at least $(1- {7}/{(2D)})\Gamma$
	and such that $V(C)\sub U_S\sm V_0$ for each $C\in \cC_S$,
	\item\label{item:R'2a} all cycle blow-ups in $\cC$ are edge-disjoint from each other,
	\item\label{item:R'2} $e( \cC_S) = (1\pm \beta^{1/14})\hat{D}^{-1} e(G)$ for all $S\in \cI$,
	\item\label{item:R'3}  $d_G(v)-\sum_{C\in \cC}d_{C}(v)\leq 4\delta_2^2n$ for $v\in V\sm (V_0 \cup R)$.\COMMENT{
	$3\delta_2^2n+\epsilon n \leq 4\delta_2^2n$ instead of $3\delta_2^2n$, since we apply Lemma~\ref{lem: decomp blow ups} to $G-R$.}
	In particular, $ e(\cC)\geq e(G) - 2\delta_2^2 n^2$.
\end{enumerate}
\COMMENT{In order to obtain \ref{item:R'1} we use that $(1-\epsilon/(1-\epsilon))(1-\epsilon)n = (\frac{1-2\epsilon}{1-\epsilon})(1-\epsilon)n = (1-2\epsilon)n.$}

Because we split $G- R$ into $\hat{D}$ collections $\cC_S$ of graphs, 
we also seek a corresponding partition of $\cF$ into $\cF_{1},\dots, \cF_{\hat{D}}$ such that 
\begin{enumerate}[label=(F\arabic*)]
\item\label{item:F1} $|\cF_{\I}|\geq (1/2+d/2)\hat{D}^{-1}n$ for all $\I\in [\hat{D}]$,
\item\label{item:F2} $e(\cF_{\I}) = e(\cC_{S_\I}) + (3\gamma_*^2 \pm 2\gamma_*^2)\hat{D}^{-1} n^2$ for all $\I\in [\hat{D}]$.
\end{enumerate}
It is not difficult so see that such a partition exists.
Indeed, for all $\I\in [\hat{D}]$, 
assign $F\in \cF$ to $\cF_\I$ with probability $e(\cC_{S_\I})/e(\cC)$, independently of all other $F'\in \cF$.
Then straightforward applications of Theorem~\ref{Azuma} show that this random partition satisfies \ref{item:F1} and \ref{item:F2} with probability at least $1/2$.\COMMENT{
Fix some $c$ small enough.
Let $F_1,\dots, F_{|\cF|}$ be the forests and for each $j\in [|\cF|]$, let $x_j \in [\hat{D}]$ be a random variable such that $\mathbb{P}[x_j=\I] = e(\cC_{S_\I})/e(\cC)$ such that $x_1,\dots, x_{|\cF|}$ are independent random variables.
And we assign $F_j$ into $\cF_{x_j}$.
We fix $\I\in [\hat{D}]$ consider exposure martingales $X_j:= \mathbb{E}[e(\cF_\I)\mid x_1,\dots, x_{j-1}]$ and $Y_j:= \mathbb{E}[|\cF_\I|\mid x_1,\dots, x_{j-1}]$.
Note that
\begin{align*}
	e(\cF)=e(G)+e(H_1)+e(H_2)+(3\gamma_*^2\pm\gamma_*^2)n^2
	= e(G) +(3\gamma_*^2\pm \gamma_*^2 \pm \gamma_*^3)n^2
	=e(\cC)+(3\gamma_*^2\pm \gamma_*^2 \pm 2\gamma_*^3)n^2.
\end{align*}
Then 
\begin{align*}
	X_1 
	= e(\cC_{S_\I}) + (3\gamma_*^2\pm\gamma_*^2\pm 2\gamma_*^3)n^2\frac{e(\cC_{S_\I})}{e(\cC)}
	= e(\cC_{S_\I}) + (3\gamma_*^2\pm\gamma_*^2\pm 2\gamma_*^3)(1\pm \beta^{1/15})n^2/\hat{D}= e(\cC_{S_\I}) + (3\gamma_*^2\pm 3\gamma_*^2/2)n^2/\hat{D},\\
	Y_1= \frac{e(\cC_{S_\I})}{e(\cC)}|\cF| 
	=
	\frac{(1\pm \beta^{1/15})}{\hat{D}}|\cF| \stackrel{({\rm f}1)}{\geq} (1+ 2d/3)\hat{D}^{-1} n.
\end{align*}
Also note that $X_i$ is $n$-Lipschitz and $Y_i$ is $1$-Lipschitz. 
Thus by Theorem~\ref{Azuma}, and the fact that $|\cF|\leq \alpha n/\gamma$ by \ref{item:T2},\ref{item:T3}
$$\mathbb{P}[X_{|\cF|}  
= e(\cC_{S_\I})+ (3\gamma_*^2\pm 2\gamma_*^2)n^2/\hat{D}] 
\geq 1 - 2 e^{-\gamma_*^4 n^4/(8\hat{D}^2n^2|\cF|)} 
\geq 1- (1-c)^n.$$ 
Also $$\mathbb{P}[Y_{|\cF|} \geq (1+d/2)\hat{D}^{-1}n]\geq 1-(1-c)^n.$$
Thus taking union bound gives us the desired partition.
}

We will cover the edges of $G,H_1,H_2$ in several steps, mainly via applications of 
Lemma~\ref{lem: blow up advanced} and
Lemma~\ref{lem: clear rl waste}--\ref{lem: clear delta_1}.
For the applications of Lemma~\ref{lem: clear rl waste}--\ref{lem: clear delta_1} 
we need to reserve an appropriate collection of forests.
Accordingly, for every $\I\in [\hat{D}]$, 
we partition $\cF_\I$ into sets $\cF_{\I}^{\eta}, \cF_{\I}^{\delta_2}, \cF_{\I}^{\delta_1}, \cF_{\I}^{\p},\cF_{\I}^0$ 
such that  
\begin{align}\label{eq: cF sizes}
|\cF_{\I}^{\eta}| = 3\left\lceil \frac{n}{6\hat{D}} \right\rceil, \quad 
|\cF_{\I}^{\delta_2}|=\delta_2^{1/2} n , \quad
|\cF_{\I}^{\delta_1}|=|\cF_{\I}^{\p}|=\gamma_* n
\end{align} and let 
$$\cF_{\I}^{\gamma_*}:= \cF_{\I}^{\delta_2}\cup \cF_{\I}^{\delta_1}\cup \cF_{\I}^{\p} , 
\enspace \cF_{\I}^0:= \cF_\I \setminus (\cF_\I^\eta\cup \cF_\I^{\gamma_*}) \enspace \text{and}
\enspace \cF^{\xi}:= \bigcup_{\I=1}^{\hat{D}} \cF^{\xi}_{\I}\text{ for }\xi \in \{\eta,\delta_2,\delta_1,\p,\gamma_*,0\}.$$
Such a partition exists by \ref{item:F1} and since 
$$(2\gamma_*+ \delta_2^{1/2})n + 3\left\lceil \frac{n}{6\hat{D}}\right\rceil  \leq (1/2+d/2)\hat{D}^{-1}n.$$
Note that 
\begin{align}\label{eq:sizeFgamma}
	2\gamma_*n \leq |\cF_\I^{\gamma_*}|\leq 3\gamma_* n.
\end{align}
We will use 
subforests of the forests in $\cF^\eta$ to cover almost all the edges at the exceptional set $V_0\cup R$ in Step~\ref{step4},
subforests of the forests in $\cF^{\delta_2}$ to cover all remaining edges inside $V\sm A$ in Step~\ref{step5} (these edges will consist mainly of edges in $H_2$).
We will use subforests of the forests in $\cF^\eta \cup \cF^\p$ to resolve parity issues in Step~\ref{step6}, and
subforests of the forests in $\cF^{\delta_1}$ to cover all remaining edges between $A$ and $V\sm A$ in Step~\ref{step7} (these edges will consist mainly of edges in $H_1$).
The forests in $\cF^0$ are simply the forests which play no specific role in Step~\ref{step4}--\ref{step7}.

In Step~\ref{step2}, we will use Lemma~\ref{lem: blow up advanced} to embed the bulk of $\cF$ into the cycle blow-ups in $\cC$.
Accordingly, for every $F\in \cF$ we now define a ``large'' subforest $\tilde{F}$ of $F$ 
which is embedded by Lemma~\ref{lem: blow up advanced}.
(The remaining subforests are used to cover the remaining edges by applications of Lemma~\ref{lem: clear rl waste}--\ref{lem: clear delta_1}, as described above.)
For every $F\in \cF$, 
we will also define a set $X_F$ containing all those vertices of $\tilde{F}$ which attach to $F- V(\tilde{F})$.

More precisely, for each $\I\in [\hat{D}]$, 
we proceed as follows. 
For each $F\in \cF^0_{\I}$, we define $\tilde{F} := F$ and $X_F:=\emptyset$. 

We now construct $\tilde{F}$ and $X_F$ for all $F\in \cF_\I^{\gamma_*}$.
Let 
\begin{align}
	\beta_{\I}:= \frac{e(\cF_{\I})- (1-\delta_2^2)e(\cC_{S_{\I}})}{|\cF_{\I}^{\gamma_*}| n}
	\stackrel{\ref{item:F2}}{=}\frac{(3\gamma_*^2 \pm 5\gamma_*^2/2)\hat{D}^{-1}n^2}{|\cF^{\gamma_*}_{\I}|n}.
\end{align}
Then \eqref{eq:sizeFgamma} implies 
$2\Delta \gamma_*^2 \leq \beta_{\I} 
\leq \gamma_*.$  
Recall from \ref{item:T2} that each $F\in \cF$ consists of two components $(T_F^1,r_F^1),(T_F^2,r_F^2)$ of order at least $\gamma n$.
Thus for all $F\in \cF_{\I}^{\gamma_*}$ we may apply Proposition~\ref{prop: taking subtree collection} with $1,\beta_\I, \gamma$ playing the roles of $c,\beta,\alpha$ to choose a subtree $T_F^1(y_F)=:T^{\gamma_*}_F$ of $F$ such that $T_F^{\gamma_*}$ has distance at least $3$ from $r^1_F$,
\begin{align}\label{eq:sizeTgamma}
	\gamma_*^2 n \leq |T^{\gamma_*}_F| \leq  \Delta\gamma_* n,
\end{align}
and 
\begin{align}\label{eq: gamma_* def}
\sum_{F\in \cF_{{\I}}^{\gamma_*}} e(T^{\gamma_*}_F) = e(\cF_{{\I}})- (1-\delta_2^2)e(\cC_{S_\I}) \pm n. 
\end{align}
For each $F\in \cF^{\gamma_*}_\I$, we let $\tilde{F} := F - V(T^{\gamma_*}_F),$ 
and let $x_F$ be the unique vertex in $\tilde{F}$ which is adjacent to $y_F$ in $F$. 
We view $y_F$ as the root of $T^{\gamma_*}_F$ and
let $X_F:=\{x_F\}$.

Next we consider any forest $F\in \cF^{\eta}_\I$.
By Proposition~\ref{prop: taking subtree collection} (applied with $2, \eta, \gamma$ playing the roles of $c,\beta,\alpha$), for every $c\in [2]$ we can choose a subtree $T_F^c(y_F^c)$ of $T_F^c$
with 
\begin{align}\label{eq:sizeTFc}
	\eta n/(2\Delta)\leq |T_F^c(y_F^c)| \leq \Delta\eta n
\end{align}
such that 
the distance from $r^c_F$ to $y_F^c$ is at least $5$.
Again, we view $y^c_F$ as the root of $T_F^c(y_F^c)$.
Let $\tilde{F} := F - V(T_F^1(y_F^1) \cup T_F^2(y_F^2))$.
For each $c\in[2]$, let $x^c_F$ be the unique vertex in $\tilde{F}$ which is adjacent to $y^c_F$ in $F$ and define $X_F := \{x^1_F, x^2_F\}$.
In addition, for all $c\in [2]$ and $F\in\cF^{\eta}_\I$, 
we choose one leaf $\ell^c_F$ of $T_F^c(y_F^c)$ at distance at least $6$ from $y^c_F$ and
let $z^c_F$ be the neighbour of $\ell^c_F$.
Hence
\begin{equation}\label{eq:5indset}
\begin{minipage}[c]{ 0.8\textwidth}\em
$\{y_F^1,z_F^1,y_F^2,z_F^2\}$ is a 5-independent set in $F$.
\end{minipage}
\end{equation}
Embedding $\ell^c_{F}$ appropriately in a separate step will help us to change the parity of its image vertex in the leftover graph if required. 
For this it is essential that $\ell^{c}_{F}$ is in fact also a leaf of the original forests $F\in \cF^\eta$. 
This is also the reason why we need to remove two subtrees in this case (rather than one): 
In total we need to reserve $2|\cF^\eta|$ leaves in order to be able to deal with possible parity problems later on,
but if $F$ is a path, then we cannot find a single small subtree which contains both leaves of $F$.

Note that in all cases we have $|X_F|\leq 2$ and that $X_F\cap (N_F[r^1_F] \cup N_F[r^2_F])=\emptyset.$
Recall that $\cF_\I^{\gamma_*}=\cF_\I^{\delta_2}\cup \cF_\I^{\delta_1} \cup \cF_\I^{\p}$.
For all $\xi \in \{\delta_2, \delta_1, \p\}$ and $F\in \cF_\I^\xi$, let $T_F^\xi:= T_F^{\gamma_*}$.
For easier referencing, 
given $\xi\in\{\eta,\gamma_*,\delta_2,\delta_1,\p\}$ and $F'\sub F\in \cF_\I^\xi$,
we sometimes write $y_{F'}$ for $y_F$.
For example, if $\xi\in\{\delta_2,\delta_1,\p\}$ and $F'=T_F^\xi$,
then this allows us to directly refer to $y_F$ as $y_{F'}$.
Analogous conventions apply to $x_F,x^1_F, x^2_F,y_F^1,y_F^2,z_F^1,z_F^2,\ell_F^1,\ell_F^2$ accordingly.

We introduce the following families of forests for every ${\I}\in [\hat{D}]$. Recall that the first three families refer to the ``large'' subforests and the other families to the ``small'' subtrees and forests.
\begin{itemize}
\item $\tilde{\cF}^{\xi}_{{\I}} := \{\tilde{F}: F\in \cF^{\xi}_{\I}\}$ for $\xi \in \{\eta,\gamma_*,\delta_2,\delta_1,\p,0\}$, 
\item $\tilde{\cF}_{{\I}}:= \{\tilde{F}: F\in \cF_{\I}\}$,
\item $\tilde{\cF}:= \bigcup_{{\I}=1}^{\hat{D}} \tilde{\cF}_{{\I}}=\{\tilde{F}:F\in \cF\}$,
\item $\sT^{\xi}_{\I} :=\{T^\xi_F: F\in \cF^\xi_\I\}$ for $\xi\in \{\gamma_*,\delta_2,\delta_1,\p\}$, 
\item $\sF^{\eta}_{\I} :=\{(T_F^1(y_F^1)-\ell_F^1) \cup (T_F^2(y_F^2)-\ell_F^2): F\in \cF^\eta_{\I}\}$,
\item $\sT^{\xi}:= \bigcup_{{\I}=1}^{\hat{D}} \sT^{\xi}_{\I}$ for $\xi\in \{\gamma_*,\delta_2,\delta_1,\p\}$,
\item $\sF^{\eta}:= \bigcup_{{\I}=1}^{\hat{D}} \sF^{\eta}_\I$.
\end{itemize}
The set $\tilde{\cF}$
contains all forests which we will embed in Step~\ref{step2} using Lemma~\ref{lem: blow up advanced}. The forests in $\tilde{\cF}$ inherit their roots from their corresponding supergraphs in $\cF$.
As every component of $F\in \cF$ has at least $\gamma n$ vertices (by \ref{item:T2}) and $e(\cF)\leq n^2$ (by \ref{item:T3}),
we also conclude
\begin{align}\label{eq: cF' size}
|\tilde{\cF}|=|\cF| \leq \gamma^{-1} n.
\end{align}
For later reference we note that 
\begin{align}\label{eq: size of sF eta}
|\sF_\I^{\eta}| = 
|\cF_\I^{\eta}| \stackrel{(\ref{eq: cF sizes})}{=} 
3 \left\lceil \frac{n}{6\hat{D}}\right\rceil
=(1\pm \epsilon) \frac{n}{2\hat{D}}.
\end{align}
Moreover, \ref{item:T2}, \eqref{eq:sizeTgamma}, and \eqref{eq:sizeTFc} imply that
\begin{align}\notag
	\gamma n \leq |\tilde{F}| \leq (1-4/D)n \quad &\text{ for all }\tilde{F}\in \tilde{\cF},\\ \label{eq:forestssizes}
	\gamma_*^2 n \leq |T| \leq \Delta\gamma_*n \quad &\text{ for all }T\in \sT^{\xi} \text{ and } \xi\in \{\gamma_*,\delta_2,\delta_1,\p\},\\ \notag
	\eta n/(3\Delta) \leq |T| \leq \Delta\eta n \quad &\text{ for every component }T \text{ of each forest in }\sF^\eta.
\end{align}

By \eqref{eq: size of sF eta} and \eqref{eq:forestssizes},
for all $\I\in\hat{D}$, we have $e(\tilde{\cF}^{\eta}_{\I} ) \geq e(\cF^{\eta}_{\I}) - 2\Delta\eta n|\cF^\eta_\I|\geq e(\cF^{\eta}_{\I}) - \eta^{2/3}n^2.$ 
By \ref{item:G2} and \ref{item:R'2}, 
we conclude $e(\cC_{S_{\I}})\geq  (1- \beta^{1/14})(1-\beta)\hat{D}^{-1}\alpha n^2/2  \geq \beta n^2$.
Thus
\begin{align}\label{eq:tildeFsize}
e(\tilde{\cF}_{{\I}})
= \left(e(\cF_\I^{\gamma_*})- \sum_{F\in \cF_\I^{\gamma_*}}(e(T_F^{\gamma_*})+1) \right)+ (e(\cF_\I^{\eta})\pm \eta^{2/3}n^2)+e(\cF_\I^{0})
\stackrel{(\ref{eq: gamma_* def})}{=} (1-\delta_2^2 \pm \eta^{1/2})e(\cC_{S_{\I}}).
\end{align}

Next we need to decide which forests will be embedded into which cycle blow-up in $\cC$.
For this, for each ${\I}\in [\hat{D}]$ we proceed as follows. 
We write $\cC_{S_\I} =:\{ C_{{\I},1},\dots, C_{{\I},c'_\I}\}$. 
Note that by \ref{item:R'1} for every $i\in [c'_{\I}]$, 
\begin{align}\label{eq: number of edges in cCij}
\frac{\delta_2^4 n^2 }{3\Gamma} \leq e(C_{{\I},i}) \leq \frac{2\delta_2^4 n^2 }{\Gamma}.
\end{align}
Together with \ref{item:R'2a}, this implies that 
\begin{align}\label{eq:sizecC}
|\cC| \leq 2\delta_2^{-4} \Gamma.
\end{align}
Moreover, \eqref{eq:tildeFsize} and \eqref{eq: number of edges in cCij} imply that we can select a partition of $\tilde{\cF}_{{\I}}$ into $\tilde{\cF}_{{\I},1},\dots, \tilde{\cF}_{{\I},c'_{{\I}}}$ such that for all $i\in [c'_{I}]$
\begin{align}\label{eq: cFij cCij size}
e(\tilde{\cF}_{{\I},i}) = (1-\delta_2^2 \pm 2\eta^{1/2}) e(C_{{\I},i}).
\end{align}
Finally, \eqref{eq:forestssizes}, \eqref{eq: number of edges in cCij}, and \eqref{eq: cFij cCij size} allow us to conclude that 
for all ${\I}\in[\hat{D}]$ and $i\in [c_{\I}']$
\begin{align}\label{eq: F'ij size}
|\tilde{\cF}_{{\I},i}|\leq \frac{2\delta_2^4 n^2}{\Gamma} \cdot \frac{1}{\gamma n} \leq \frac{n}{\Gamma}.
\end{align}

\step{Approximate covering}\label{step2}
In this step, we find an approximate covering of the edges of $G$. 
To make the leftover sufficiently small, we will actually do this by finding an approximate covering of the cycle blow-ups $C_{I,i}$ defined in Step~1 (which themselves cover almost all edges of $G$). 
Recall that assumption \ref{item:T4} prescribes an embedding $\phi'$ of the roots of all $F\in \cF$
and thus of the roots of all $\tilde{F}\in \tilde{\cF}$.
Ideally, we want to pack the forests $\tilde{F}\in \tilde{\cF}_{{\I},i}$ without their roots into $C_{{\I},i}$
in such a way that the edges incident to their roots are present in $H_1$.
However, one difficulty is that for a root $r$ of $F$, the vertex $\phi'(r)$ might not have neighbours in $V(C_{{\I},i})$.
Therefore, we introduce an intermediate step.
We only pack the forests without their roots and without the neighbours of the roots into $C_{{\I},i}$
and embed the neighbours of the roots onto appropriate vertices to complete the packing.
We arrange this so that all edges incident to a neighbour of a root are mapped to edges in $H_1$.

More precisely,
recall that we denote the roots of each $F\in \cF$ by $r_F^1$ and $r_F^2$.
We will now extend the domain of the function $\phi'$
to 
$$\{p_1,\dots, p_m\}:=\{ x \in N_F(r^c_F): c\in [2], F\in \cF\}.$$
We will achieve this via Proposition~\ref{prop: sparse edge embedding}.
Note that by \eqref{eq: cF' size},
\begin{align}\label{eq:m}
	m\leq 2\Delta\gamma^{-1}n.
\end{align}
For each $i'\in [m]$, let $F_{i'}\in \cF$ be such that $p_{i'}\in F_{i'}$. 
Then $|\{i'\in [m]:F_{i'}=F\}|=d_{F}(r^1_F)+ d_{F}(r^2_F)$ for each $F\in \cF$.
Let $H^{\rm index}$ be the graph on $[m]$ such that $i'j' \in H^{\rm index}$ if $F_{i'}=F_{j'}$. 
Thus $\Delta(H^{\rm index})\leq 2\Delta$. 
For each $i'\in [m]$, let 
$$q_{i'}:= \{r^{1}_{F_{i'}}, r^{2}_{F_{i'}}\}\cap N_{F_{i'}}(p_{i'}) \enspace \text{and} \enspace u_{i'}:=\phi'(q_{i'}).$$ 
For each $v\in R$, we obtain $|\{i': v=u_{i'}\}|\leq \Delta |\{F\in \cF : v\in \phi'(\{r^1_F,r^2_F\}) \}|\leq \Delta \epsilon^{-2}$ by \ref{item:T4}.
Next, for each $i'\in[m]$, we define sets of vertices $W_{i'}'$ and $W_{i'}''$ which we want to exclude as possible images of $p_{i'}$.
Let ${\I}(i'), i(i')$ be the numbers defined by $\tilde{F_{i'}}\in \tilde{\cF}_{{\I}(i'),i(i')}$. 
We define
\begin{align*}
	W'_{i'}&:=N_{H_1}(u_{i'})\cap A \cap U_{S_{{\I}(i')}},\\
 \widehat{W}'_{i'}&:=N_{H_1}(u_{i'})\cap A \cap U_{[D]\setminus S_{{\I}(i')}}, \text{ and}\\
	W''_{i'}&:=\{ w \in \widehat{W}'_{i'}: d_{H_1,V(C_{{\I}(i'),i(i')})\sm A}(w)< \delta_1^2 n\}. 
\end{align*}
Note $u_{i'}\in R \subseteq V\setminus A$ by \ref{item:T4}. 
Thus \ref{item:U2} and \ref{item:G3} imply that
\begin{align}\label{eq:1}
	|\widehat{W}'_{i'}| \geq (\delta_1 -4\epsilon)|A| \cdot 3/D - \epsilon n \geq \delta_1^{2}n.
\end{align}
Theorem~\ref{thm: almost quasirandom} and \ref{item:G3} imply that
$H_1[A,V\setminus A]$ is $(\epsilon^{1/7},\delta_1)$-super-regular. 
Observe that $|V(C_{{\I(i')},i(i')})\setminus A|
\geq (1-7/(2D))\Gamma \cdot (1-2\epsilon) \frac{n}{\Gamma} -\gamma n
\geq n/3$ by \ref{item:R'1} and \ref{item:G1}. 
Together with Proposition~\ref{prop: reg smaller} this implies that the graph $H_1[\widehat{W}'_{i'},V(C_{I(i'),i(i')})\sm A]$ is $(\epsilon^{1/8},\delta_1)$-regular.
Thus $|W''_{i'}|\leq 2\epsilon^{1/8} n$, by Proposition~\ref{prop: reg right deg}, and therefore
\begin{align*}
d_{H_1,A}(u_{i'}) - |W'_{i'}\cup W''_{i'}| 
\geq |\widehat{W}'_{i'}|- 2\epsilon^{1/8}n 
\stackrel{\eqref{eq:m},\eqref{eq:1}}{\geq}  3\Delta\epsilon^{-2} + m/\epsilon^{-1} + \epsilon^{-1}.
\end{align*} 
Thus we can apply Proposition~\ref{prop: sparse edge embedding} with the following graphs and parameters. \newline
  
\noindent
{ 
\begin{tabular}{c|c|c|c|c|c|c|c|c}
object/parameter & $H_1$ & $H^{\rm index}$ & $\Delta\epsilon^{-2}$ & $A$ &   $m$ & $\epsilon^{-1} $  &  $W'_{i'}\cup W''_{i'}$ & $u_{i'}$
\\ \hline
playing the role of & $G$ & $H$ & $\Delta$ & $A$ &  $m$ & $s$ &$W_{i'}$ & $u_{i'}$
\end{tabular}
}\newline \vspace{0.2cm}

\noindent
We obtain $v_1,\dots, v_m$ so that 
\begin{enumerate}[label=($\alpha$\arabic*)]
	\item\label{item:alpha1} $v_{i'}\in (A \cap U_{[D]\setminus S_{{\I}(i')}})\sm W_{i'}''$,
	\item\label{item:alpha2} $u_{1}v_{1},\ldots,u_{m}v_{m}$ are distinct edges in $H_1$, and
	\item\label{item:alpha3} $|\{ i' : v=v_{i'}\}|\leq \epsilon^{-1}$ for all $v\in A$ and $v_{i'}\notin \{u_{j'},v_{j'}\}$ if $F_{i'}=F_{j'}$ for $i' \neq j'$.
\end{enumerate}
Let $E':=\{u_{i'}v_{i'}: i'\in [m]\}$ and $H'_1:=H_1- E'$. 
Thus 
\begin{align}\label{eq: H1-H'1 sparse}
\Delta(H_1-E(H'_1))=\Delta(E')\leq \epsilon^{-1}+ \Delta\epsilon^{-2} \leq 2\Delta\epsilon^{-2}.
\end{align}
Given a vertex $x\in F$,
recall that $N_F[x] = N_{F}(x)\cup \{x\}$. Moreover, we write $N^2_F[x]$ for the set of all vertices having distance at most $2$ from $x$ in $F$.
Now let $\phi'(p_{i'}):= v_{i'}$ for all $i'\in [m]$. 
Note \ref{item:alpha3} implies that $\phi'$ is injective on $N_F[r^1_F]\cup N_F[r^2_F]$ for any $F\in \cF$.
Also the domain of $\phi'$ is $\{ x \in N_F[r^c_F]: c\in [2], F\in \cF\}$ and for any $v\in V$ we obtain
\begin{align}\label{eq: phi'' size}
|\phi'^{-1}(v)|\leq \Delta(E') \stackrel{\eqref{eq: H1-H'1 sparse}}{\leq} 2\Delta\epsilon^{-2}.
\end{align}
For all $I\in [\hat{D}]$ and $i\in [c'_I]$, 
let $$R_{{\I},i}:= \{ \phi'(x) : x \in N_F(r^c_F) \text{ for some }\tilde{F}\in \tilde{\cF}_{{\I},i} \text{ and }c\in [2]\}.$$
Note that for all $v\in R_{I,i}$, 
there exists $i'\in [m]$ such that $v=v_{i'}$, $I=I(i')$, and $i=i(i')$.
Then 
\begin{align}\label{eq:RIi}
	R_{{\I},i}\sub A \cap U_{[D]\sm S_\I}
\end{align}
 by \ref{item:alpha1}.
Thus $R_{{\I},i}\cap V(C_{{\I},i})=\es$ by \ref{item:R'1}. Also for all $v\in R_{I,i}$, we have
\begin{align}\label{eq: A1 precondition}
d_{H'_1, V(C_{{\I},i})\sm A}(v) 
\stackrel{\eqref{eq: H1-H'1 sparse}}{\geq} d_{H_1, V(C_{{\I},i})\sm A}(v) - 2\Delta\epsilon^{-2}
\stackrel{(\alpha1)}{\geq} \delta_1^2 n/2. \end{align}

Recall that $X_F$ was defined together with $\tilde{F}$ in Step~1. For each $F \in \cF$, we let 
$$X'_{F}:= (N_F^2[r^1_F]\setminus N_F[r^1_F]) \cup (N_F^2[r^2_F]\setminus N_F[r^2_F])\cup X_F.$$
Thus $|X_F'|\leq 2\Delta^2$ (recall that $|X_F|\leq 2$).
Now we sequentially pack the forests in $\tilde{\cF}_{{\I},i}$ into $C_{{\I},i}\cup H_1$ according to the lexicographic order\footnote{We use $\leq$ to equip tuples of reals with a lexicographic ordering; that is, $({\I},i)\leq ({\I}',i')$ if ${\I}<{\I}'$ or ${\I}={\I}'$ and $i\leq i'$. 
For convenience, let $({\I},0):= ({\I}-1,c'_{{\I}-1})$ and $({\I},c'_{{\I}}+1):=({\I}+1,1)$.} 
on the set of tuples $\{({\I},i): I\in [\hat{D}], i\in [c_\I'] \} $. 

More precisely, for each tuple $(\I,i)$ in turn we will construct functions $\tau_{\tilde{F}}$ consistent with $\phi'$ which embed each $\tilde{F} \in \tilde{\cF}_{{\I},i}$ into $C_{{\I},i}\cup H_1$
such that the following properties hold:

\begin{enumerate}[label=(P\arabic*)$_{{\I},i}$]
\item\label{item:P1} $\Delta(C_{{\I},i}  - \bigcup_{\tilde{F}\in \tilde{\cF}_{{\I},i}} E(\tau_{\tilde{F}}(\tilde{F})))
\leq 20\delta_2^6 n/\Gamma$,
\item\label{item:P2} $\tau_{\tilde{F}}(\tilde{F} - (N_{F}[r^1_F]\cup N_{F}[r^2_F])) \subseteq C_{{\I},i}$ 
and thus $V(\tau_{\tilde{F}}(\tilde{F}-(N_{F}[r^1_F]\cup N_{F}[r^2_F])))\sub U_{S_\I}\sm V_0$
for every $\tilde{F} \in \tilde{\cF}_{{\I},i}$, 
\item\label{item:P3} for any vertex $v\in V$, we have 
$$\left|\left\{\tilde{F}  \in \bigcup_{({\I}',i')\leq ({\I},i)} \tilde{\cF} _{{\I}',i'}: v\in \tau_{\tilde{F} }(X'_{F})\right\}\right| 
\leq \frac{2n}{\Gamma},$$
\item\label{item:P4} $\tau_{\tilde{F}}(\tilde{F}-\{r_F^1,r_F^2\})\sub C_{{\I},i}\cup H_{{\I},i}'$,
where $E(\cF,{\I},i):= \bigcup_{({\I}',i')< ({\I},i)} \bigcup_{\tilde{F}\in \tilde{\cF}_{{\I}',i'}} E(\tau_{\tilde{F}}(\tilde{F}))$
and $H'_{{\I},i}:=  H'_1 - E(\cF,{\I},i)$.
Moreover, the graphs $\tau_{\tilde{F}}(\tilde{F})$ are pairwise edge-disjoint for different $\tilde{F}\in \tilde{\cF}_{\I,i}$.
\end{enumerate}
Assume that for some $(I,i)$ we have already constructed $\tau_{\tilde{F}}$ consistent with $\phi'$ for all $\tilde{F} \in \tilde{\cF}_{{\I}',i'}$ with $({\I}',i')< ({\I},i)$ such that (P1)$_{{\I}',i'}$--(P4)$_{{\I}',i'}$ hold.
Let $U'\sub V(C_{\I,i})$ be the set of vertices which lie in $\tau_{\tilde{F}}(X'_F)$ for at least $n/\Gamma$ forests $\tilde{F}\in \tilde{\cF}_{{\I}',i'}$ with $({\I}',i')< ({\I},i)$. 
Since $|X'_{F}|\leq 2\Delta^2$ for any $F\in \cF$, \eqref{eq: cF' size} implies that 
\begin{align}\label{eq: U' size}
|U'| \leq \frac{2 \Delta^2\gamma^{-1}n }{n/\Gamma} \leq \Gamma^{2}.
\end{align}
Let $\ell$ be the length of $C_{{\I},i}$. 
Hence, by \ref{item:R'1}, $\ell$ is odd and $\ell\geq (1- 7/(2D))\Gamma$. 
In particular, together with \ref{item:R'1} and \eqref{eq:forestssizes}, 
this implies $(1-\gamma^2)|C_{{\I},i}| \geq |\tilde{F}|$ for any $\tilde{F}\in \tilde{\cF}_{{\I},i}$. 

Let $\hat{F}:= \tilde{F}- \{r^1_{F}, r^2_{F}\}$ and let $\hat{\cF}_{{\I},i}:=\{\hat{F}: \tilde{F}\in \tilde{\cF}_{{\I},i}\}$.
We wish to apply Lemma~\ref{lem: blow up advanced} with the following graphs and parameters. \newline

\noindent
{
\begin{tabular}{c|c|c|c|c|c|c|c|c}
object/parameter & $C_{{\I},i}$& $\hat{\cF}_{{\I},i}$ & $H'_{{\I},i}[V(C_{I,i})\sm A, R_{I,i}]$ &  $\ell$ & $\delta_2^2/2$ &$\delta_2^4$ & $\delta_1^2/3$& $\gamma^2$  
\\ \hline
playing the role of & $G$ &$\cF$ &$G'$ &$\ell$ & $\alpha$ & $d$ & $d_0$ & $\eta$ 
\\ \hline \hline
object/parameter  & $\phi'$ & $R_{{\I},i}$ & $N_{F}(r^1_F)\cup N_{F}(r^2_F)$ & $U'$ & $X'_F$ & $\Delta$& $3\epsilon$ 
\\ \hline
playing the role of  & $\phi'$& $R$& $r(F)$ & $U'$ & $X'_F$ & $\Delta$ & $\epsilon$
\end{tabular}
}\newline \vspace{0.2cm}

\noindent
So we need to verify conditions \ref{item:A1}--\ref{item:A7} of Lemma~\ref{lem: blow up advanced}.
Condition \ref{item:A1} holds by \ref{item:R'1}.
For any $v\in R_{\I,i}$, by (P2)$_{\I',i'}$ and (P3)$_{\I',i'}$ for $(\I',i')<(\I,i)$,\COMMENT{Actually, subtracting $2n/\Gamma$ below is not necessary since 2. neighbours of roots $r_F^1,r_F^2$ will always be mapped to vertices in $V\sm A$
(since $R_{\I,i}\sub A$ and $H_1$ is bipartite) and so no vertex in $R_{\I,i}$ can ever be an image of such a second neighbour of a root.}
we obtain
$$d_{H'_{{\I},i}, V(C_{{\I},i})\sm A}(v) 
\geq d_{H'_1,V(C_{{\I},i})\sm A}(v) - \Delta |\phi'^{-1}(v)| - \frac{2n}{\Gamma}
\stackrel{\eqref{eq: phi'' size},\eqref{eq: A1 precondition}}{\geq} \frac{\delta_1^2 n}{3}
\geq \frac{\delta_1^2|C_{\I,i}|}{3},$$
which verifies \ref{item:A2}.
Condition \ref{item:A3} follows from \ref{item:T2}, \eqref{eq:forestssizes} 
as well as the fact that each $\tilde{F}$ has at most two components and $\hat{F}= \tilde{F} - \{r^1_{F}, r^2_{F}\}$, 
\ref{item:A4} follows from \eqref{eq: cFij cCij size}, 
\ref{item:A5} is trivial from the definition of $\hat{F}$ and $r^c_F$, and 
\ref{item:A6} follows from \eqref{eq: U' size} and the definition of $X'_F$.
Also \eqref{eq: phi'' size} implies that $|\phi'^{-1}(v)|\leq 2\Delta\epsilon^{-2} \leq \ell^{1/2}$ 
for each $v\in R_{\I,i}$,
and as remarked before \eqref{eq: phi'' size}, $\phi'$ is injective on 
$N_F(r^1_F)\cup N_F(r^2_F)$ for any $F\in \cF$, which verifies \ref{item:A7}. 
 
Thus Lemma~\ref{lem: blow up advanced} gives a function $\tau_{\hat{\cF}_{{\I},i}}$ which is consistent with $\phi'$ and which packs $\hat{\cF}_{{\I},i}$ into  $C_{{\I},i}\cup H'_{{\I},i}[V(C_{{\I},i})\sm A,R_{\I,i}]$ so that
\begin{enumerate}[label=(A$'$\arabic*)]
\item\label{item:B'1} $\Delta( C_{{\I},i}-E(\tau_{\hat{\cF}_{\I,i}}(\hat{\cF}_{\I,i})) ) \leq 5\delta_2^6 n/\ell$,
\item\label{item:B'2} $\tau_{\hat{\cF}_{\I,i}}(V(\hat{F})\setminus(N_{F}(r^1_F) \cup N_{F}(r^2_F) )) \cap R_{I,i} =\emptyset$ for every $\hat{F}\in \hat{\cF}_{\I,i}$, and
\item\label{item:B'3} $\tau_{\hat{\cF}_{\I,i}}(X'_F) \cap U' =\emptyset$ for every $\hat{F}\in \hat{\cF}_{\I,i}$.
\end{enumerate}
Define 
$\tau_{\hat{F}}:=\tau_{\hat{\cF}_{\I,i}}|_{\hat{F}}$
and  $\tau_{\tilde{F}}:=\tau_{\hat{F}}\cup \phi'|_{\{r^1_F,r^2_F\}}$.

Next we verify \ref{item:P1}--\ref{item:P4}.
Properties \ref{item:B'1} and \ref{item:B'2} directly imply \ref{item:P1} and \ref{item:P2}.
For any vertex $v\in U'$ and $\tilde{F}\in \tilde{\cF}_{\I,i}$, \ref{item:B'3} implies that $v\notin \tau_{\tilde{F}}(X'_{F})$.
If $v\in V\sm V(C_{\I,i})$,
then by \ref{item:P2} we also have that $v\notin\tau_{\tilde{F}}(X_F')$.
Thus for each $v\in U'\cup (V\sm V(C_{\I,i}))$, 
by (P3)$_{I',i'}$ with $(I',i')<(I,i)$, we have 
$$\left|\left\{\tilde{F}  \in \bigcup_{({\I}',i')\leq ({\I},i)} \tilde{\cF} _{{\I}',i'}: v\in \tau_{\tilde{F} }(X'_{F})\right\}\right| 
= \left|\left\{\tilde{F}  \in \bigcup_{({\I}',i') <({\I},i)} \tilde{\cF} _{{\I}',i'}: v\in \tau_{\tilde{F} }(X'_{F})\right\}\right|
\leq \frac{2n}{\Gamma}.$$
On the other hand, for each vertex $v \in V(C_{\I,i})\sm U'$, by definition of $U'$ we have
$$\left|\left\{\tilde{F}  \in \bigcup_{({\I}',i')\leq ({\I},i)} \tilde{\cF} _{{\I}',i'}: v\in \tau_{\tilde{F} }(X'_{F})\right\}\right| 
\leq  \frac{n}{\Gamma} + |\tilde{\cF}_{{\I},i}| 
\stackrel{\eqref{eq: F'ij size}}{\leq}\frac{2n}{\Gamma}.$$
So \ref{item:P3} holds.
Observe that \ref{item:P4} holds by construction (here we also use \ref{item:alpha2}, \ref{item:alpha3} and the fact that
$\tau_{\hat{\cF}_{{\I},i}}$ is consistent with $\phi'$).

Let $\tilde{\tau}:=\bigcup_{\tilde{F}\in \tilde{\cF}}\tau_{\tilde{F}}$.
Note the properties \ref{item:P2} and \ref{item:P4} imply that $\tilde{\tau}$ packs $\tilde{\cF}$ into $H_1 \cup \bigcup\cC$ and is consistent with $\phi'$.
Let us define the leftover graphs.
\begin{align}\label{eq: def G1 H1}
G^1:= G - E(\tilde{\tau}(\tilde{\cF})) \enspace \text{and} \enspace H^1_1:= H_1- E(\tilde{\tau}(\tilde{\cF})).
\end{align}
This with \ref{item:P1} and \ref{item:R'3} imply that for every $v\in V\sm (R\cup V_0)$,
we have
\begin{align}\label{eq: G^1 deg}
d_{G^1}(v)  
\leq \frac{20\delta_2^6 n}{\Gamma} \cdot |\cC|+ 4\delta_2^2 n  
\stackrel{(\ref{eq:sizecC})}{\leq} 50 \delta_2^2 n.
\end{align}\COMMENT{
Note that $d_{G^1}(v) \leq \sum_{(\I,i)} \Delta(C_{(\I,i)}- \bigcup_{\tilde{F}\in \tilde{\cF}_{(\I,i)}}E(\tau_{\tilde{F}}(\tilde{F}))) + d_{G}(v)-\sum_{C\in \cC}d_{C}(v) \leq\sum_{(\I,i)} \frac{20\delta_2^6 n}{\Gamma} + 4\delta_2^2 n $ }
Hence $\tilde{\tau}$ forms a (very efficient) approximate cover of the edges of $G$.
Note \ref{item:P2} implies that for any vertex $v\in V$ and $F\in \cF$, 
we can only have $d_{\tilde{\tau}(\tilde{F})\cap H_1}(v) >0$ 
if $v\in \tilde{\tau}(N^2_F[r_F^c])$ (for some $c\in[2]$) which in turn only holds
if $v \in \phi'( N_{F}[r^c_F] )$ or $v\in \tilde{\tau}(X'_F)$.
By \eqref{eq: phi'' size} the former holds for at most $2\Delta\epsilon^{-2}$ forests $F\in \cF$ and by \ref{item:P3} the latter holds for at most $2n/\Gamma$ forests $F\in \cF$. 
Since $\Delta(F)\leq \Delta$, it follows that
\begin{align}\label{eq: E0 deg}
 \Delta(H_1-E(H_1^1))
 =\Delta( E(\tilde{\tau}(\tilde{\cF}))\cap H_1) 
 \leq \Delta( 2\Delta\epsilon^{-2} + 2n/\Gamma) 
 \leq \epsilon n.
\end{align}
Thus $H^1_1$ inherits all the relevant properties of $H_1$.
We proceed with a few observations.
Recall that for all $F\in\cF^0$, 
we have $\tilde{F}=F$.
Thus $\tilde{\tau}$ actually packs $F$ into $H_1 \cup \bigcup\cC$.
On the other hand,
for all forests $F\in \cF^{\gamma_*} \cup\cF^\eta $,  
the function $\tilde{\tau}$ only packs a proper subgraph of $F$ into $H_1 \cup \bigcup\cC$.
In the next few steps we will extend this by packing the forests in 
$\sT^{\gamma_*}\cup \sF^{\eta}$.
As a preliminary step towards this, we embed the roots of these forests via a function $\tau'$ in Step~\ref{step3}.

\step{Embedding the roots of the forests in $\sT^{\gamma_*}\cup \sF^\eta$}\label{step3}
In this step, 
we define the function $\tau'$ which,
for every ${\I}\in [\hat{D}]$, 
embeds the roots $\{y^{c}_F : F\in \sF^{\eta}_{\I}, c\in [2]\} \cup \{y_F: F\in \sT^{\gamma_*}_{\I}\}$ into $A\cap U_i$ for some $i\notin S_\I$ 
(recall that their neighbours $x_F^c$ and $x_F$ have already been embedded in the previous step as they belong to $\tilde{F}$). 
We stress here that one key point is the following.
Given ${\I} \in [\hat{D}]$ and $F\in \cF_{\I}^{\gamma_*}$,
in Step~\ref{step2} the subforest $\tilde{F}$ of $F$ was embedded into $U_{S_{\I}}$ 
(apart from its roots $r_F^1,r_F^2$ which are mapped to $R$ and the neighbours of those roots which are embedded into $A\cap U_{[D]\setminus S_I}$ by \eqref{eq:RIi}),
while the forest $F-V(\tilde{F})\in \sT^{\gamma_*}_\I$ will be embedded into $U_{[D]\sm S_{\I}}$.
So after we have chosen a suitable image $\tau'(y_F)$ of the root $y_F$ of $F-V(\tilde{F})$ inside $U_{[D]\sm S_{\I}}$,
any embedding of $F-V(\tilde{F})$ into $U_{[D]\sm S_{\I}}$ which is consistent with $\tau'(y_F)$ and which avoids $\tilde{\tau}(N(r_F^1)\cup N(r_F^2))$
yields an embedding of $F$.
For $F\in \cF_{\I}^{\eta}$ the strategy is similar, but slightly more complicated.

Consider $F\in \cF$.
For any $F'\sub F$, we define
\begin{align}\label{eq:WF}
	W_{F'}:= \tilde{\tau}(N_{F}[r^1_F]\cup N_{F}[r^2_F])=\phi'(N_{F}[r^1_F]\cup N_{F}[r^2_F]).
\end{align}
\COMMENT{intentionally no primes on the right hand side}
(So for example, if $F\in \cF^{\gamma_*}$ and $T$ is the subtree of $F$ belonging to $\sT^{\gamma_*}$, 
then this allows us to refer directly to $W_F$ as $W_T$.)
Note that by \ref{item:T4}, \ref{item:P2} and \eqref{eq:RIi}, for all $F\in \cF_\I$, 
we have 
\begin{align}\label{eq:WF1}
	\tilde{\tau}(\tilde{F})\cap (U_{[D]\sm S_{\I}}\cap V_0)= W_F\sm \tilde{\tau}(\{r^1_F,r^2_F\})
	\enspace\text{ and }\enspace
	\tilde{\tau}(\tilde{F})\cap R = \tilde{\tau}(\{r_F^1,r_F^2\})\sub W_F.
\end{align}
As discussed above, we need to avoid $W_F$ when defining $\tau'$ in the current step.
Note 
\begin{align}\label{eq: WF size}
|W_F|\leq 2\Delta+2 \leq \Delta^2.
\end{align}
Furthermore, \eqref{eq: phi'' size} states that for all $v\in V$, 
\begin{align}\label{eq: not too many W_F}
|\{F\in \cF : v\in W_F )\}| \leq 2\Delta\epsilon^{-2} \leq \epsilon^{-3}.
\end{align}

Consider any ${\I}\in [\hat{D}]$.
Recall from \eqref{eq: size of sF eta} that $|\sF^{\eta}_{\I}|=|\cF^{\eta}_{\I}|=3\lceil n/(6\hat{D}) \rceil$.
Let $\{j,j',j''\}$ be such that $\{j,j',j''\}=[D]\sm S_{\I}$. 
Split the forests in $\sF^{\eta}_{\I}$ into three sets $\sF^{\eta}_{\I}(1),\sF^{\eta}_{\I}(2),\sF^{\eta}_{\I}(3)$ of size $\lceil n/(6\hat{D}) \rceil$
and define 
\begin{align*}
	j(F)=
	\left\{ \begin{array}{ll}
j &\text{ if } F\in\sF^{\eta}_{\I}(1),\\
j' &\text{ if } F\in\sF^{\eta}_{\I}(2),\\
j'' &\text{ if } F\in\sF^{\eta}_{\I}(3).\\
\end{array}\right.
\end{align*}
In a similar way we define $j(T)$ for the trees $T$ in the families $\sT^{\delta_1}_\I,\sT^{\delta_2}_\I,\sT^{\p}_\I$
(again, assigning each of the three possible values in $[D]\sm S_\I$ to a third of the trees in each of these families).
For each $i\in [D]$, let
\begin{itemize}
	\item $\sF^{i,\eta}:= \{ F \in \sF^{\eta}: j(F)=i\}$ and
	\item $\sT^{i,\xi}:= \{ T \in \sT^{\xi}: j(T)=i\}$ for all $\xi\in\{\gamma_*,\delta_1,\delta_2,\p\}$.
\end{itemize}
Thus \eqref{eq:WF1} implies the following:
 \begin{equation}\label{eq:tauWF}
 \begin{minipage}[c]{0.8\textwidth}\em
 For all $i\in[D]$ and $F\in \cF^{\eta}\cup \cF^{\gamma_*}$,
	let $F'$ be the unique subforest of $F$ belonging to $\sF^{\eta}\cup \sT^{\gamma_*}$.
	If $F'\in \sF^{i,\eta}\cup \sT^{i,\gamma_*}$, then
	$\tilde{\tau}(\tilde{F})\cap (U_i \cup V_0\cup R)
	= W_F$.
 \end{minipage}\ignorespacesafterend 
\end{equation}
Moreover, since $|R|,|V_0|\leq 2\epsilon n$,
\begin{align}\label{eq: cFi eta size}
|\sF^{i,\eta}| = \binom{D-1}{D-3} \left\lceil \frac{n}{6\hat{D}}\right\rceil = (1\pm 5D\epsilon) \frac{|U_{i}\cup V_0\cup R|}{2}.
\end{align}
For each $i\in [D]$ in turn, 
we will now use Proposition~\ref{prop: sparse edge embedding} to define a function $\tau'$ 
packing $\{x_Fy_F: F\in \sT^{i,\gamma_*}\}\cup \{x_F^1y_F^1,x_F^2y_F^2:F\in \sF^{i,\eta}\}$ into $G'\cup H_1^1$
such that $y_F,y_F^1,y_F^2\in A \cap U_{i}$ and $\tau'$ is consistent with $\tilde{\tau}$.
Hence assume that for some $i\in [D]$ we have already determined $\tau'$
for all $F\in \bigcup_{j=1}^{i-1}(\sF^{j,\eta}\cup \sT^{j,\gamma_*})$
such that
\begin{align}\label{eq:sizeEj}
	\Delta(E_1^{j})\leq \epsilon^2 n
\end{align}
for all $j<i$,
where
$$E_1^j:= 
\bigcup_{F\in \sF^{j,\eta}} \{ \tau'(x^c_F)\tau'(y^c_F): c\in [2]\}
\cup \bigcup_{F\in \sT^{j,\gamma_*}} \{ \tau'(x_F)\tau'(y_F)\}.
$$
Let $G^1_{i}:= (G' \cup H^1_1) - \bigcup_{j=1}^{i-1} E_1^j$. 
By \ref{item:G2}, \ref{item:G3}, \ref{item:U2}, \eqref{eq: E0 deg} and \eqref{eq:sizeEj}, for every $v\in V$,%
\COMMENT{Here, either $v\in A$ or $v\in V\setminus A$. 
In the former case, $d_{G^1_{i},U_{i}\cap A}(v) \geq d_{G',U_i\cap A}(v) - D\epsilon^2 n \geq p\gamma n/(2D)$. 
In the latter case, $d_{G^1_i,U_{i}\cap A}(v) \geq d_{H^1_1,U_i\cap A}(v) - D\epsilon^2 n \geq \delta_1\gamma n/(2D) - \epsilon n - D\epsilon^2 n \geq \delta_1^2 n$.} 
we have 
\begin{align}\label{eq:dG1U}
	d_{G^1_{i},U_{i}\cap A}(v)  \geq \delta_1\gamma n/(2D) - \epsilon n - D\epsilon^2 n \geq \delta_1^2 n.
\end{align}
Write $\{{F_1^{\eta}},\dots, F_{m_1}^{\eta}\}:=\sF^{i,\eta}$ and $\{T_{1}^{\gamma_*}, \dots, T^{\gamma_*}_{m_2}\}:=\sT^{i,\gamma_*}$. 
Let $u_j:= \tilde{\tau}(x^1_{F_j^\eta})$, $u_{m_1+j}=\tilde{\tau}(x^2_{F_j^\eta})$ for each $j\in [m_1]$ and 
$u_{2m_1+j}= \tilde{\tau}(x_{T_{j}^{\gamma_*}})$ for each $j\in [m_2]$. 
We define $W_j:=W_{m_1+j}:=W_{F_j^{\eta}}$ if $j\in [m_1]$ and 
$W_{2m_1+j}:=W_{T_{j}^{\gamma_*}}$ if $j\in [m_2]$.
By  \ref{item:P3} and the definition of $X_F$, we conclude that for any $v\in V$, we have 
\begin{align}\label{eq:settingedges}
	|\{j \in [2m_1+m_2]: v= u_j\}| =
|\{F\in \cF^{\gamma_*} : v=\tilde{\tau}(x_F)\}|+|\{F\in \cF^{\eta}: v\in \tilde{\tau}(\{x^1_F,x^2_F\})| 
 \leq 2n/\Gamma.
\end{align}
Let $H^{\rm index}_{i}$ be the graph on $[2m_1+m_2]$ such that $jj'\in H^{\rm index}$ if $j'= j+m_1$ and $j\leq m_1$. 
Note that for all $j\in [2m_1+m_2]$, we have
\begin{align*}
	d_{G^1_{i}, U_{i}\cap A}(u_j) - |W_j|-|V_0|
	\stackrel{\eqref{eq: V0 size},\eqref{eq: WF size},\eqref{eq:dG1U}}{\geq}
	 \delta_1^2 n -\Delta^2 - 2\epsilon n
	 \stackrel{\eqref{eq: size of sF eta}}{\geq}  6\Gamma^{-1} n + (2m_1+m_2)/\epsilon^{-1} + \epsilon^{-1}.
\end{align*}
Thus we can apply Proposition~\ref{prop: sparse edge embedding} with the following parameters and graphs.\newline
  
\noindent
{
\begin{tabular}{c|c|c|c|c|c|c|c|c}
object/parameter & $G^1_{i}$ & $H^{\rm index}_{i}$ & $2n/\Gamma$ & $A\cap U_{i}$ &   $2m_1+m_2$ & $\epsilon^{-1} $  &  $W_{j}\cup V_0$ & $u_j$
\\ \hline
playing the role of & $G$ & $H$ & $\Delta$ & $A$ &  $m$ & $s$ &$W_{j}$ & $u_j$
\end{tabular}
}\newline \vspace{0.2cm}

\noindent
We obtain a sequence of vertices $v_1,\dots, v_{2m_1+m_2}$ 
such that the following hold:
\begin{enumerate}[label=($\beta$\arabic*)]
	\item\label{item:beta1} $u_1v_1,\ldots,u_{2m_1+m_2}v_{2m_1+m_2}$ are distinct edges in $G^1_i$,
	\item\label{item:beta2} $v_j\notin \{u_{j'},v_{j'}\}$ whenever $|j-j'|=m_1$ and $j,j'\in [2m_1]$,
	\item\label{item:beta3} $v_j,v_{m_1+j}\in (U_i \cap A)\sm (W_{F_j^\eta}\cup V_0)$ for all $j\in [m_1]$ and 
	$v_{2m_1+j}\in (U_i \cap A)\sm (W_{T_j^{\gamma_*}}\cup V_0)$ for all $j\in [m_2]$, and
	\item\label{item:beta4} every vertex $v \in V$ satisfies $|\{j\in [2m_1+m_2]: v=v_j\}|\leq \epsilon^{-1}$. 
\end{enumerate}
We define $\tau'(y^1_{F_j^\eta}):=v_j, \tau'(y^2_{F_j^\eta}):= v_{m_1+j}$ for all $j\in [m_1]$ and 
$\tau'(y_{T_j^{\gamma_*}}):= v_{2m_1+j}$ for $j\in[m_2]$. 
Let 
$$E^{i}_1:= \{u_jv_j: j\in [2m_1+m_2]\}.$$
Then \eqref{eq:settingedges} and \ref{item:beta4} imply that $\Delta(E^{i}_1) \leq 2n/\Gamma + \epsilon^{-1} \leq \epsilon^2 n$ as required in \eqref{eq:sizeEj}. 
By repeating this procedure for every $i\in [D]$, we define $\tau'$ as desired.
We claim that $\tau'$ satisfies the following properties:
\begin{enumerate}[label=($\gamma$\arabic*)]
	\item \label{item:gamma0} for every $F\in \cF^\eta\cup \cF^{\gamma_*}$ and every $y\in \{y_F, y^1_F, y^2_F\}$,
	we have
	$\tau'(y)\notin \tilde{\tau}(\tilde{F})$;
	moreover,
	$\tau'(y_F^1)\neq \tau'(y_F^2)$ for each $F\in \cF^\eta$,
	\item \label{item:gamma1} for any $v\in V$,
	there are  at most $\epsilon^{-1}$ forests $F\in \sF^{\eta}\cup \sT^{\gamma_*}$ 
	such that $v=\tau'(y)$ for some $y \in \{y_F, y^1_F, y^2_F\}$,
	\item \label{item:gamma2} for all $j\in [D]$ and $F\in \sF^{j,\eta}\cup \sT^{j,\gamma_*}$, 
	every root $y$ of a component of $F$ satisfies $\tau'(y) \in (U_{j}\cap A)\sm (W_F\cup V_0)$, and
	\item \label{item:gamma3} if $\I\in [\hat{D}]$ and $F \in \sF_\I^\eta\cup \sT^{\gamma_*}_\I$, 
	then every root $y$ of a component of $F$ satisfies $\tau'(y)\in A_\I\sm W_F=(A\cap U_{[D]\sm S_{\I}})\sm W_F$.
	\item\label{item:gamma4} for all $T\in \sT^{\gamma_*}$, 
	we have that $\tilde{\tau}(x_{T})\tau'(y_{T})$ is an edge of $G'\cup H_1^1$ and 
	for all $F\in \sF^{\eta}$ and $c\in [2]$, 
	we have that $\tilde{\tau}(x^c_F)\tau'(y^c_F)$ is an edge of $G'\cup H_1^1$.
\end{enumerate}
Indeed,
\ref{item:gamma0} follows from \ref{item:beta2}, \ref{item:beta3} and \eqref{eq:tauWF},
\ref{item:gamma1} follows from \ref{item:beta4},
while \ref{item:gamma2} and \ref{item:gamma3} both follow from \ref{item:beta3} and \ref{item:gamma4} follows from \ref{item:beta1}.

Let $E_1:=\bigcup_{i=1}^{D} E_1^{i}$.
We conclude from \eqref{eq:sizeEj} that
\begin{align}\label{eq: Delta E1}
\Delta(E_1)\leq D\epsilon^2 n \leq \epsilon n.
\end{align}
Let 
\begin{align}\label{eq:G2}
	G'^2:= G'-E_1, \quad G^2:= G^1 \quad \text{and} \quad H^2_1:= H^1_1 -E_1.
\end{align}
Following on from \eqref{eq: def G1 H1}, these updates track the edges of $G',G$ and $H_1$ which are still available after this step.
We use \eqref{eq: H1 quasi-random}, \eqref{eq: E0 deg}, \eqref{eq: Delta E1} and Proposition~\ref{prop: quasi-random subgraph} to conclude the following.
\begin{equation}\label{eq: H12 quasi-random}
\begin{minipage}[c]{0.9\textwidth}\em
For all $j\in [D]$ and $I\in [\hat{D}]$,
the bipartite graphs $H_1^2[U_j]=H_1^2[U_j\cap A, U_j \sm A]$, $H_1^2[U_j\cup R]=H_1^2[U_j\cap A, (U_j \cup R) \sm A]$
and $H_1^2[A_\I\cup B_\I]$ are $( \epsilon^{1/3},\delta_1)$-quasi-random.
\end{minipage}
\end{equation}

\step{Covering almost all the edges incident to the exceptional set}\label{step4}
Property \eqref{eq: G^1 deg} shows that in Step~\ref{step2}, 
we have covered almost all edges of $G$ incident to vertices in $V\sm (V_0 \cup R)$.
In the following step,
we use the forests in $\sF^\eta$ to cover almost all edges incident to $V_0 \cup R$.
We will achieve this through several applications of Lemma~\ref{lem: clear rl waste}.

For each $i\in [D]$, we consider the graph  
$$G(i):= (G^2\cup H_2) [U_i\cup R\cup V_0] - E((G^2\cup H_2)[R\cup V_0]).$$ 
Observe that $G(1),\dots, G(D)$ are pairwise edge-disjoint. 
Note that by \ref{item:U2} and \ref{item:G4}, 
any two vertices $u,v\in U_i$ satisfy
\begin{align}\label{eq: G(i) codegree}
d_{G(i)}(u,v) \geq 3\delta_2^2 n/(5D).
\end{align}
Recall that in Step~\ref{step3} we partitioned $\sF^\eta$ into $\sF^{1,\eta},\ldots,\sF^{D,\eta}$.
For each $F\in \sF^{i,\eta}$, 
we choose two distinct vertices $q^1_F, q^2_F$ in $(U_i\cap A)\setminus (W_F \cup V_0 \cup \tau'(\{y^1_F, y^2_F\}))$ uniformly at random. 
Note that by \ref{item:U1} and \eqref{eq: WF size}, 
$$|(U_i\cap A)\setminus (W_F \cup V_0 \cup \tau'(\{y^1_F, y^2_F\}))| \stackrel{\eqref{eq: V0 size}}{\geq} D^{-1}\gamma n - 2\epsilon n - 2\Delta^2 \geq \gamma^2 n$$ 
and recall that $|\sF^{i,\eta}|$ is given by \eqref{eq: cFi eta size}.
Thus by straightforward applications of Lemma~\ref{lem: chernoff}, 
with probability at least $1/2$, for each $v\in U_i\cap A$, we have
\COMMENT{For a vertex $v$, $\mathbb{P}[v\in \{q^1_F,q^2_F\}]\leq  \frac{2}{\gamma^2 n}.$ 
Thus $\mathbb{E}[|\{F: v\in \{q^1_F,q^2_F\}\}|] \leq  \frac{2(1\pm \epsilon^{1/2})D^{-1}n/2}{\gamma^2 n}$. 
Then Chernoff gives us that
$\mathbb{P}[|\{F: v\in \{q^1_F,q^2_F\}\}|\leq  n^{2/3}] \geq 1 - c^{n^{4/3}/n}\geq 1- c^{n^{1/4}}$ for a constant $0<c<1$. 
Thus with probability at least $1- nc^{n^{1/4}}$, this happens for all $v\in A$ and all $i\in [D]$. } 
\begin{align}\label{eq: not too many q}
|\{ F\in \sF^{i,\eta}: v\in \{q^1_F,q^2_F\}\}| 
\leq n^{2/3}. 
\end{align}
Thus there exists a choice such that \eqref{eq: not too many q} holds for all vertices $v\in A$. 
Recall from Step~\ref{step1} that for all $F\in \sF^{i,\eta}$ and $c\in[2]$,
the vertex $z_F^c$ is the unique neighbour of the leaf $\ell^c_F $.
We define $\tau'(z^c_F):= q^c_F$ for all $F\in \sF^{i,\eta}$ and $c\in [2]$. 
For all $i\in [D]$, we now wish to apply Lemma~\ref{lem: clear rl waste} with the following graphs and parameters. \newline

\noindent
{
\begin{tabular}{c|c|c|c|c|c|c}
object/parameter & $G(i)$& $ \sF^{i,\eta} $& $|U_i\cup R\cup V_0|$ & $V_0\cup R$ & $5D \epsilon$ &$\eta$ 
\\ \hline
playing the role of & $G$ &$\cF$ & $n$ & $V_0$ & $\epsilon$ & $\eta$ 
\\ \hline 
\hline
object/parameter &  $\delta_2$& $3\Delta$ & $W_{F}$ & $\tau'|_{\{y^1_F, y^2_F, z^1_F, z^2_F\}}$ & $\{y^1_F,y^2_F\}$ & $\{z^1_F, z^2_F\}$  
\\ \hline
playing the role of & $\delta_2$ & $\Delta$ & $W_F$& $\tau'_F$& $\{y^1_F, y^2_F\}$& $\{z^1_F,z^2_F\}$
\end{tabular}
}\newline \vspace{0.2cm}

\noindent
Assumptions \ref{item:L11}--\ref{item:L15}  of Lemma~\ref{lem: clear rl waste} hold in the above set-up, 
as \eqref{eq: G(i) codegree} implies \ref{item:L11}, 
\eqref{eq:forestssizes} implies \ref{item:L12}, 
\eqref{eq:5indset}, \ref{item:gamma1}, \ref{item:gamma2}, and \eqref{eq: not too many q} imply \ref{item:L13}, 
\eqref{eq: WF size} implies \ref{item:L14}, and \eqref{eq: not too many W_F} implies \ref{item:L15}. 
Also \eqref{eq: cFi eta size} ensures that $\sF^{i,\eta}$ contains the appropriate number of forests. 

From Lemma~\ref{lem: clear rl waste}
we obtain a function $\tau^{\eta}_i$ packing $\sF^{i,\eta}$ into $G(i)$ which is consistent with $\tau'$ such that the following properties hold:
\begin{enumerate}[label=(Q1.\arabic*)$_i$]
\item\label{item:Q11i} $V(\tau_i^\eta(F))\sub U_i \cup R\cup V_0$ and $\tau^{\eta}_i(F)\cap W_F=\emptyset$ for every $F\in \sF^{i,\eta}$,
\item\label{item:Q12i}  $d_{\tau^{\eta}_i(\sF^{i,\eta})}(v)\leq \eta^{1/3}| U_i \cup R\cup V_0|$ for all $v\in U_i\setminus(V_0\cup R)$, and
\item\label{item:Q13i}  $d_{\tau^{\eta}_i(\sF^{i,\eta})}(v) \geq d_{G(i)}(v) - \epsilon^{1/3} | U_i \cup R\cup V_0|$ for every $v\in V_0\cup R$.
\end{enumerate}
For every $F'\in \sF^{i,\eta}$ and $F\in \cF^\eta$ such that 
$F'$ is the unique subforest of $F$ belonging to $\sF^\eta$,
\eqref{eq:tauWF} and \ref{item:Q11i} imply 
that $\tilde{\tau}(\tilde{F})\cap \tau^\eta_i(F')=\es$. Together with \ref{item:gamma4} this means that $\tilde{\tau}(\tilde{F}) \cup \tau^\eta_i(F')$ yields an embedding of $F - \{\ell^1_F, \ell^2_F\}$.

We apply Lemma~\ref{lem: clear rl waste} for every $i\in [D]$
and let $\tau^{\eta}:= \bigcup_{i=1}^{D} \tau^{\eta}_i$.
Thus $\tau^{\eta}$ packs $\sF^{\eta}$ into $G^2 \cup H_2$.
Moreover, 
the function $\tilde{\tau}\cup \tau^{\eta}$ packs $\{F-\{\ell_F^1,\ell_F^2\}:F\in \cF^{\eta}\}$ into $G\cup G'\cup H_1 \cup H_2$.
Let 
\begin{align}\label{eq:G3}
	G^{3}:= G^{2}-E(\tau^{\eta}(\sF^{\eta})),\enspace G'^{3}:= G'^2,
	\enspace H^3_1:= H^2_1 \enspace\text{and}\enspace H^3_2:= H_2 -E(\tau^{\eta}(\sF^{\eta})).
\end{align}
Following on from \eqref{eq:G2}, these updates track the edges which are still available after this step.
Let 
\begin{align}\label{eq:defG3}
	G^*_3:= G'^3\cup G^3\cup H^3_1\cup H^3_2.
\end{align}
Thus $G^*_3$ is the graph consisting of all the leftover edges.
By \ref{item:Q12i} for each $i\in [D]$, every $v\in V\sm (V_0\cup R)$ satisfies 
\begin{align}\label{eq: E2 deg}
d_{\tau^{\eta}(\sF^{\eta})}(v) \leq \sum_{i=1}^{D} \eta^{1/3} |U_i\cup R\cup V_0| \leq 2\eta^{1/3} n.
\end{align}
By \ref{item:Q13i} for each $i\in [D]$, every $v\in R\cup V_0$ satisfies  
\begin{align}\label{eq: G3 V0 density}
d_{G^{3}\cup H^3_2}(v) \leq  \sum_{i=1}^{D}\epsilon^{1/3} | U_i \cup R\cup V_0| + |R \cup V_0|\stackrel{\eqref{eq: V0 size}}{\leq} 2\epsilon^{1/3} n.
\end{align}
Inequality \eqref{eq: not too many q} implies that for every vertex $v \in V$, we have
\begin{align}\label{eq: not too many z}
 |\{ F \in \sF^\eta: v\in \tau^{\eta}(\{z^{1}_F,z^{2}_F\})\}| \leq n^{2/3}.
\end{align}

\step{Covering the remaining edges in  $V\sm A$}\label{step5}
Recall that $E(G^*_3)$ is precisely the total set of uncovered edges at this point.
Next we pack the trees in $\sT^{\delta_2}$ 
into $G^*_3$ so that they cover all edges in $G_3^*[V\sm A]$
and so that the packing is consistent with $\tau'$;
that is, $\tau'$ prescribes the image of the root $y_T$ of every tree $T\in\sT^{\delta_2}$. 
Lemma~\ref{lem: clear delta2} is the main tool for this step.

More precisely,
we will sequentially construct functions $\tau^{\delta_2}_{1},\ldots,\tau^{\delta_2}_{\hat{D}}$ 
which are consistent with $\tau'$ and such that $\tau^{\delta_2}_{{\I}}$ packs $\sT^{\delta_2}_{\I}$ into $G_3^*$.
For $I\in [\hat{D}]$, let $$G^*_{3,{\I}}:= G^*_{3}[A_{\I}\cup B_{\I}] - \bigcup_{{\I}'=1}^{{\I}-1} E(\tau^{\delta_2}_{{\I}'}(\sT_{{\I}'}^{\delta_2})).$$ 
(Recall from \eqref{eq:defAB} that $A_{\I}= A\cap U_{[D]\setminus S_{\I}}$ and $B_{\I}= (U_{[D]\setminus S_{\I}}\cup  R)\setminus A$.)
We will construct the functions $\tau_\I^{\delta_2}$ such that they satisfy the following properties:
\begin{enumerate}[label=(Q2.\arabic*)$_{\I}$]
\item\label{item:Q21i} $\tau_\I^{\delta_2}(\sT^{\delta_2})\sub G^*_{3,{\I}}$,
\item\label{item:Q22i} $E(G^*_{3,{\I}}[B_{\I}])\subseteq E(\tau^{\delta_2}_{\I}(\sT_{\I}^{\delta_2}))$, and
\item\label{item:Q23i} $W_T \cap \tau^{\delta_2}_{\I}(T)=\emptyset$ for all $T \in \sT_{\I}^{\delta_2}$. 
\end{enumerate}

Assume that for some $\I\in [\hat{D}]$ we have already defined $\tau^{\delta_2}_{1},\ldots,\tau^{\delta_2}_{{\I}-1}$
which are consistent with $\tau'$ and so that $\tau^{\delta_2}_{{\I}'}$ satisfies (Q2.1)$_{{\I}'}$--(Q2.3)$_{{\I}'}$ for each $\I'< \I$. 
Recall from \eqref{eq:G3} that $H_1^3=H_1^2$.
By \eqref{eq: cF sizes} and \eqref{eq: H12 quasi-random}, we obtain
for every vertex $v\in B_{\I}$ that 
\begin{eqnarray}\label{eq: min deg G3}
d_{G^*_{3,\I}[A_{\I},B_{\I}]}(v) 
&\geq& d_{H^3_1[A_{\I},B_{\I}]}(v) - \sum_{{\I}'< {\I}} \Delta|\cF^{\delta_2}_{{\I}'}| 
\geq \delta_1|A_{\I}|/2 - \hat{D} \Delta \delta_2^{1/2} n \geq \delta_1 |A_{\I}|/3.
 \end{eqnarray}
Note that $G'[A_{\I}] - E( G^*_{3,{\I}})$ consists of those edges of $G'[A_\I]$ which have been used by the packing so far.
More precisely, 
\eqref{eq:G2}, \eqref{eq:G3} and the fact that $G_3^*[A_\I]=G'^3[A_\I]$ (see \ref{item:G1}) together imply that 
$E(G'[A_{\I}] - E( G^*_{3,{\I}}))$ is the union of the edges in 
$G^*_3[A_{\I}]- E(G^*_{3,{\I}}[A_{\I}]) \sub \bigcup_{{\I}'=1}^{{\I}-1}\tau^{\delta_2}_{\I'}(\sT^{\delta_2}_{{\I}'})$ and of $E(G'[A_{\I}]-E(G'^3[A_{\I}])) \sub E_1$.
Thus \eqref{eq: cF sizes} and \eqref{eq: Delta E1} imply that 
$$\Delta(G'[A_{\I}] - E(G^*_{3,{\I}}))\leq \Delta \hat{D} \delta_2^{1/2} n +\epsilon n \leq \gamma^{2}|A_\I|.$$ 
In a similar way, 
it follows that for each $j\in [D]$, 
we have $\Delta(G'[A\cap U_j]-E(G^*_3[A\cap U_j]))\leq \epsilon n$ and that $\Delta(G'[A_I] - E(G^*_3[A_I]))\leq \epsilon n.$
Using \ref{item:G2}, \ref{item:U2} and Proposition~\ref{prop: quasi-random subgraph},
we conclude for each $j\in [D]$ that 
\begin{equation}\label{eq: G*3,i quasi-random}
\begin{minipage}[c]{0.8\textwidth}\em
$G^*_3[A_I]$, $G^*_{3,{\I}}[A_{\I}]$ and $G^*_{3}[A\cap U_j]$ are $(\gamma^{1/4},p)$-quasi-random.
\end{minipage}
\end{equation}
Note that for each $v\in B_{\I}\setminus (V_0\cup R)$, we have
\begin{eqnarray*}
d_{G^*_{3,{\I}},B_{\I}}(v) 
&\leq& d_{G^*_{3},B_{\I}}(v)
\stackrel{\text{\ref{item:G2},\ref{item:G3}}}{\leq} d_{G^3\cup H^3_2,B_{\I}}(v) 
\leq d_{G^1\cup H_2,B_{\I}}(v) \\
&\stackrel{\text{\eqref{eq: G^1 deg},\ref{item:G4}},\text{\ref{item:U2}}}{\leq}&
50\delta_2^2 n + \frac{5}{3}\delta_2|U_{[D]\sm S_{\I}}| +|R|\\
&\leq& 2\delta_2 |A_{\I}\cup B_{\I}|.
\end{eqnarray*}
Moreover, for each $v\in B_{\I}\cap(V_0\cup R)$, we have
\begin{eqnarray*}
d_{G^*_{3,{\I}},B_{\I}}(v) 
\leq d_{G^*_{3},B_{\I}}(v)
\leq d_{G^3\cup H^3_2}(v)
\stackrel{ \eqref{eq: G3 V0 density} }{\leq} 2\epsilon^{1/3} n 
\leq 2\delta_2|A_{\I}\cup B_{\I}|.
\end{eqnarray*}
Thus 
\begin{align}\label{eq: G*3i max deg}
\Delta(G^*_{3,{\I}}[B_{\I}])\leq 2\delta_2|A_{\I}\cup B_{\I}|.
\end{align}
We now wish to apply Lemma~\ref{lem: clear delta2} with the following graphs and parameters. \newline

\noindent
{
\begin{tabular}{c|c|c|c|c|c|c|c|c|c|c|c|c|c}
object/parameter & $G^*_{3,{\I}}$ & $\sT^{\delta_2}_{\I}$ & $A_{\I}$ & $|A_{\I}\cup B_{\I}|$ &$D\epsilon$ &$\delta_2$ & $\delta_1$ &  $\gamma$ &  $\Delta$   &  $p$ & $y_{T}$ & $\tau'|_{\{y_T\}}$ 
& $W_{T}$\\ \hline
playing the role of & $G$ & $\cT$ & $A$ & $n$ & $\epsilon$ & $\delta_2$ & $\delta_1$ &  $\gamma$ &  $\Delta$  & $p $ & $y_{T}$ & $\tau'_{T}$ & $ W_T$
\end{tabular}
}\newline \vspace{0.2cm}

\noindent
Observe that condition \ref{item:L21} of Lemma~\ref{lem: clear delta2} follows from \eqref{eq: G*3,i quasi-random},
\ref{item:L22} from \eqref{eq: G*3i max deg},
\ref{item:L23} from \eqref{eq: min deg G3},
\ref{item:L24} from \eqref{eq: cF sizes},
\ref{item:L25} from \eqref{eq:forestssizes}, \eqref{eq: WF size} and \ref{item:gamma3}, and
\ref{item:L26} from \eqref{eq: not too many W_F}. Thus Lemma~\ref{lem: clear delta2} gives us a function $\tau^{\delta_2}_{\I}$ which is consistent with $\tau'$,
packs $\sT_\I^{\delta_2}$ into $G_{3,{\I}}^*$, and satisfies \ref{item:Q21i}--\ref{item:Q23i}.

Assume now that
for all $\I\in [\hat{D}]$
we have defined such a function $\tau^{\delta_2}_{\I}$.
Let $\tau^{\delta_2} := \bigcup_{{\I}=1}^{\hat{D}} \tau^{\delta_2}_{\I}$.
Note that for every edge $uv \in E(G^*_3[V\setminus A])$, 
there exists (a smallest) ${\I}'\in [\hat{D}]$ such that $u,v \in B_{{\I}'}$. 
Thus (Q2.2)$_{{\I}'}$ implies that $uv \in \tau^{\delta_2}(\sT^{\delta_2})$
and so
\begin{align}\label{eq: delta 2 cleared}
E(G^*_3[V\setminus A]) \sub \tau^{\delta_2}(\sT^{\delta_2}).
\end{align}
Property \ref{item:Q23i} for all ${\I}\in [\hat{D}]$,
\eqref{eq:WF1}
 and the fact that $\tau^{\delta_2}(\sT_{\I}^{\delta_2})\subseteq U_{[D]\setminus S_{\I}}\cup R$
together ensure that for each $F\in \cF^{\delta_2}$, we have
$	\tau^{\delta_2}(F-V(\tilde{F}))\cap \tilde{\tau}(\tilde{F}) =\emptyset.$
Together with \ref{item:gamma4} this means that $\tilde{\tau}\cup \tau^{\delta_2}$ packs the forests in $\cF^{\delta_2}$ into $G\cup G'\cup H_1\cup H_2$.
We define
\begin{align}
	G^4&:=G^3-E(\tau^{\delta_2}(\sT^{\delta_2})), \quad \notag
	G'^4:=G'^3-E(\tau^{\delta_2}(\sT^{\delta_2})), \\ 
	H^4_1&:=H^3_1 - E(\tau^{\delta_2}(\sT^{\delta_2})) =H^2_1 - E(\tau^{\delta_2}(\sT^{\delta_2}))\quad \text{and} \quad \label{eq:G4}
	H^4_2:=H^3_2-E(\tau^{\delta_2}(\sT^{\delta_2})).
\end{align}
This follows on from the previous update in \eqref{eq:G3} and again tracks the current set(s) of leftover edges. To track the total set of leftover edges, 
let 
\begin{align}\label{eq:G*4}
G^*_4:= G^4\cup G'^4\cup H^4_1 \cup H^4_2=G_3^*-E(\tau^{\delta_2}(\sT^{\delta_2})).
\end{align}
Note that
\begin{align}\label{eq: Delta E3}
\Delta(\tau^{\delta_2}(\sT^{\delta_2}))
\leq \Delta |\sT^{\delta_2}|
\stackrel{(\ref{eq: cF sizes})}{\leq} \hat{D}\Delta\delta_2^{1/2} n.
\end{align}
Thus by \eqref{eq: G*3,i quasi-random}, \eqref{eq: Delta E3}, and Proposition~\ref{prop: quasi-random subgraph} 
for every ${\I}\in [\hat{D}]$ and $j\in [D]$, we obtain
\begin{equation}\label{eq: G*4 quasirandom1}
\begin{minipage}[c]{0.8\textwidth}\em
$G^*_{4}[A_{\I}]$ and $G^*_{4}[A\cap U_j]$ are $(\gamma^{1/5},p)$-quasi-random.
\end{minipage}
\end{equation}
By \ref{item:G2},  \ref{item:G4}, \eqref{eq: G^1 deg}, \eqref{eq: G3 V0 density}, 
we have
\begin{align*}
	\Delta((G^4\cup {G'}^4 \cup H_2^4)[U_j\cap A, (U_j\cup R)\setminus A])\leq \delta_{2}^{1/2}n.
\end{align*}
Thus together with \eqref{eq: H12 quasi-random}, \eqref{eq: Delta E3}, and Proposition~\ref{prop: quasi-random subgraph}, 
we obtain that
\begin{equation}\label{eq: G*4 quasirandom2}
\begin{minipage}[c]{0.8\textwidth}\em
$G^*_4[U_j\cap A, U_j\setminus A]$ and $G^*_4[U_j\cap A, (U_j\cup R)\setminus A]$ are $(\delta_2^{1/3}, \delta_1)$-quasi-random.
\end{minipage}
\end{equation}

\step{Resolving the parity}\label{step6}
Note that by \eqref{eq: delta 2 cleared} and \eqref{eq:G*4},
in the current leftover graph $G^*_4$ every edge is incident to a vertex in $A$. We would like to cover the edges of $G^*_4[A,V\setminus A]$ via Lemma~\ref{lem: clear delta_1}. For this,
we first need to ensure 
for each $v\in V\sm A$ that the number of edges incident to $v$ which have not been covered yet is even.
To achieve this, in this step we extend the current packing by defining 
$\tau^{\eta}(\ell^c_F)$ for all $c\in [2]$ and $F\in \cF^{\eta}$.
(Recall that by the end of Step~\ref{step4} we have already defined the packing of $\{F- \{\ell_F^1,\ell_F^2\}:F\in \cF^\eta\}$.)
This will already resolve most of the parity problems. 
Note that by \eqref{eq: size of sF eta}, the total number of leaves of the form $\ell^c_F$ is close to $n$, so we do have sufficiently many of them for this purpose.
We then define $\tau^\p$ which packs $\sT^{\p}$ into $G\cup G'\cup H_1\cup H_2$.
This then takes care of the remaining (comparatively small) set of parity problems.
Recall that for each $F\in \cF^\eta$,
the vertex $z_F^c$ denotes the neighbour of $\ell_F^c$ in $F$.

For all $i\in [D]$, let
$$V^{\rm odd}_i:=\{v\in U_i\setminus A: d_{G^*_4}(v) \text{ is odd}\}.$$ 
For any $v\in V^{\rm odd}_i$,
our aim is to find a tree $F\in \cF^\eta$ so that $v$ can play the role of $\ell_F^c$ in $F$.
This will ensure that the degree of $v$ in the resulting remaining subgraph of $G^*_4$ is even.
Recall that for all $i\in[D], c\in [2]$ and $F\in \sF^{i,\eta}$, 
as described in Step~\ref{step4}, 
$\tau^\eta(z_F^c)\in U_i\cap A$ is the image of $z_F^c$ in our current packing.

Consider the set 
$$U^{z}_i:= \{(u,F,c): u\in U_i\cap A, u= \tau^{\eta}(z^c_F), c\in [2], F\in \sF^{i,\eta}\}.$$
Thus $(1+\gamma^{1/2})|U_i\sm A| \geq  2|\sF^{i,\eta}|=|U^{z}_i| \geq (1- \epsilon^{1/2}) |U_i \setminus A|$ by \eqref{eq: cFi eta size}.
We define a bipartite graph $H^i$ with vertex partition $(U^{z}_i,U_i\setminus A )$ 
and edge set
$$E(H^i) := 
\{(u,F,c)v: uv \in E(G^*_4), c\in [2], v\notin \tau^{\eta}(F)\cup W_F, F\in \sF^{i,\eta}\}.$$
\COMMENT{Note that we don't have to avoid edges $(w,F,c)u$ such that $u\in W_F$ because $W_F\subseteq A\cup R$. Thus the term "$\cup W_F$" is technically not necessary here. Felix: I don't care, just include the set.} 
By \eqref{eq:forestssizes} and \eqref{eq: WF size}, 
$|F|+|W_F|\leq  \eta^{1/2} n$ for each $F\in \sF^{i,\eta}$. 
This together with \eqref{eq: G*4 quasirandom2} implies that 
for all $(u,F,c),(u',F',c')\in U_i^z$ with $u\neq u'$, we have
 \begin{align}\label{eq: auxiliary graph degree}
d_{H^i}((u,F,c)) &= (1\pm \delta_2^{1/3})\delta_1 |U_i\setminus A|\pm \eta^{1/2} n 
= (1\pm 2 \delta_2^{1/3})\delta_1|U_i\setminus A|, \\
\label{eq: auxiliary graph codegree}
d_{H^i}((u,F,c),(u',F',c')) &= (1\pm \delta_2^{1/3})\delta_1^2 |U_i\setminus A| \pm 2\eta^{1/2} n
=(1\pm 2 \delta_2^{1/3})\delta_1^{2}|U_i\setminus A|.
 \end{align}

By \eqref{eq: not too many z}, 
the number of pairs $(u,F,c),(u',F',c')$ with $u=u'$ is at most $2\delta_2^{1/3} |U^z_i|^2$.
Thus \eqref{eq: auxiliary graph codegree} holds for all but at most $2\delta_2^{1/3} |U^z_i|^2$ pairs $(u,F,c),(u',F',c')$ 
and hence Theorem~\ref{thm: almost quasirandom} implies that $H^i$ is $(\delta_2^{1/20},\delta_1)$-regular.
Now Proposition~\ref{prop: matching} in turn implies that $H^i[U^z_i,V^{\rm odd}_i]$ contains 
a matching $M_i$ of size at least $|V^{\rm odd}_i| - 2\delta_2^{1/20} n$.

Next we define $\tau^\eta(\ell^c_F)$ for all $F\in \sF^{i,\eta}$.
For each edge $(u,F,c)v$ in $M_i$,
let $\tau^\eta(\ell^c_F) := v$. 
We will now embed the remaining ``unused'' leaves $\ell_F^c$ inside $A$ using Proposition~\ref{prop: sparse edge embedding}.
In this way, they do not affect the parity of the vertices outside $A$.
Let $E'_4:= \{uv: (u,F,c)v \in M_i, i\in [D]\}$. 
Then 
\begin{align}\label{eq: E'4 deg}
\Delta(E'_4)\leq n^{2/3}
\end{align} 
by \eqref{eq: not too many z} and the fact that $d_{E'_4}(v)\leq 1$ for any vertex $v\in U_{[D]}\sm A$.
Note that
\begin{align*}
	G_4^*[A\cap U_i]=(G_4^*-E_4')[A\cap U_i].
\end{align*}
For each $i\in [D]$, let
\begin{align*}
	\sF^{\rm match}_i:= &\{ (F,c) : (u,F,c) \in M_i \text{ for some } u\in A\cap U_i\},\\
\sF^{\rm unmatch}_i:= &\{(F,c) : (F,c)\notin \cF^{\rm match}_i, F\in \sF^{i,\eta}, c\in [2] \}.
\end{align*}
Thus we have already defined $\tau^\eta(\ell^c_F)$ for all $(F,c)\in \cF^{\rm match}_i$ and $i\in [D]$. 
Let $\{ (F^i_1,c_1),\dots, (F^i_{m'_i},c_{m'_i})\}:= \sF^{\rm unmatch}_i$ and 
let $u_{i,j}:= \tau^\eta(z_{F^i_{j}}^{c_j})$ for all $j\in [m'_i]$.
Then 
\begin{align}\label{eq:sizesFmatch}
	m'_i= |\sF^{\rm unmatch}_i| \leq 2| \cF^{\eta}| \stackrel{(\ref{eq: cF sizes})}{\leq} 2n.
\end{align}
Let $H'^{\rm index}_i$ be the graph on $[m'_i]$ such that $jj'\in E(H'^{\rm index}_i)$ if $F^i_j= F^i_{j'}$. 
Then $\Delta(H'^{\rm index}_i)\leq 1$.
Let 
\begin{align}\label{eq: Wij}
W^i_j:=
\tau^{\eta}({F^i_j})\cup W_{F^i_j}. 
\end{align}
Note that by \eqref{eq:forestssizes} and \eqref{eq: WF size},
$|W^i_j|\leq 2\Delta\eta n +\Delta^2 \leq \eta^{1/2} n$. 
Since $G^{*}_4[A\cap U_i]$ is $(\gamma^{1/5},p)$-quasi-random, by \eqref{eq: G*4 quasirandom1}, 
and since $u_{i,j}\in A\cap U_i$,
this implies that
$$d_{G^{*}_4,A\cap U_i}(u_{i,j}) - |W^i_j| 
\geq \gamma^2 n - \eta^{1/2} n 
\stackrel{\eqref{eq:sizesFmatch}}{\geq} 3 n^{2/3} + m'_i/n^{2/3} + n^{2/3}.$$ 
We now apply Proposition~\ref{prop: sparse edge embedding} with the following parameters and graphs for each $i\in [D]$.
\newline

\noindent
{
\begin{tabular}{c|c|c|c|c|c|c|c|c}
object/parameter & $G^{*}_4[A\cap U_i]$ & $H'^{\rm index}_i$ & $ n^{2/3} $ & $A\cap U_i$ &   $m'_i$ & $n^{2/3} $  &  $W^i_j$ & $u_{i,j}$
\\ \hline
playing the role of & $G$ & $H$ & $\Delta$ & $A$ & $m$ & $s$ &$W_j$ & $u_j$
\end{tabular}
}\newline \vspace{0.2cm}

\noindent
(Condition (ii) of Proposition~\ref{prop: sparse edge embedding} follows from \eqref{eq: not too many z}.)
We obtain distinct edges $u_{i,1}v_{i,1},\ldots, u_{i,m_i'}v_{i,m_i'}$ in $G^{*}_4[A\cap U_i]$ 
such that for each $j\in[m_i']$, $v_{i,j}\notin W_j^i$, and for each $v \in A$, 
\begin{align}\label{eq: E''4 deg}
|\{(i,j): v=v_{i,j}\}|\leq n^{2/3}
\end{align}
and such that $v_{i,j}\neq v_{i,j'}$ if $F_j^i=F_{j'}^i$.
Since $A\cap U_i$ are disjoint for different $i\in [D]$, 
each application of Proposition~\ref{prop: sparse edge embedding} gives us distinct edges.

For each $i\in [D]$ and $j\in [m'_i]$, let $\tau^{\eta}(\ell^{c_j}_{F^i_j}) := v_{i,j}$.
Then $\tau^{\eta}(\ell^{c_j}_{F^i_j}) \notin \tau^{\eta}({F^i_j})$ by \eqref{eq: Wij}.
In addition, $\tau^{\eta}(\ell^{c_j}_{F^i_j})\neq \tau^{\eta}(\ell^{3-c_j}_{F^i_j})$.%
\COMMENT{If $(F^i_j,c)\in F^{\rm unmatch}_i, (F^i_j,3-c)\in F^{\rm match}_i$, then $\tau^\eta(\ell^{3-c}_{F^i_j})$ is not in $A$, so we have $\tau^\eta(\ell^{c}_{F^i_j})\neq \tau^\eta(\ell^{3-c}_{F^i_j})$ since one of them is in $A$ and the other one is not in $A$. 
Otherwise, $H'^{\rm index}_i$ will ensure $\tau^\eta(\ell^{c}_{F^i_j})\neq \tau^\eta(\ell^{3-c}_{F^i_j})$}
Altogether, this defines $\tau^\eta(\ell^c_F)$ for all $F\in \cF^{\eta}$.
Using \eqref{eq:tauWF} and \ref{item:Q11i}, it is easy to see that
$\tau^\eta\cup \tilde{\tau}$ now packs $\cF^\eta$ into $G\cup G'\cup H_1 \cup H_2$.

Let $E''_4:=\{u_{i,j}v_{i,j}: i\in [D], j\in m'_i\}$. 
Let $E_4:=E'_4\cup E''_4$, so \eqref{eq: not too many z}, \eqref{eq: E'4 deg} and \eqref{eq: E''4 deg} imply that
\begin{align}\label{eq: E4 deg}
\Delta(E_4)\leq 3n^{2/3}.
\end{align}
In what follows, we use Lemma~\ref{lem: clear parity} and the trees in $\sT^\p$ to adjust the parity of those vertices in $V\sm A$
which were not involved in the above matching approach (in particular, we now also adjust the parity of the vertices in $R$).
Let 
\begin{align}\label{eq:G5}
	G^*_5:= G^*_4 - E_4
\end{align}
 and
for each $i\in [D]$, let 
$$V^\p_i:= \left\{ \begin{array}{ll}
\{ u\in (U_i\cup R)\setminus A: d_{G^*_5}(u) \text{ is odd}\} &\text{ if }i=1,\\
\{u\in U_i \setminus A : d_{G^*_5}(u) \text{ is odd}\} &\text{ if }i\neq 1.
\end{array}\right.$$
As $|R|= \epsilon n$ and  $|V^{\rm odd}_i|-|M_i|\leq 2\delta^{1/20}_2 n$,
this implies that $|V^{\p}_{i}|\leq 3\delta_2^{1/20} n$.
Let $G^5_1:= G^*_5[A\cap U_1,(U_1\cup R)\sm A]$ and $G^5_i:= G^*_5[A\cap U_i,U_i\sm A]$ for $i\geq 2$. 
For each $i\in [D]$, we let $\{u'_{i,1},\dots, u'_{i,m''_i}\} := V^{\p}_{i}$.

Observe that \eqref{eq: G*4 quasirandom2} and \eqref{eq: E4 deg} imply $d_{G^5_i,A\cap U_i}(u'_{i,j}) 
{\geq} \delta_1^2 n \geq 3+ m''_i/n^{2/3} + n^{2/3}$.
Let $H$ consist of $m_i''$ isolated vertices.
We can apply Proposition~\ref{prop: sparse edge embedding} with the following parameters and graphs.
\newline

\noindent
{
\begin{tabular}{c|c|c|c|c|c|c|c|c}
object/parameter & $G^{5}_i$ & $H$ & $ 1 $ & $A\cap U_i$ &  $m''_i$ & $n^{2/3} $  &  $\emptyset$ & $u'_{i,j}$
\\ \hline
playing the role of & $G$ & $H$ & $\Delta$ & $A$ &  $m$ & $s$ &$W_i$ & $u_j$
\end{tabular}
}\newline \vspace{0.2cm}

\noindent
We obtain a set of distinct edges $E_5^i:= \{ u'_{i,j}v'_{i,j}: j\in [m''_i]\}$ in $G^5_i$ such that 
$\Delta(E_5^i)\leq n^{2/3}$  and $d_{E_5^i}(u)=1$ for all $u\in V^{\p}_{i}$.
Furthermore, 
\begin{align}\label{eq: E5i size}
|E_5^i|
= |V_i^\p|
\leq 3\delta_2^{1/20} n 
\leq \delta_1^2 n.
\end{align}
Now we wish to extend each $e\in E_5^i$ into a tree from $\sT^{i,\p}$ such that all other edges of this tree lie in $A$. (Note that it is not possible to
only use trees in $\sT^{i,{\rm par}}$ to correct the parities of all vertices in $V^{\rm odd}_i$ as $|V^{\rm odd}_i|$ may be much larger than $|\sT^{i,{\rm par}}|$ and each tree in $ \sT^{i,{\rm par}}$ might contain at most two vertices of odd degree. That is why we already have corrected the parities of vertices in $V^{\rm odd}_i\setminus V_i^{\rm par}$ by using~$E_4$.)
To achieve this,
for each $i\in [D]$, we apply Lemma~\ref{lem: clear parity} with the following graphs and parameters. 
Let $U'_1:= U_1\cup R$ and $U'_i:=U_i$ for $i\geq 2$. 
(Recall that $\tau'(y_T)$ was defined in Step~\ref{step3} and the domain of $\tau'$ on $T\in\sT^{i,{\rm par}}$ is exactly $\{y_T\}$.)\newline

\noindent
{
\begin{tabular}{c|c|c|c|c|c|c}
object/parameter & $G^*_5[A\cap U_i]\cup E_5^i$ & $\sT^{i,\p}$ & $\gamma $ & $U'_i$ &$\epsilon$ & $\delta_1$   \\ \hline
playing the role of & $G$ & $\cT$ & $\gamma$ & $V(G)$ & $\epsilon$ & $\delta_1$  
\\ \hline \hline
object/parameter &  $A\cap U_i$  & $p$   &  $\Delta$ & $y_{T}$ & $\tau'|_{\{y_T\}}$ & $W_{T}$  \\ \hline
playing the role of & $A$ & $p$& $\Delta$  & $y_{T} $ & $\tau'_{T}$ & $W_{T}$ 
\end{tabular}
}\newline \vspace{0.2cm}

\noindent
Note that \eqref{eq: G*4 quasirandom1} and \eqref{eq: E4 deg} together with Proposition~\ref{prop: quasi-random subgraph} imply condition \ref{item:L31} of Lemma~\ref{lem: clear parity}\COMMENT{$E_5^i$ removal is irrelevant here}, 
and \eqref{eq: delta 2 cleared} implies \ref{item:L32}, 
\eqref{eq: E5i size} and \eqref{eq: cF sizes} imply \ref{item:L33}, 
and \eqref{eq:forestssizes}, \eqref{eq: WF size} and \ref{item:gamma2} imply \ref{item:L34}. 
Condition \ref{item:L35} holds because of \eqref{eq: not too many W_F} and \ref{item:gamma1}. 
Thus Lemma~\ref{lem: clear parity} gives a function $\tau^\p_i$ packing $\sT^{i,\p}$ into $G_5^*[A\cap U_i]\cup E_5^i$ which is consistent with $\tau'$
and satisfies $E_5^i\sub E(\tau_i^\p(\sT^{i,\p}))$
and $\tau_i^\p(T)\cap W_T=\es$ for every $T\in \sT^{i,\p}$.
We let $\tau^\p:= \bigcup_{i=1}^{D}\tau^\p_i$. 
Note that
\begin{align}
\Delta(\tau^\p(\sT^\p))
&\leq \Delta |\sT^{\p}|\label{eq: E5 deg}
\stackrel{(\ref{eq: cF sizes})}{\leq} \hat{D}\Delta \gamma_*n.
\end{align} 
Moreover, using \eqref{eq:tauWF} and \ref{item:gamma4} it is easy to see that $\tilde{\tau}\cup \tau^{\p}$ packs $\cF^{\p}$ into $G\cup G'\cup H_1\cup H_2$.
Let 
\begin{align}\label{eq:G6}
	G^*_6:= G^*_5 - E(\tau^\p(\sT^\p))=G^*_4 - E_4 - E(\tau^\p(\sT^\p)).
\end{align}
After the updates in \eqref{eq:G*4} and \eqref{eq:G5}, $G^*_6$ consists of the current set of leftover edges in $G\cup G'\cup H_1\cup H_2$.
Note that by \eqref{eq: G*4 quasirandom1}, \eqref{eq: E4 deg}, \eqref{eq: E5 deg}, and Proposition~\ref{prop: quasi-random subgraph} for each ${\I}\in [\hat{D}]$,
\begin{equation}\label{eq: G*6 quasirandom}
\begin{minipage}[c]{0.8\textwidth}\em
$G^*_{6}[A_{\I}]$ is $(\gamma^{1/6},p)$-quasi-random.
\end{minipage}
\end{equation}
By construction every vertex in $V\sm A$ has even degree in $G_6^*$.

\step{Covering the edges between $A$ and $V\sm A$}\label{step7}
We finally complete the proof by covering the edges in $G_6^*[A,V\sm A]$ (i.e.~all remaining edges between $A$ and $V\sm A$) by the trees in $\sT^{\delta_1}$.
Again, we embed the trees sequentially according to the order $\sT^{\delta_1}_1,\dots, \sT^{\delta_1}_{\hat{D}}$. 
Assume that for some $\I\in [\hat{D}]$ we have already defined functions $\tau^{\delta_1}_{1},\ldots,\tau^{\delta_1}_{{\I-1}}$ such that 
for all $\I'<\I$, the function
$\tau^{\delta_1}_{\I'}$ packs $\sT^{\delta_1}_{{\I}'}$ into $G^6_{{\I'}}$, where
$$G^6_{\I'}:= G^*_6 - \bigcup_{{\I}''=1}^{{\I}'-1} E(\tau^{\delta_1}_{{\I}''}(\sT^{\delta_1}_{\I''})).$$
Recall from \eqref{eq:defAB} that $B_\I=(U_{[D]\sm S_\I}\cup R)\sm A$.
Note that \ref{item:G2} implies $G_I^6[A_I, B_I] \subseteq G^1\cup H_1\cup H_2$.
Thus \eqref{eq: G^1 deg}, \eqref{eq: G3 V0 density}, \ref{item:G3}, \ref{item:G4} and \ref{item:U2} imply that 
\begin{align}\label{eq: degree between}
\Delta(G^6_{\I}[A_{\I}, B_\I])\leq 2\delta_1 |A_\I\cup B_\I|.
\end{align}
In addition,  $G^6_{\I}[A_{\I}]$ is $(\gamma^{1/7},p)$-quasi-random by \eqref{eq: G*6 quasirandom}, Proposition~\ref{prop: quasi-random subgraph} and the fact that 
$$\Delta(\bigcup_{{\I}'=1}^{{\I}-1} \tau^{\delta_1}_{{\I'}}(\sT^{\delta_1}_{\I'}))
 \leq \Delta|\cF^{\delta_1}|
  \stackrel{\eqref{eq: cF sizes}}{\leq} \Delta \hat{D}\gamma_* n
  \leq \gamma n.$$ 
Now for each ${\I}\in [\hat{D}]$, 
we apply Lemma~\ref{lem: clear delta_1} with the following graphs and parameters. \newline

\noindent
{
\begin{tabular}{c|c|c|c|c|c|c|c|c|c|c|c}
object/parameter & $G^6_{\I}[A_\I\cup B_\I]$ & $\sT^{\delta_1}_{{\I}}$ & $A_\I$ & $\delta_1$ & $\gamma$ & $\Delta$ & $p$   & $y_{T}$ & $\tau'|_{\{y_T\}}$ & $W_{T}$ & $\epsilon$ \\ \hline
playing the role of & $G$ & $\cT$ & $A$ &$\delta_1$ & $\gamma$ & $\Delta$  & $p$ & $y_{T} $ & $\tau'_{T}$ & $W_{T}$  & $\epsilon$
\end{tabular}
}\newline \vspace{0.2cm}

\noindent
Observe that condition \ref{item:L41} of Lemma~\ref{lem: clear delta_1} holds since $G^6_{\I}[A_{\I}]$ is $(\gamma^{1/7},p)$-quasi-random,
\eqref{eq: delta 2 cleared} implies \ref{item:L42}, 
\eqref{eq: degree between} implies \ref{item:L43}, 
\eqref{eq: cF sizes} implies \ref{item:L44}, 
\eqref{eq:forestssizes}, \eqref{eq: WF size}, and \ref{item:gamma3} imply \ref{item:L45}, and 
\eqref{eq: not too many W_F} and \ref{item:gamma1} give us \ref{item:L46}.
Lemma~\ref{lem: clear delta_1} provides a function $\tau^{\delta_1}_{\I}$ packing $\sT^{\delta_1}_{\I}$ into $G_{\I}^6[A_\I\cup B_\I]$ 
which is consistent with $\tau'$ and satisfies
\begin{enumerate}[label=(Q4.\arabic*)$_{\I}$]
\item\label{item:Q41i} $d_{G^6_{\I+1},A_\I}(v) \leq 1$ for every $v\in B_\I$,
\item\label{item:Q42i} $W_{T}\cap \tau^{\delta_1}_{\I}(T) =\es$ for every $T\in \sT^{\delta_1}_{\I}$, and
\item\label{item:Q43i} $d_{G^6_{\I}, A}(v) - d_{G^6_{{\I}+1},A}(v)$ is even for all $v\in V\sm A$. 
\end{enumerate}

Indeed, to check \ref{item:Q43i} note that for all $v\in B_\I$
we have $d_{G^6_{\I}, A\sm A_\I}(v) = d_{G^6_{{\I}+1},A\sm A_\I}(v)$, 
that $d_{G^6_{\I}, A_\I}(v) - d_{G^6_{{\I}+1},A_\I}(v)$ is even by \ref{item:Q43} of Lemma~\ref{lem: clear delta_1},
and that for all $v\in (V\sm A)\sm B_I$ we have $d_{G^6_{\I}, A_\I}(v) = d_{G^6_{{\I}+1},A_\I}(v)$.

Let $\tau^{\delta_1}:= \bigcup_{{\I}=1}^{\hat{D}} \tau^{\delta_1}_{\I}$, then $\tau^{\delta_1}$
packs $\sT^{\delta_1}$ into $G_6^*$ and is consistent with $\tau'$. 
Moreover, recall that $A_I\cup B_I = U_{[D]\setminus S_I} \cup R$.
This together with \ref{item:gamma4}, \eqref{eq:WF1},
and \ref{item:Q42i} implies that
$\tilde{\tau}\cup \tau^{\delta_1}$ packs $\cF^{\delta_1}$ into $G\cup G'\cup H_1\cup H_2$.

Let 
\begin{align}\label{eq:G7}
	G^*_7:= G^*_6 - E(\tau^{\delta_1}(\sT^{\delta_1}))=G^6_{\hat{D}+1}.
\end{align}
After the previous update in \eqref{eq:G6}, this is the (final) set of edges left over by the packing we defined so far.
We claim that for all $v\in V\sm A$,
we have $d_{G^*_7}(v)=0$.
Assume for a contradiction that there is a vertex $v\in V\setminus A$, with $d_{G^*_7}(v)>0$. 
Thus $d_{G^*_7,A}(v)>0$ by \eqref{eq: delta 2 cleared}.
Moreover, $d_{G^*_7,A}(v)\geq 2$, since $v$ has even degree in $G_6^*=G^6_1$ and \ref{item:Q43i} holds for all ${\I}\in [\hat{D}]$. 
Let $i,i',i^{*}\in[D]$ be such that $v\in U_i \cup R$, and $v$ has a neighbour in $U_{i'}\cap A$ and another neighbour in $U_{i^*}\cap A$ in $G^*_7$. 
Consider ${\I}\in [\hat{D}]$ such that $S_{{\I}} \sub [D]\setminus\{i,i',i^{*}\}$. 
Then $d_{G^6_{\I+1},A_{\I}}(v)\geq d_{G^*_7,A_{\I}}(v) \geq 2$, which is a contradiction to \ref{item:Q41i}
and proves the claim. (This is the point where we use that $|S_{I}|=D-3$.)

We define $\phi:= \tilde{\tau}\cup\tau^{\eta} \cup \tau^{\delta_1} \cup \tau^{\delta_2}\cup \tau^{\p}$.
Thus $\phi$ packs $\cF$ into $G\cup G'\cup H_1\cup H_2$.
The above claim and \ref{item:G1} imply \ref{item:Phi1}.
Since the sequence of all updates of the leftover edges is given by
\eqref{eq: def G1 H1}, \eqref{eq:G2}, \eqref{eq:G3}, \eqref{eq:G4}, \eqref{eq:G*4}, \eqref{eq:G5}, \eqref{eq:G6}, and \eqref{eq:G7}, we have 
$E(G'\cap \phi(\cF))\sub E_1\cup E_4 \cup(\tau^{\delta_1} \cup \tau^{\delta_2} \cup \tau^{\p})(\sT^{\gamma_*})$.
As $|\sT^{\gamma_*}|\leq 3\hat{D}\gamma_* n$ by \eqref{eq:sizeFgamma}, 
we conclude  by \eqref{eq: Delta E1} and \eqref{eq: E4 deg} that each $v\in A$ satisfies
\begin{align*}
d_{G'\cap \phi(\cF)}(v)\leq
d_{E_1}(v) + d_{E_4}(v) + \Delta|\sT^{\gamma_*}| \leq \epsilon n + 3 n^{2/3} + 3\hat{D} \Delta\gamma_* n  \leq \gamma_*^{1/2}|A|.
\end{align*}
This implies \ref{item:Phi2}. 
Note that here we make crucial use of the fact that when we packed the collection of forests which contain many trees (i.e.~$\tilde{\cF}$ and $\sF^{\eta}$) we did not use any edges of $G'$.
\end{proof}

Recall that the aim of Step~2 was to find a near-optimal packing of trees which cover most of the edges of $G$. Since $G$ is quasi-random one could directly apply the results of \cite{KKOT16} to achieve this. In other words, it would seem much more straightforward to apply Theorem~\ref{thm: any graph packing} from \cite{KKOT16} directly to obtain such a packing, rather than proving and applying the results in Section~\ref{sec: trees}--\ref{sec:iteration}, which are based on more technical results from \cite{KKOT16} and on Szemer\'{e}di's regularity lemma. However, the seemingly more straighforward approach would lead to the leftover density after Step~2 being too large (compared to that of $H_1$ and $H_2$) for the remaining steps to be feasible.

\section{Orientations with regular outdegree}
\label{sec:orientation}

The following lemma states that every quasi-random graph has an ``out-regular'' orientation if the average degree is an even integer.
We will need Lemma~\ref{lem: orientation} at the very end of the proof of our main theorem 
where we need to embed only a single leaf for each tree in a given collection of trees.
It will turn out that such an ``out-regular'' orientation of the remaining uncovered graph $G$ will give rise to a valid embedding.
We use $\bar{d}(G)$ to denote the average degree of a graph $G$.

\begin{lemma}\label{lem: orientation}
Suppose $n, \bar{d} \in \N$ and $1/n \ll \beta \ll p \leq 1$.
If $G$ is a $(\beta,p)$-quasi-random graph on $n$ vertices
such that $\bar{d}(G)=2\bar{d}$, then $G$ has an orientation such that every vertex has exactly $\bar{d}$ outneighbours.
\end{lemma}
\begin{proof}
We start by introducing for any graph $H$ a function $Z(H)$ which measures the distance of the degree sequence of $H$ to a $\bar{d}(H)$-regular degree sequence.
We define $$Z(H):=\sum_{v\in V(H)}|d(v)-\bar{d}(H)|.$$
Since $G$ is $(\beta,p)$-quasi-random, we have $\Delta(G)- \delta(G)\leq 2p\beta n$ and thus $Z(G)\leq 2p\beta n^2$. 
We iteratively construct a sequence of graphs $G=G_0\supseteq G_1\supseteq\ldots\supseteq G_s$ (for some $s\in \N\cup \{0\}$) such that 
$G_s$ is regular and for all $j\in [s]$ the following hold:
\begin{enumerate}[label=(D\arabic*)$_j$]

	\item\label{item:D1}  $\bar{d}(G_{j})= 2(\bar{d}-j)$,
	\item\label{item:D3}  $\Delta(G_{j})-\delta(G_{j})\leq \Delta(G_{j-1})-\delta(G_{j-1})$ and
	$Z(G_{j})\leq Z(G_{j-1})$,
	\item\label{item:D4}  $\Delta(G_{j})-\delta(G_{j})\leq \Delta(G_{j-1})-\delta(G_{j-1})-1$ or
	$Z(G_{j}) = Z(G_{j-1}) -p^2 n/2$,
	\item\label{item:D5}  $G_{j-1}- E(G_{j})$ has an orientation such that every vertex has outdegree $1$, and
	\item\label{item:D6}  $\Delta(G_{j-1}- E(G_{j}))\leq 3$.
\end{enumerate}
Having defined such a sequence $G=G_0\supseteq G_1\supseteq\ldots\supseteq G_s$,
by (D1)$_s$ the graph $G_s$ is Eulerian and so has an Eulerian circuit $W$.
By orienting $W$ consistently and orienting edges in $G_{j-1}-E(G_{j})$ as in \ref{item:D5} for each $j\in [s]$, this leads to the desired orientation of $G$.

Note that (D1)$_{0}$ trivially holds and we view (D2)$_0$--(D5)$_0$ as being vacuously true.
Suppose that for some $i\geq 0$ we have already defined graphs $G_0,\ldots,G_{i}$ such that \ref{item:D1}--\ref{item:D6} hold for all $j\leq i$.
By \ref{item:D3} and \ref{item:D4} for every $j\leq i$, 
we conclude  that 
\begin{align}\label{eq:orient}
	i\leq (p^2n/2)^{-1}Z(G)+ (\Delta(G)-\delta(G)) \leq 4\beta p^{-1}n+ 2p\beta n \leq 6\beta p^{-1}n.
\end{align}
If $G_i$ is regular, then set $s:=i$. Suppose next that $G_i$ is not regular. We show how to define $G_{i+1}$.

The fact that \ref{item:D6} holds for all $j\leq i$,
together with
\eqref{eq:orient} implies that $d_{G_i}(v) \geq d_{G}(v)-3i \geq d_{G}(v) - 18 \beta p^{-1} n$. 
Thus $G_i$ is $(\beta^{1/2},p)$-quasi-random by Proposition~\ref{prop: quasi-random subgraph},
and hence $(\beta^{1/12},p/2)$-dense, by Proposition~\ref{prop: quasi-random implies dense}.

Let $U$ be the set of vertices of maximum degree in $G_i$ and let $V$ be the set of vertices of minimum degree in $G_i$. 
Let $t:=\min\{|U|,|V|, p^2 n/4\}$.
Let $U':=\{u_1,\ldots,u_t\}\subseteq U$ and $V':=\{v_1,\ldots,v_t\}\subseteq V$.
Because $G_i$ is $(\beta^{1/2},p)$-quasi-random and thus $d_{G_i}(u,v)\geq (1-\beta^{1/2})p^2 n \geq 3t$ for all $u,v\in V(G_i)$, 
there is a set $W':=\{w_1,\ldots,w_t\}\subseteq V(G)\sm (U'\cup V')$ such that $P_j:=u_jw_jv_j$ is a path in $G_i$ for all $j\in [t]$.

As $G_i$ is $(\beta^{1/12},p/2)$-dense,
by Proposition~\ref{prop: dense induced sub}, the graph $G':=G_i-(V'\cup W')$ is a robust $(p\beta^{1/12}/2,4\beta^{1/12})$-expander 
and $\delta(G')\geq (1-\beta^{1/2})pn - 2t \geq pn/3$. 
Thus, by Theorem~\ref{thm: expander ham}, $G'$ has a Hamilton cycle $C$.
Let $G_{i+1} := G_i - E(C) - \bigcup_{j=1}^{t} E(P_j)$.
It is easy to see that $G_{i+1}$ satisfies (D1)$_{i+1}$,  (D4)$_{i+1}$ and (D5)$_{i+1}$.

As, by (D1)$_{i}$, 
the average degree of $G_i$ is an integer, we have $\delta(G_{i})+1\leq \bar{d}(G_i)\leq \Delta(G_i)-1$. 
Moreover, (D1)$_{i+1}$ implies $\overline{d}(G_{i+1})= \overline{d}(G_i)-2$. 
It is easy to check that together with our construction of $G_{i+1}$ this implies that $|d_{G_i}(v)-\bar{d}(G_i)|\geq |d_{G_{i+1}}(v)-\bar{d}(G_{i+1})|$ for all $v\in V(G)$, 
and hence (D2)$_{i+1}$ follows.

If $t=\min\{|U|,|V|\}$,
then $\Delta(G_{i+1})-\delta(G_{i+1})\leq \Delta(G_i)-\delta(G_i)-1$ holds
and if $t=p^2n/4$,
then $Z(G_{i+1})=Z(G_i)- p^2 n/2$.
Thus we have (D3)$_{i+1}$. 
This completes the construction of $G_{i+1}$ satisfying (D1)$_{i+1}$--(D5)$_{i+1}$.
\end{proof}

\section{Proof of Theorem~\ref{thm: main result}}
\label{sec:final}

In this section we prove Theorem~\ref{thm: main result}.
In Section~\ref{sec:9.1}
we first state and prove Theorem~\ref{thm: main result step1} (which is weaker than Theorem~~\ref{thm: main result}).
We extend Theorem~\ref{thm: main result step1} to Theorem~\ref{thm: main result} in Section~\ref{sec:9.2}
by combining it with a result from~\cite{KKOT16}.

\subsection{Optimal tree packing}\label{sec:9.1}

\begin{theorem}\label{thm: main result step1}
For all $\Delta \in \N$ and $\delta>0$, 
there exist $N\in \N$ and $\epsilon>0$ such that
for all $n\geq N$ the following holds.
Suppose $G$ is an $(\epsilon,p)$-quasi-random graph on $n$ vertices and $\cT$ a set of trees satisfying 
\begin{enumerate}[label=(\roman*)]
	\item\label{item:MT11} $\Delta(T)\leq \Delta$ and $\delta n \leq |T|\leq (1- \delta)n$ for all $T\in \cT$,
	\item\label{item:MT13} $|\cT|\geq (1/2+ \delta)n$, and
	\item\label{item:MT14} $e(\cT)=e(G)$.
\end{enumerate}
Then $\cT$ packs into $G$.
\end{theorem}

Note that the conditions of the theorem imply that $p\geq \delta$.

We will use an iterative approach for the proof of Theorem~\ref{thm: main result step1}.
As explained earlier,
in each iteration step,
we will apply Lemma~\ref{lem: iteration} to cover most of the current leftover graph with forests.
To this end, we introduce a vortex of a quasi-random graph 
(a similar notion was already introduced in~\cite{GKLMO16} to underpin the iterative absorption process carried out there).
Suppose $G$ is an $(\epsilon,p)$-quasi-random graph on $n$ vertices and $1/n\ll\epsilon\ll\gamma\ll p$.
A $\gamma$-\emph{vortex} of $G$ is a collection of vertex sets $A_0,\ldots, A_\Lambda, R_0,\ldots, R_{\Lambda-1}$
such that
\begin{enumerate}[label=(V\arabic*)]
	\item\label{item:O1} $A_\Lambda \sub \ldots \sub A_0=V(G)$,
	\item\label{item:O2} $|A_{i}|=\lfloor\gamma^in\rfloor=:n_i$,
	\item\label{item:O3} $n_\Lambda\leq n^{1/3}$ and $\Lambda$ is minimal with respect to this,
	\item\label{item:O4} $R_i\sub A_i\sm A_{i+1}$ and $|R_i|=\lfloor\epsilon n_i\rfloor$ for every $i \in\{0,\ldots, \Lambda-1\}$, 
	\item\label{item:O5} $d_{G,A_i}(u)= (1\pm 3\epsilon/2)p|A_i|$ and $d_{G,A_i}(u,v)= (1\pm 3\epsilon/2)p^2|A_i|$
	for all $i\in \{0,\ldots, \Lambda\}$ and  all distinct $u,v\in V(G)$, and
	\item\label{item:O6} $d_{G,R_i}(u)= (1\pm 2\epsilon)p|R_i|$ and $d_{G,R_i}(u,v)= (1\pm 2\epsilon)p^2|R_i|$
	for every $i\in \{0,\ldots, \Lambda-1\}$ and all distinct $u,v\in V(G)$.
\end{enumerate}

Is is not difficult to see that $G$ has a $\gamma$-vortex,
as the following random process produces a $\gamma$-vortex with probability, say, at least $1/2$.
Let $\Lambda$ and $n_i$ be defined as in \ref{item:O2} and \ref{item:O3}.
Consider a random partition $(U_0,\ldots,U_\Lambda)$
such that $|U_\Lambda|=n_\Lambda$ and 
$|U_i|=n_i-n_{i+1}$ for every $i \in\{0,\ldots, \Lambda-1\}$.
Let $A_i:=\bigcup_{j=i}^\Lambda U_j$. 
Then \ref{item:O1}--\ref{item:O3} hold.
Lemma~\ref{lem: chernoff} shows that \ref{item:O5} holds with probability at least $1/2$.
Now it is easy to construct a $\gamma$-vortex of $G$, as we only need to find suitable sets $R_i$.
Indeed, random sets $R_0,\ldots, R_{\Lambda-1}$ where $R_i\sub U_i$ and $|R_i|=\lfloor\epsilon n_i\rfloor$ have the desired properties with probability at least $1/2$, by Lemma~\ref{lem: chernoff}.
In particular, such sets exist and \ref{item:O6} holds.

We make the following observations, which follow directly from the definition of a vortex:
\begin{itemize}
	\item for each $i\in [\Lambda]$, the bipartite graph $G[A_{i-1}\sm A_i,A_i]$ is $(2\epsilon, p)$-quasi-random, and
	\item $G[A_{i-1}] -E(G[A_i])$ is $(\gamma^{3/4} ,p )$-quasi-random by Proposition~\ref{prop: quasi-random subgraph}. 
\end{itemize}

When applying Lemma~\ref{lem: iteration} in the proof of Theorem~\ref{thm: main result step1}, 
$A_{i+1}$ will play the role of $A$ and $R_i$ the role of $R$.

\begin{proof}[Proof of Theorem~\ref{thm: main result step1}]
First, we choose the following constants:
$$0<1/N \ll \epsilon \ll \delta_2 \ll \delta_1 \ll \gamma_*\ll \gamma \ll \beta \ll\delta,1/\Delta.$$
In view of the statement of Theorem~\ref{thm: main result step1}, 
we may assume $1/\Delta \leq \delta/2$ by increasing the value of $\Delta$ if necessary. Let $G$ be a graph satisfying the conditions of Theorem~\ref{thm: main result step1}.
As observed above, $p\geq \delta$ and $G$ contains a $\gamma$-vortex $A_0,\ldots, A_\Lambda, R_0,\ldots, R_{\Lambda-1}$.

Note that for all distinct $i,j\in \{0,\ldots,\Lambda-1\}$, 
the graphs $G[A_i] - E(G[A_{i+1}])$ and $G[A_j] - E(G[A_{j+1}])$ are edge-disjoint.
For every $i\in \{0,\ldots,\Lambda-1\}$, 
we decompose $G[A_i] - E(G[A_{i+1}])$ into three spanning  subgraphs $G_i, H^i_1, H^i_2$ such that 
\begin{enumerate}[label=(H\arabic*)$_i$]
\item\label{item:H1} $H^i_1$ is a $(4\epsilon,\delta_1)$-quasi-random bipartite graph with vertex partition $(A_{i}\setminus A_{i+1}, A_{i+1})$,
\item\label{item:H2} $d_{H^i_2}(u)\leq 3\delta_2 n_{i}/2$ and $d_{H^i_2}(u,v) \geq 2\delta_2^2 n_{i}/3$ for any two vertices $u,v\in A_{i}$,
\item\label{item:H3} $G_i$ is $(\gamma^{1/2},p)$-quasi-random and if $i\leq \Lambda-2$, then $d_{G_i,R_{i+1}}(v)\geq p|R_{i+1}|/2$ for any vertex $v \in A_i\setminus A_{i+1}$, and in addition, if $i=0$, then $d_{G_0,R_{0}}(v)\geq p\epsilon n/2$ for all $v\in V(G)$, and
\item\label{item:H4} $A_{i+1}$ is an independent set in $H_1^i\cup H^i_2\cup G_i$.
\end{enumerate}
It is not difficult to see that such a decomposition exists.
Indeed,
if we assign any edge in $G[A_{i}\setminus A_{i+1},A_{i+1}]$
to $H^i_1$ with probability $\delta_1/p$,
to $H^i_2$ with probability $\delta_2/p$, and
to $G_i$ with probability $1-(\delta_1+\delta_2)/p$
and any edge in $G[A_i\sm A_{i+1}]$
to $H^i_2$ with probability $\delta_2/p$, and
to $G_i$ with probability $1-\delta_2/p$,
then we obtain graphs
satisfying these conditions, with probability at least $1/2$.

Let $G_{\Lambda}:=G[A_{\Lambda}]$.
Before we proceed with an inductive argument, we carry out the following preparation step. \newline

\noindent {\bf Preparation step.}
We will first remove a leaf from those trees which will play a role in the final absorbing step and
embed the remainder of these trees into $G_0$ in a suitable way.
For this, let $t_{\Lambda}$ be an integer such that 
\begin{align}\label{eq:edgesGlast}
	t_{\Lambda} n_{\Lambda} = e(G[A_{\Lambda}])- 3\gamma_*^2 n_{\Lambda-1}^2 \pm n_{\Lambda},
\end{align}
and consider an arbitrary collection of trees 
\begin{align}\label{eq:laststeptrees}
	\cT^*:= \{ T^*_1,\dots, T^*_{t_{\Lambda}n_{\Lambda}}\} \subseteq \cT.
\end{align}

Let $\ell^*_i$ be a leaf of $T^*_i$ and $z^*_i$ be the unique neighbour of $\ell^*_i$ in $T^*_i$. 
For each $i\in [t_{\Lambda}n_{\Lambda}]$, let $T_{*,i}:= T^*_i - \ell^*_i$ and $\cT_*:=\{ T_{*,1},\dots, T_{*,t_{\Lambda}n_{\Lambda}}\} $.
Let $v^*_1,\dots, v^*_{t_{\Lambda}n_{\Lambda}}$ be a sequence of not necessarily distinct vertices such that 
every vertex in $A_{\Lambda}$ appears in the sequence exactly $t_{\Lambda}$ times. 
We sequentially construct functions $\phi_{T_{*,i}}$ which pack $T_{*,i}$ into $G_0[(A_0\setminus A_1)\cup \{v^*_i\}]$ such that $\phi_{T_{*,i}}(z^*_i)= v^*_i$.  
(The embedding of $\ell_1^*,\ldots,\ell_{t_\Lambda n_\Lambda}^*$ will be deferred until the final step.)
Assume that for some $i\in [t_\Lambda n_\Lambda]$ we already have constructed $\phi_{T_{*,1}},\dots, \phi_{T_{*,i-1}}$, and let 
$$G(i):= G_0[ (A_0\setminus A_1) \cup \{v^*_i\}] - \bigcup_{j=1}^{i-1}E(\phi_{T_{*,j}}(T_{*,j})).$$
Note that $|T_{*,i}|\leq (1- \delta)n\leq (1-\gamma)n+1 = |(A_0\setminus A_1) \cup\{v^*_i\}|$.
For any $v\in V(G(i))$, 
we have $d_{G_0}(v) - d_{G(i)}(v) \leq |A_{1}| + \Delta t_{\Lambda}n_{\Lambda} \leq 2\gamma n$ 
and hence $G(i)$ is $(\gamma^{2/5},p)$-quasi-random by (H3)$_0$ and Proposition~\ref{prop: quasi-random subgraph}. 
Thus we can apply Lemma~\ref{blow up with pre-embedding} to obtain $\phi_{T_{*,i}}$ such that 
\begin{align}\label{eq:treeslast}
\phi_{T_{*,i}}(z^*_{i}) = v^*_i \text{ and }\phi_{T_{*,i}}(T_{*,i}) \subseteq (A_0\setminus A_1)\cup \{v^*_i\}. 
\end{align}
Repeating this process for every $i\in [t_{\Lambda}n_{\Lambda}]$
gives rise to functions $\{\phi_{T_{*,i}}\}_{i\in [t_{\Lambda}n_{\Lambda}]}$. 
We define $\phi_{-1}:=\phi_{T_{*,1}}\cup \ldots \cup \phi_{T_{*,t_{\Lambda}n_{\Lambda}}}$ and
\begin{align}\label{eq:treeslast2}
	G_{*,0}:=G_0 - E(\phi_{-1}(\cT_*)).
\end{align}
Then $G_{*,0}$ is $(\gamma^{1/3},p)$-quasi-random as
\begin{align}\label{eq:degG*0}
	 d_{G_0}(v)- d_{G_{*,0}}(v) \leq \Delta t_{\Lambda}n_{\Lambda} \leq n^{3/4}
\end{align}
for all $v\in V(G_{*,0})$. 
Similarly, $G_{*,0}[A_0\setminus A_1]$ is also $(\gamma^{1/3},p)$-quasi-random. 

For each tree $T\in \cT\sm\cT^*$, 
we select an edge $e_T = r^{1}_{T} r^{2}_{T}$ such that 
$T-e_T$ consists of two components of size at least $\delta n/(2\Delta)\geq n/\Delta^2$ (it is easy to see that such an edge exists).
We also select $|\cT\sm\cT^*|$ arbitrary edges $\{e'_T : T\in \cT\sm \cT^*\}$ in $G_{*,0}[R_0]$ 
such that $\Delta( \{e'_1,\dots, e'_{|\cT\sm\cT^*|}\})\leq \epsilon^{-2}$. 
By \eqref{eq:degG*0} and (H3)$_{0}$ we can greedily select these edges.%
\COMMENT{Use the edges incident to a vertex $u$ as long as $u$ is incident to at most $\epsilon^{-2}$ selected edges.
Call the others 'bad'.
Note that $|\cT\sm \cT^*| \leq |\cT| \leq n/\delta$.
So $|bad|\leq (n/\delta)/\epsilon^{-2}\leq \epsilon^{3/2}n=\epsilon^{1/2}|R_0|$. 

So since $\delta(G_{0,*}[R_0])\geq p \epsilon n/2 - n^{3/4}\geq p|R_0|/3$ by (H3)$_0$ and \eqref{eq:degG*0},
we can find the required edges.
}
Next we define $\phi'_{-1}$ by $\phi'_{-1}(e_T) :=e'_T$ for all $T\in \cT\setminus \cT^*$.
We also extend $\phi_{-1}$ be setting $\phi_{-1}(e_T):=\phi_{-1}'(e_T)$.
Let $F^0_T:= T-e_T$.
We view the endvertices of $e_T$ as the roots of $F_T^0$.
Define 
\begin{align}\label{eq:G0}
	\cF^0:=\{F^0_T: T\in \cT\setminus \cT^*\}, \quad
	G'_0:= G_{*,0} - \{e'_T: T\in \cT\sm \cT^*\}\enspace \text{ and } \enspace
	G_0'':=G_0-E(G_0'). 
\end{align}
We point out that $G_0'$ is $(\beta^2,p)$-quasi-random.

We say a function $\psi$ \emph{packs} a family $\cH$ of graphs \emph{perfectly} into a graph $H$
if $\psi$ packs $\cH$ and every edge of $H$ is covered.
Thus
\begin{align}\label{eq:phi-1}
	\phi_{-1} \text{ packs } \cT_*\cup \{e_T:T\in \cT\sm \cT^*\} \text{ perfectly into } G_0''.
\end{align}
Thus $G_0''$ consists of the edges of $G_0$ that have been covered in this step and $G_0'$ denotes the (uncovered) leftover of $G_0$.
Moreover, assumption (iii) of Theorem~\ref{thm: main result step1} implies that 
\begin{eqnarray*}
	e(\cF^0)
	&= &e(\cT)-e(\cT^*)- |\cT\sm \cT^*|\\
	&= & (e(G_0)+e(H^0_1) + e(H^0_2)+ e(G[A_1]))- (e(\cT_*)+ t_{\Lambda}n_{\Lambda})- |\cT\sm \cT^*|\\
	&\stackrel{\eqref{eq:G0},\eqref{eq:phi-1}}{=} &e(G'_0)+ e(H^0_1) + e(H^0_2) + e(G[A_1])- t_{\Lambda}n_{\Lambda}.
\end{eqnarray*}
This completes the Preparation step.

\medskip

We proceed with an inductive approach. 
We will define a packing of $\cF^0$ step by step.
In the $t$th step we will be given a collection $\cF^t$ consisting of those subforests of the forests $\cF^0$ 
which (apart from their roots) are not embedded yet.
For each $F\in \cF^t$ we will embed either a part of $F$ or the entire forest $F$
in such a way that all the edges of $G-E(G_0'')$ incident to a vertex in $A_t\sm A_{t+1}$ are covered.
The key tool for this is Lemma~\ref{lem: iteration}.
In the $(t+1)$th step
we will then continue with the collection $\cF^{t+1}$ consisting of all those subforests which are still not embedded yet.
We will denote by $\phi_t$ the packing of the subforests defined by the end of Step~$t$.
The roots of each $F\in\cF^{t+1}$ will already be embedded into $R_{t+1}\sub A_{t+1}\sm A_{t+2}$ by $\phi_{t}$
(but as described above, no other vertex of $F$ will be embedded yet).

More precisely,
for each $0\leq t\leq \Lambda$,
we will define a packing $\phi_{t-1}$
and a collection $\cF^t$ of rooted forests  
so that the following hold:
\begin{enumerate}[label=(E\arabic*)$_t$]
	\item\label{item:E1} There is a $(\beta^2,p)$-quasi-random spanning subgraph $G'_t$ of $G_t$ 
	and if $t\leq \Lambda-2$, then $G'_t$
	satisfies $d_{G'_t,R_{t+1}}(v) \geq p|R_{t+1}|/5$ for any vertex $v\in A_{t}\setminus A_{t+1}$.
	\item\label{item:E2} If $t=\Lambda$, then $\cF^{t}=\es$,
	and if $t\leq \Lambda-1$, 
	then $\cF^{t}$ consists of rooted subforests of the forests in $\cF^0$ and $|\cF^t|\geq (1/2+1/\Delta)n_t$.
	Moreover, each $F\in \cF^t$ satisfies $|F|\leq (1-1/\Delta)n_t$ and consists of two components $(K_F^1,r_F^1)$, $(K_F^2,r_F^2)$
	with $|K_F^1|,|K_F^2|\geq p\Delta^{-3}n_t/5$.	
	\item\label{item:E3} For each $F\in \cF^t$, no vertex of $F-\{r_F^1,r_F^2\}$ is embedded by $\phi_{t-1}$
	and $\phi_{t-1}(r_F^1),\phi_{t-1}(r_F^2)\in R_t$.
	Let $\phi_{t-1}':\{r_F^c:c\in [2],F\in \cF^t\}\to R_t$ be defined by $\phi_{t-1}'(r_F^c):=\phi_{t-1}(r_F^c)$.
	Then $|\phi_{t-1}'^{-1}(v)|\leq \epsilon^{-2}$ for every $v\in R_{t}$.	
	\item\label{item:E4} For each $F\in \cF^t$, 
	let $F^\ori$ be the unique forest in $\cF^0$ with $F\sub F^\ori$.
	Let $\cF^{0,t}$ be the collection of all those $F'\in \cF^0$ for which $\cF^t$
	does not contain a subforest of $F'$.
	If $t\geq 1$,
	then $\phi_{t-1}$ is consistent with $\phi_{t-2}$
	and packs $\cF^{0,t}\cup \{F^\ori-(V(F)\sm \{r_F^1,r_F^2\}):F\in \cF^t\}$ perfectly into
	\begin{align*}
		(G_t-E(G'_t))\cup \bigcup_{i=1}^{t-1} G_i \cup \bigcup_{i=0}^{t-1} (H_1^i\cup H_2^i)\cup G_0'
		= (G_t-E(G_t'))\cup (G-E(G[A_t]\cup G_0'')).
	\end{align*}
	Moreover, for each $F\in \cF^t$ we have $A_t\cap \phi_{t-1}(F^{\ori}-V(F))=\es$.
	\item\label{item:E5} If $t\leq \Lambda-1$, then 
	$e(\cF^t)= e(G'_t)+ e(H^t_1) + e(H^t_2) + e(G[A_{t+1}]) - t_{\Lambda}n_{\Lambda}$.
\end{enumerate}
Note that (E1)$_0$--(E5)$_0$ hold. 
(To check (E2)$_0$,
note that $|\cT^*|=t_{\Lambda}n_{\Lambda}\leq n^{3/4}$ and recall that $1/\Delta\leq \delta/2$.)
Hence suppose that $0\leq t\leq \Lambda-1$ and that we have defined $\phi_{t-1}$ and $\cF^t$ satisfying \ref{item:E1}--\ref{item:E5}.
Our aim of Step~$t$ is to show that there is a function $\phi_t$ consistent with $\phi_{t-1}$  and for all $F\in \cF^t$
there is a subforest $F^*\sub F$ so that  
(E1)$_{t+1}$--(E5)$_{t+1}$ hold,
where $\cF^{t+1}:=\{F-V(F^*):F\in \cF^t\}$.
In particular,
for each $F\in \cF^t$,
the function $\phi_t$ will embed $F^*$ as well as the two roots of $F-V(F^*)$ into $G_t'\cup H_1^t\cup H_2^t\cup G_{t+1}$,
but will embed no other vertices of $F-V(F^*)$.
For some of the $F\in \cF^t$ we will have $F=F^*$.
Those forests $F$ will be embedded in Step $t.2$,
while the part $F^*$ of all those $F$ with $F\neq F^*$ will be embedded in Step $t.1$.

Note that $\cF^{0,t}$ consists of those forests in $\cF^0$ which have been completely embedded prior to Step~$t$.1.
Also we remark that $G_t'$ is the subgraph of $G_t$ which has not been covered prior to Step~$t$.1.
The graphs $G_t'$, $H_1^t$, and $H_2^t$ will be covered entirely in Step~$t$.

\bigskip

\noindent {\bf Step $t$.1.}
If $t = \Lambda -1$, let $\cF:= \cF^t$, $\cF^{\Lambda}:=\es$ and $G^{\Lambda-1}:=G_{\Lambda-1}'$.
Then \eqref{eq:edgesGlast} and (E5)$_{\Lambda-1}$ imply
\begin{align}\label{eq:edgeslast}
e(\cF)
=e(\cF^{\Lambda-1})
=e(G^{\Lambda-1})+e(H_1^{\Lambda-1})+e(H_2^{\Lambda-1})+3\gamma_*^2 n_{\Lambda-1}^2 \pm n_{\Lambda}.
\end{align}
We proceed to Step $t$.2.

From now on we assume that 
$t< \Lambda-1$.
In what follows we prepare the set $\cF^{t+1}$ for the next iteration step.
For this, we choose an arbitrary sub-collection $\cF'\subseteq \cF^t$ of $\Delta^2 \gamma n_{t}=\Delta^2n_{t+1}$ forests and 
let $\cF := \cF^t \setminus \cF'$. 
Then 
\begin{align}\label{eq:sizeF1}
	|\cF|\geq (1/2+ 1/\Delta)n_t - \Delta^2 \gamma n_{t} \geq (1/2 + 1/(2\Delta)) n_t. 
\end{align}

The forests in $\cF$ will be embedded in {Step $t$.2.}~via Lemma~\ref{lem: iteration},
and $\cF^{t+1}$ will consist of subforests of the forests in $\cF'$.
The set $\cF^*$ of remaining subforests of the forests in $\cF'$ will be embedded greedily in the current Step $t$.1.
In particular, we will have $\cF^{0,t+1}=\cF^{0,t}\cup \cF$,
where $\cF^{0,t}$ is as defined in \ref{item:E4}.
Since $t<\Lambda-1$, 
\ref{item:O5} and \eqref{eq:edgesGlast} give us
\begin{align*}
 e(G[A_{t+1}]) - t_{\Lambda} n_{\Lambda} = (1\pm \gamma^{1/2}){pn_{t+1}^2}/{2} - t_{\Lambda} n_{\Lambda}= (1\pm \gamma^{1/3}) p n_{t+1}^2/2.
\end{align*}
Thus 
$$\beta':=\frac{ e(G[A_{t+1}]) - t_{\Lambda}n_{\Lambda} - 3\gamma_*^2 n_t^2}{n_t|\cF'|} 
= (1\pm \gamma^{1/4})\frac{p  n_{t+1}}{2\Delta^2n_t}
=(1\pm \gamma^{1/4})\frac{\gamma p}{2\Delta^{2}}.
$$
By Proposition~\ref{prop: taking subtree collection} with
and $\beta'$, $p\Delta^{-3}/5$, $n_t$ playing the roles of $\beta$, $\alpha$, $n$,\COMMENT{
So the previous display shows $\beta'\ll \alpha$.}
we can find a subforest $\tilde{F}$ of each $F\in \cF'$ such that 
\begin{enumerate}[label=(P\arabic*)]
	\item\label{item:P1'} both $\tilde{F}$ and $F^*:=F- V(\tilde{F})$ consist of two components,
	\item\label{item:P2'} each component of $\tilde{F}$ has size between $\Delta^{-3} p n_{t+1}/5$ and $\Delta^{-1}pn_{t+1}$, and
	\item\label{item:P3'}  $\tilde{F}$ does not contain $r^{1}_{F}$ or $r^{2}_{F}$.
	\item\label{item:P4'} Moreover, let $\cF^{t+1}:=\{ \tilde{F} : F\in \cF'\}$. Then
$e(\cF^{t+1})=\sum_{F \in \cF'} e(\tilde{F}) = e(G[A_{t+1}]) - t_{\Lambda}n_{\Lambda} - 3\gamma_* n_t^2 \pm n_t$.
\end{enumerate}
Note that \ref{item:E2} implies that for all $F\in \cF'$, we have
\begin{align}\label{eq:sizeF*}
	|F^*|\leq |F| \leq (1- 1/\Delta)n_t.
\end{align}
For each $F\in \cF'$ and $c\in [2]$, 
let $r^{c}_{\tilde{F}}$  be the unique vertex of $\tilde{F}$ in the component of $F$ containing $r_F^c$  
which is adjacent to a vertex in $V(F^{*})$ in $F$.
Note that \ref{item:P2'} implies that $\cF^{t+1}$ satisfies (E2)$_{t+1}$ with $r_{\tilde{F}}^c$ playing the role of $r_{F}^c$.
For every $c\in [2]$, let $x^{c}_{F}$ be the unique neighbour of $r_{\tilde{F}}^c$ in $V(F^{*})$ in $F$. 

Let $\{F_1,\dots, F_{m}\}:= \cF'$
and $\cF^*:=\{F^*:F\in \cF'\}=\{F_1^*,\ldots,F_m^*\}$.
Note that 
\begin{align}\label{eq:sizeF'}
	e(\cF')=e(\cF^*)+e(\cF^{t+1})+2m.
\end{align}
Recall from \ref{item:E2} that $r_{F_i}^1$ and $r_{F_i}^2$ are the roots of $F_i$ and thus,
by \ref{item:P3'}, the roots of $F_i^*$.
In what follows,
for all $i\in [m]$, we construct embeddings $\phi_{F_i^*}$  satisfying
the following:
\begin{enumerate}[label=(Q\arabic*)$_{i}$]
	\item\label{item:Q1i} $\phi_{F_1^*}\cup \dots \cup \phi_{F_{i}^*}$ packs $\{F_1^*,\dots, F_{i}^*\}$ into $G'_t[A_t\setminus A_{t+1}]$ 
	such that $\phi_{F_j^*}$ is consistent with $\phi_{t-1}'$ for all $j\in [i]$, and
	\item\label{item:Q2i} $p(v,i) \leq \epsilon^{-2}$ for any $v\in A_t\sm A_{t+1}$, where $p(v,i):=|\{ j\in[ i]: v\in \phi_{F_j^*}(\{x^{1}_{F_j^*},x^{2}_{F_j^*}\})\}|.$ 
	\item\label{item:Q3i} $\phi_{F_i^*}(F_i^*)\cap R_t = \{\phi_{t-1}'(\{r_{F_i}^1,r_{F_i}^2\})\}$.
\end{enumerate}
Assume for all $j<i$ we have constructed $\phi_j$ satisfying (Q1)$_j$--(Q3)$_j$.
Now we construct $\phi_{F_i}$.
Recall that 
\begin{align}\label{eq:sizeF'2}
	m=|\cF'|=\Delta^2 \gamma n_t=\Delta^2 n_{t+1}.
\end{align}
Let 
$$B:=\{ v \in A_t\sm (A_{t+1}\cup R_t) : p(v,i-1) > \epsilon^{-2}-1\}.$$ 
Hence $|B|\leq |\cF'|/\epsilon^{-1} =\epsilon\Delta^2 \gamma n_t \leq \epsilon n_t$. 
Define
$$G'_t(i):= G'_t[A_t\setminus (A_{t+1}\cup B \cup (R_t\sm \{\phi_{t-1}'(\{r_{F_i}^1,r_{F_i}^2\}) \}))]  - \bigcup_{j=1}^{i-1} E( \phi_{F_j^*}(F_j^*)).$$ 
Note that 
\begin{align}\label{eq:maxdeg}
	\Delta( \bigcup_{j=1}^{i-1} E( \phi_{F_j^*}(F_j^*))) 
	\leq \Delta m 
	= \Delta^3 \gamma n_t.
\end{align}
Thus $G'_t(i)$ is $(\beta^{3/2},p)$-quasi-random by \ref{item:E1} and Proposition~\ref{prop: quasi-random subgraph}.
In addition, 
\ref{item:O2}, \ref{item:O4} and \eqref{eq:sizeF*} imply 
$|G'_t(i)|\geq (1-\gamma-2\epsilon) n_{t} \geq |F_i^*|$.
So we can apply Lemma~\ref{blow up with pre-embedding} (with $\{r_{F_i}^1,r_{F_i}^2\}$ playing the role of $I$) to obtain a function $\phi_{F_i^*}$ packing $F_i^*$ into $G'_t(i)$ 
which is consistent with $\phi'_{t-1}$.
Thus \ref{item:Q1i} and \ref{item:Q3i} hold,
and \ref{item:Q2i} follows from the definition of $B$ and \ref{item:E3}.

By repeating this procedure for every $i\in [m]$, 
we obtain a collection of functions $\{\phi_{F^*_i}\}_{i\in [m]}$ satisfying \ref{item:Q1i}--\ref{item:Q3i} for every $i\in [m]$.
We let $\phi^*_t:=\phi_{F_1^*}\cup \dots \cup \phi_{F_{m}^*}$.
Let $G^t_*:= G'_t - E( \phi^*_t(\cF^*))$
and observe that $G^t_*$ is $(\beta^{3/2},p)$-quasi-random
by 
\eqref{eq:maxdeg}, \ref{item:E1} and Proposition~\ref{prop: quasi-random subgraph}.

Now for each $F\in \cF'$, we will embed the roots $r_{\tilde{F}}^1$ and $r_{\tilde{F}}^2$ of $\tilde{F}$ so that their embedding satisfies (E3)$_{t+1}$.
Let $u_j:= \phi^*_t(x^1_{F_j})$ and $u_{m+j}:= \phi^*_t(x^{2}_{F_j})$ for each $j\in [m]$. 
Note that by (Q2)$_m$ no vertex occurs more than $\epsilon^{-2}$ times in the list $u_1,\dots, u_{2m}$.
Since $\phi_{F^*_j}(F^*_j) \cap R_{t+1} \subseteq \phi_{F^*_j}(F^*_j) \cap A_{t+1} = \emptyset$ for all $j\in [m]$,  
\ref{item:E1} implies that 
for each vertex $u_j$ we have 
$d_{G^t_{*},R_{t+1}}(u_j)
= d_{G_t',R_{t+1}}(u_j)
\geq p\epsilon n_{t+1}/5 
\geq 3\epsilon^{-2} + 2m/\epsilon^{-2}  +\epsilon^{-2}$.

Let $H^{\rm ind}$ be the graph on $[2m]$ such that $ij\in E(H^{\rm ind})$ if $i=j+m$.
Hence we can apply Proposition~\ref{prop: sparse edge embedding} with the following parameters and graphs. \newline
  
\noindent
{
\begin{tabular}{c|c|c|c|c|c|c|c|c}
object/parameter & $G^t_{*}$ & $H^{\rm ind}$ & $ \epsilon^{-2} $ & $R_{t+1}$ &  $2m$ & $\epsilon^{-2} $  & $\emptyset$ & $u_{j}$
\\ \hline
playing the role of & $G$ & $H$ & $\Delta$ & $A$ &  $m$ & $s$  &$W_j$ & $u_j$ 
\end{tabular}
}\newline \vspace{0.2cm}

\noindent
From Proposition~\ref{prop: sparse edge embedding}, 
we obtain $v_1,\ldots,v_{2m}\in R_{t+1}$ such that $E:=\{u_iv_i:i\in[2m]\}$ is a collection of distinct edges in $G^t_{*}$,
and we have $v_i\neq v_{i+m}$ for every $i\in [m]$.
Moreover, for any $v\in R_{t+1}$, we have $|\{i\in[2m]: v=v_i\}|\leq \epsilon^{-2}$. 
Thus $\Delta(E)\leq \epsilon^{-2}$. 
We extend $\phi_t^*$ and define a new function $\phi_t'$ by setting for all $i \in [m]$
\begin{align}\label{eq:defphit*}
	\phi_t^*(r^{1}_{\tilde{F_i}})=\phi'_t(r^{1}_{\tilde{F_i}}) := v_i \text{ and }
	\phi_t^*(r^{2}_{\tilde{F_i}})=\phi'_t(r^{2}_{\tilde{F_i}}) := v_{m+i}.
\end{align}
Thus $|\phi'^{-1}_t(v)|\leq \epsilon^{-2}$ for every $v\in R_{t+1}$. 
As for each $i\in [m]$ we will not embed any vertex of $\tilde{F_i}-\{r^{1}_{\tilde{F_i}},r^{2}_{\tilde{F_i}}\}$ in Step~$t$,
this will imply that (E3)$_{t+1}$ holds.
Moreover,
note that for every $i\in [m]$
\begin{align}\label{eq:At+1F_i}
	A_{t+1}\cap \phi^*_t(F_i-V(\tilde{F_i}))
	\stackrel{\text{\ref{item:P1'}}}{=}A_{t+1}\cap \phi^*_t(F_i^*)
	\stackrel{\text{\ref{item:Q1i}}}{=}\es.
\end{align}
Let 
$$G^{t} := G^t_{*}-E.$$ 
Since $G^t_*$ is $(\beta^{3/2},p)$-quasi-random and $\Delta(E)\leq \epsilon^{-2}$, 
\begin{align}\label{eq:Gtqr}
	G^{t} \text{ is } (\beta,p)\text{-quasi-random.}
\end{align}
Note that $G^t$ is the current leftover of the graph $G_t'$ given in \ref{item:E1} and $e(G_t')=e(G^t)+e(\cF^*)+2m$.
Hence
\begin{eqnarray}
e(\cF) &=& e(\cF^{t})  - e(\cF')  \nonumber \\
&\stackrel{\text{\ref{item:E5}}}{=}&  e(G'_t)+ e(H^t_1) + e(H^t_2)+ e(G[A_{t+1}]) - t_{\Lambda}n_{\Lambda} - e(\cF') \nonumber \\
&=&  e(G^t)+ e(\cF^*) + 2m + e(H^t_1) + e(H^t_2)+ e(G[A_{t+1}]) - t_{\Lambda}n_{\Lambda} - e(\cF') \nonumber \\
&\stackrel{\eqref{eq:sizeF'},\text{\ref{item:P4'}}}{=}& 
e(G^t)+ e(H^t_1) + e(H^t_2)+  3\gamma_*n_t^2 \pm n_t.\label{eq:sizeF2}
\end{eqnarray}
We proceed to Step $t$.2., where we pack $\cF$ into $G^t$ so that the packing is consistent with $\phi'_{t-1}$.\newline

\noindent {\bf Step $t$.2.}
Let $\hat{G}_{t+1}$ be a $(\gamma^{1/3},p/2)$-quasi-random subgraph of the graph $G_{t+1}$ (described in (H3)$_{t+1}$)
such that,
if $t\leq \Lambda-3$, then for any vertex $v \in A_{t+1}\setminus A_{t+2}$, we have 
\begin{align}\label{eq:deg Rt}
	d_{G_{t+1}-E(\hat{G}_{t+1}),R_{t+2}}(v)\geq p|R_{t+2}|/5. 
\end{align}
Such a subgraph of $G_{t+1}$ exists as a random subgraph chosen by including  every edge of $G_{t+1}$ independently with probability $1/2$ has these properties with probability at least $1/2$, by (H3)$_{t+1}$.

Now we apply Lemma~\ref{lem: iteration} with the following parameters and graphs. \newline
  
\noindent
{
\begin{tabular}{c|c|c|c|c|c|c|c|c|c|c}
object/parameter & $G^t$ & $\hat{G}_{t+1}$ & $H^t_1$ & $H^t_2$ & $A_{t+1}$ & $R_t$ & $\cF$ & $n_t$ & $\epsilon$ & $\delta_1$
\\ \hline
playing the role of & $G$ & $G'$ & $H_1$ & $H_2$ & $A$ & $R$ & $\cF$ & $n$ & $\epsilon$ & $\delta_1$
\\ \hline 
\hline
object/parameter & $\delta_2$ & $\gamma_*$ & $\gamma$ & $\beta$ & $p$ & $ 10\Delta$ & $\Delta$ & $p/2$ & $1/(2\Delta)$ & $\phi'_{t-1}$ 
\\ \hline
playing the role of & $\delta_2$ & $\gamma_*$ & $\gamma$ & $\beta$ & $\alpha$ & $D$ & $\Delta$ & $p$ & $d$ & $\phi'$ 
\end{tabular}
}\newline \vspace{0.2cm}

Observe that 
\ref{item:G1} holds by \ref{item:O2}, \ref{item:O4} and (H4)$_t$,
\ref{item:G2} holds by \eqref{eq:Gtqr} and our choice of $\hat{G}_{t+1}$,
\ref{item:G3} holds by (H1)$_t$, and
\ref{item:G4} holds by (H2)$_t$.
Furthermore, 
\ref{item:T1} holds by \eqref{eq:sizeF1} if $t<\Lambda-1$, or (E2)$_{\Lambda-1}$ if $t=\Lambda-1$,
\ref{item:T2} holds by \ref{item:E2},
\ref{item:T3} holds by \eqref{eq:sizeF2} or \eqref{eq:edgeslast}, and
\ref{item:T4} holds by \ref{item:E3}.

We obtain a function $\hat{\phi}_{t}$ packing $\cF$ into $G^t\cup \hat{G}_{t+1}\cup H^t_1\cup H^t_2$ consistent with $\phi'_{t-1}$ 
such that 
\begin{enumerate}[label=$(\Phi\arabic*)_t$]
	\item\label{item:Phi1t} $E(G^t)\cup E(H^t_1)\cup E(H^t_2) \subseteq E(\hat{\phi}_{t}(\cF))$, and
	\item\label{item:Phi2t} for all $v\in A_{t+1}$, we have $d_{\hat{G}_{t+1}\cap \hat{\phi}_t(\cF)}(v)\leq \gamma_*^{1/2}|A_{t+1}|$.
\end{enumerate}

Let $\phi_t:=\phi_t^*\cup \hat{\phi}_t\cup \phi_{t-1}$.
By (H3)$_{t+1}$ and \ref{item:Phi2t}, 
we conclude that $G'_{t+1}:= G_{t+1} -  E(\hat{\phi}_{t}(\cF))$ is $(\beta^2,p)$-quasi-random.
Thus (E1)$_{t+1}$ holds, by \eqref{eq:deg Rt}.
We have already observed that (E2)$_{t+1}$ and (E3)$_{t+1}$ hold.
Property (E4)$_{t+1}$ follows from \eqref{eq:At+1F_i}, \ref{item:Phi1t} and \ref{item:E4}.
In particular,
note that in Step~$t$ we have now covered the (previous) leftover $G_t'$ of $G_t$ which was not covered prior to Step~$t$.
The graphs $H_1^t$ and $H_2^t$ are covered entirely in Step~$t$ and $G_{t+1}'$ is the (new) leftover of $G_{t+1}$,
which will be covered in Step~$t+1$.
To check (E5)$_{t+1}$,
note that (E4)$_{t+1}$ implies that 
\begin{align}\label{eq:diffFt}
	e(\cF^t)-e(\cF^{t+1})=e(G_t')+e(H_1^t)+e(H_2^t)+e(G_{t+1})-e(G_{t+1}').
\end{align}
Thus if $t<\Lambda-1$, then
\begin{eqnarray*}
e(\cF^{t+1})
&=& e(\cF^{t})-(e(\cF^{t})-e(\cF^{t+1}))\\
&\stackrel{\text{\ref{item:E5},\eqref{eq:diffFt}}}{=}& e(G[A_{t+1}]) -e(G_{t+1})+e(G_{t+1}')- t_{\Lambda}n_{\Lambda}
\\
&=&e(G_{t+1}')+ e(H_1^{t+1})+e(H_2^{t+1})+e(G[A_{t+2}])-t_{\Lambda}n_{\Lambda}.
\end{eqnarray*}
This verifies (E5)$_{t+1}$.

If $t=\Lambda-1$, then we proceed to the Final step. Otherwise we proceed to Step $(t+1)$.1.
\newline

\noindent {\bf Final step}.
Recall that $G_{\Lambda}=G[A_\Lambda]$.
Since $\cF^{\Lambda}=\es$,
properties (E4)$_1$--(E4)$_\Lambda$ imply that $\phi_{\Lambda-1}$ is consistent with $\phi_{-1}$ and packs $\cF^{0,\Lambda}=\cF^0$ perfectly into
\begin{align*}
	(G_\Lambda-E(G_\Lambda'))\cup \bigcup_{i=1}^{\Lambda-1}G_i \cup \bigcup_{i=0}^{\Lambda-1}(H_1^i\cup H_2^i)\cup G_0'
	=G-(E(G_\Lambda')\cup E(G_0'')).
\end{align*}
Together with \eqref{eq:phi-1}, this implies that $\phi_{\Lambda-1}\cup \phi_{-1}$ packs
$\cT_{*}\cup (\cT\sm \cT^*)$ perfectly into $G-E(G'_\Lambda)$.
In particular,
the current set of leftover edges is given by $G_\Lambda'$.
Recall that the trees in $\cT_*$ were obtained from those in $\cT^*=\{T_1^*,\ldots,T^*_{t_\Lambda n_\Lambda}\}$ by deleting the edges $z^*_i\ell_i^*$ for all $i\in [t_\Lambda n_\Lambda]$.
In particular,
\begin{align}\label{eq:edgeslast1}
	e(G_{\Lambda}')
	=e(G)-e(\cT_*\cup (\cT\sm \cT^*))
	=e(\cT)-(e(\cT^*)-t_\Lambda n_\Lambda)- e(\cT\sm \cT^*)
	=t_\Lambda n_\Lambda.
\end{align}
We will now extend our current packing to a packing of $\cT$ into $G$ by embedding all these edges $z^*_i\ell_i^*$.

Since $|G_\Lambda'|=|A_{\Lambda}|=n_\Lambda$, by \ref{item:O2},
\eqref{eq:edgeslast1} implies that $\overline{d}(G'_{\Lambda}) = 2t_{\Lambda}$. 
Recall that $G'_{\Lambda}$ is $(\beta^2,p)$-quasi-random, by (E1)$_\Lambda$.
Thus by Lemma~\ref{lem: orientation}, 
$G'_{\Lambda}$ has an orientation $D$ such that every vertex in $V(G_\Lambda')=A_{\Lambda}$ has exactly $t_{\Lambda}$ out-neighbours.
For a vertex $v\in V(G_\Lambda')$, we denote by $N_D^+(v)$ the set of out-neighbours of $v$.
By \eqref{eq:treeslast} and our choice of the vertices $v_i^*$,
for each $v\in A_{\Lambda}$, we have $|\{i\in [t_\Lambda n_\Lambda]: \phi_{-1}(z^*_i) = v\}|=t_{\Lambda}=|N^+_{D}(v)|$.
Thus there is a bijection  $\tau_v:\{\ell_i^*: \phi_{-1}(z^*_i) = v\}\to N^+_{D}(v)$.
Let 
$$\phi_{-1}(\ell^*_i):= \tau_{\phi_{-1}(z^*_i)}( \ell^*_i) .$$
Then $E(G_{\Lambda}') = \{ \phi_{-1}(z^*_i) \phi_{-1}(\ell^*_i) : i\in [t_{\Lambda}n_{\Lambda}]\}$ and 
this completes the partial embedding of $T^*$ to a complete one for every $T^*\in\cT^*$ 
(indeed, recall that $T_{*,i}=T_i^*-\ell^*_i$ and $\phi_{-1}(T_{*,i})\cap A_{\Lambda}=\{\phi_{-1}(z^*_i)\}$ by \eqref{eq:treeslast},
thus $\tau_{\phi_{-1}(z^*_i)}(\ell^*_i)\notin \phi_{-1}(T_{*,i})$).

Now, by construction, $\phi_{\Lambda-1}\cup \phi_{-1}$ packs $\cT$ into $G$.
\end{proof}

\subsection{Proof of Theorem~\ref{thm: main result} and its consequences}\label{sec:9.2}

To prove Theorem~\ref{thm: main result} we need Theorem~1.2 from \cite{KKOT16}. 
It allows us to pack a suitable collection of bounded degree graphs into a quasi-random graph $G$, and 
guarantees a near-optimal packing, i.e.~almost all edges of $G$ are used by the packing.

\begin{theorem} \label{thm: any graph packing}
Suppose $n,\Delta \in \N$ with $1/n \ll \epsilon \ll p, \alpha, 1/\Delta\leq 1$.
Suppose $\cH$ is a collection of graphs on $n$ vertices with $\Delta(H)\leq\Delta$ for all $H\in \cH$ and
$(1-2\alpha) \binom{n}{2}p \leq e(\cH) \leq (1-\alpha) \binom{n}{2}p$. Suppose that $G$ is an $(\epsilon,p)$-quasi-random graph on $n$ vertices.
Then there exists a function $\phi$ which packs $\cH$ into $G$ such that $\Delta(G - E(\phi(\cH)))\leq4\alpha p n$.
\end{theorem}

\begin{proof}[Proof of Theorem~\ref{thm: main result}]
We choose $\epsilon, \epsilon'$ and $N$ such that
\begin{align*}
	1/N\ll \epsilon \ll \epsilon'\ll \delta,1/\Delta.
\end{align*}
Let $p'$ be a real number such that $e(\cH)= p'\binom{n}{2}$.
Note that by assumptions \ref{item:MT2}--\ref{item:MT4} of Theorem~\ref{thm: main result},
we have $p'\leq p -\delta$.
Let $G'$ be a subgraph of $G$ such that 
\begin{enumerate}[label=(Z\arabic*)]
\item\label{item:Z1} $G - E(G')$ is $(\epsilon^{1/2},p-p'-\epsilon'^2)$-quasi-random, and
\item\label{item:Z2} $G'$ is $(\epsilon^{1/2}, p'+\epsilon'^2)$ quasi-random.
\end{enumerate}
Note that $G'$ exists, as a random subgraph of $G$ obtained by independently including each edge with probability $(p'+\epsilon'^2)/p$ has these properties with probability at least $1/2$, by Lemma~\ref{lem: chernoff}.

Since $\epsilon \ll \epsilon',1/\Delta$,
we can use Theorem~\ref{thm: any graph packing} to find a function $\phi_{\cH}$ packing $\cH$ into $G'$ such that 
\begin{align}\label{eq:split}
	\Delta(G' - E(\phi_{\cH}(\cH)))\leq 4 \cdot \frac{\epsilon'^2}{p'+\epsilon'^2} \cdot (p'+\epsilon'^2) n = 4\epsilon'^2 n.
\end{align}
\COMMENT{Note that $G'$ is $(\epsilon^{1/2},p'+\epsilon'^2)$ quasi-random and $e(\cH)= p'\binom{n}{2} = (1- \epsilon'^2/(p'+\epsilon'^2))(p'+\epsilon'^2)\binom{n}{2}$. 
Thus   $\epsilon'^2/(p'+\epsilon'^2)$ is playing the role of $\alpha$ and 
$p'+\epsilon'^2$ is playing the role of $p$. 
Thus we get 
$4 (\epsilon'^2/(p'+\epsilon'^2)) (p'+\epsilon'^2) n=4\epsilon'^2 n$.}
Define $G_1:= G-E(\phi_{\cH}(\cH))$.
Then
$G_1$ is $(\epsilon',p-p')$-quasi-random  by Proposition~\ref{prop: quasi-random subgraph}, \ref{item:Z1}, 
and \eqref{eq:split}.
Note that $e(G_1)=e(\cT)$.
Now we apply Theorem~\ref{thm: main result step1} to $G_1$ and $\cT$ to complete the proof.
\end{proof}

Next we deduce Corollary~\ref{cor: trees into quasi-random} from Theorem~\ref{thm: main result}.

\begin{proof}[Proof of Corollary~\ref{cor: trees into quasi-random}]
We may assume $\Delta\geq 2$ and $\alpha < 1/3$. 
Let $N\in \N$ and $\epsilon>0$ be the constants given by Theorem~\ref{thm: main result} for the parameters $\Delta$ and $\delta:=p_0\alpha/2$.
We may assume that $\epsilon\ll \alpha,p_0$. 
We add arbitrary trees on at most $(1-\alpha)pn$ vertices with maximum degree at most $\Delta$ to $\cT$
to obtain a set of trees $\cT'$ so that $e(\cT') = e(G)$.
For two trees $T,T' \in \cT'$ of order at most $\delta n$, 
we choose leaves $\ell$ of $T$ and $\ell'$ of $T'$,
delete both $T$ and $T'$ from $\cT'$ and add the tree obtained from $T$ and $T'$ by identifying $\ell$ and $\ell'$. 
Then every tree in $\cT'$ still has maximum degree at most $\Delta$ and its order is at most $\max\{ 2\delta n, (1-\alpha)pn\} \leq (1-\alpha)pn$,
since $\alpha < 1/3$ and $p_0\leq p$. 
By repeating this process, 
we may assume that every tree in $\cT'$ but at most one has at least $\delta n$ vertices and at most $(1-\alpha)pn\leq (1-\delta)n$ vertices. 
Let $\cT''\sub \cT'$ consist of all trees having at least $\delta n$ vertices.
Let $\cH:=\cT'\sm \cT''$.
Then $|\cH|\leq 1$.\COMMENT{Since 
$$|\cT''| \geq \frac{e(G)-e(\cH)}{(1-\alpha)pn} 
\geq \frac{(p-\epsilon)\binom{n}{2}-\delta n}{(1-\alpha)pn} 
\geq \frac{(1-\alpha^2)p}{(1-\alpha)p}\cdot \frac{n}{2}
\geq \left(\frac{1}{2}+\delta\right) n,$$}
Note that $\cT''$ and $\cH$ satisfy conditions (i)--(iv) of Theorem~\ref{thm: main result}.
Thus $\cT''\cup \cH$ packs into $G$, 
and a packing of $\cT'=\cT''\cup \cH$ into $G$ naturally gives a packing of $\cT$ into $G$.
\end{proof}

Finally we deduce Theorem~\ref{thm:glboundeddegree} from Theorem~\ref{thm: main result}.

\begin{proof}[Proof of Theorem~\ref{thm:glboundeddegree}]
Let $N$ and $\epsilon$ be 
such that the statement of Theorem~\ref{thm: main result} holds with 
$\Delta$, $1/10$, $5\epsilon$ playing the roles of $\Delta$, $\delta$, $\epsilon$.
We may assume that $\epsilon\leq 1/100$.
Next we iteratively pack the trees $T_1,\ldots,T_{\epsilon n}$ into $G_0:=K_n$
in such a way that we cover less than $2\epsilon n$ edges incident to each vertex.

Let $t\in \{0,\ldots, \epsilon n-1\}$.
Suppose we have already packed $T_1,\ldots,T_t$ into $G$ via a function $\tau_t$
such that $\delta(G_t)\geq (1-2\epsilon)n$, where
$G_t:=G-\bigcup_{i=1}^{t}E(\tau_t(T_i))$.
Let $X_t\sub V(G)$ be the set of all vertices with degree less than $(1-\epsilon)n$ in $G_t$.
Observe that $|X_t|\leq 2\sum_{i=1}^{\epsilon n}e(T_i)/(\epsilon n)\leq \epsilon n$.
Thus we can choose an arbitrary embedding of $T_{t+1}$ into $G_t- X_t$ (as $\delta(G_t- X_t)\geq |T_{t+1}|$, such an embedding exists).
Then $\delta(G_{t+1})\geq (1-2\epsilon)n$. This shows that we can pack $T_1,\dots, T_{\epsilon n}$ such that $\delta(G_{\epsilon n}) \geq  (1-2\epsilon )n$.

In particular, $G_{\epsilon n}$ is $(5\epsilon,1)$-quasi-random.
Hence we can apply Theorem~\ref{thm: main result} 
with $\cT:=\{T_{n/10},\ldots,T_{9n/10}\}$, $\cH:=\{T_{\epsilon n+1}, \dots, T_{n/10-1}, T_{9n/10+1},\dots, T_{n}\}$, $\delta=1/10$,
and obtain a packing of $\cH\cup \cT$ into $G_{\epsilon n}$, which completes the proof.
\end{proof}

\section{Acknowledgements}

We are grateful to the referee for a careful reading of the manuscript.

\bibliographystyle{amsplain}

\providecommand{\bysame}{\leavevmode\hbox to3em{\hrulefill}\thinspace}
\providecommand{\MR}{\relax\ifhmode\unskip\space\fi MR }
\providecommand{\MRhref}[2]{%
  \href{http://www.ams.org/mathscinet-getitem?mr=#1}{#2}
}
\providecommand{\href}[2]{#2}

\end{document}